\documentclass[11pt,a4paper,reqno]{article}
\usepackage{graphicx} 
\usepackage[english]{babel} 
\usepackage{csquotes} 
\usepackage{filecontents} 
\usepackage[dvipsnames]{xcolor}
\usepackage{color}
\usepackage[colorinlistoftodos]{todonotes}

\usepackage[
  backend=biber,
  hyperref=true,
  backref=true,
  isbn=false,
  doi=true,
  url=false,
  natbib=true,
  eprint=true,
  useprefix=true,
  maxcitenames=99,
  maxbibnames=99,  
  maxalphanames=99, 
  minalphanames=99,
  safeinputenc,
  style=alphabetic,
  citestyle=alphabetic,
  block=space,
  datamodel=ext-eprint,
]{biblatex}
\usepackage[
  hypertexnames = false,
  colorlinks    = true,
  citecolor     = blue,
  linkcolor     = blue,
  urlcolor      = blue,
  linktocpage = true,
  breaklinks
  ]{hyperref}
\renewbibmacro{in:}{} 

\DeclareSourcemap{
  \maps[datatype=bibtex]{
    \map{
      \step[fieldsource=pmid, fieldtarget=pubmed]
    }
  }
}

\makeatletter
\DeclareFieldFormat{arxiv}{%
  arXiv\addcolon\space
  \ifhyperref
    {\href{http://arxiv.org/\abx@arxivpath/#1}{%
       \nolinkurl{#1}%
       \iffieldundef{arxivclass}
         {}
         {\addspace\texttt{\mkbibbrackets{\thefield{arxivclass}}}}}}
    {\nolinkurl{#1}
     \iffieldundef{arxivclass}
       {}
       {\addspace\texttt{\mkbibbrackets{\thefield{arxivclass}}}}}}
\makeatother
\DeclareFieldFormat{pmcid}{%
  PMCID\addcolon\space
  \ifhyperref
    {\href{http://www.ncbi.nlm.nih.gov/pmc/articles/#1}{\nolinkurl{#1}}}
    {\nolinkurl{#1}}}
\DeclareFieldFormat{mrnumber}{%
  MR\addcolon\space
  \ifhyperref
    {\href{http://www.ams.org/mathscinet-getitem?mr=MR#1}{\nolinkurl{#1}}}
    {\nolinkurl{#1}}}
\DeclareFieldFormat{zbl}{%
  Zbl\addcolon\space
  \ifhyperref
    {\href{http://zbmath.org/?q=an:#1}{\nolinkurl{#1}}}
    {\nolinkurl{#1}}}
\DeclareFieldAlias{jstor}{eprint:jstor}
\DeclareFieldAlias{hdl}{eprint:hdl}
\DeclareFieldAlias{pubmed}{eprint:pubmed}
\DeclareFieldAlias{googlebooks}{eprint:googlebooks}

\renewbibmacro*{eprint}{%
  \printfield{arxiv}%
  \newunit\newblock
  \printfield{jstor}%
  \newunit\newblock
  \printfield{mrnumber}%
  \newunit\newblock
  \printfield{zbl}%
  \newunit\newblock
  \printfield{hdl}%
  \newunit\newblock
  \printfield{pubmed}%
  \newunit\newblock
  \printfield{pmcid}%
  \newunit\newblock
  \printfield{googlebooks}%
  \newunit\newblock
  \iffieldundef{eprinttype}
    {\printfield{eprint}}
    {\printfield[eprint:\strfield{eprinttype}]{eprint}}}

\usepackage[margin=1.8cm,includehead]{geometry}

\usepackage{amsmath,stmaryrd, upgreek}
\usepackage{amsthm,amssymb}
\usepackage{latexsym}
\usepackage{amscd}
\usepackage{mathrsfs}
\usepackage{url}
\usepackage{mathtools}
\usepackage{cleveref}
\usepackage{doi}

\usepackage[T1]{fontenc}
\usepackage{tikz-cd}
\usepackage{titlesec}

\usepackage[titletoc,toc,title]{appendix}

\makeatletter 
\renewcommand\tableofcontents{%
    \@starttoc{toc}%
}
\makeatother

\usepackage{enumerate}
\usepackage{slashed}
\graphicspath{}
\usepackage{fancyhdr} 

\usepackage{changepage}

\newcommand\shorttitle{Harmonic flow of $\S7$-structures} 
\newcommand\authors{Dwivedi--Loubeau--S{\'{a}} Earp} 

\newcounter{commentCounter}
\titleformat{\section}    
       {\normalfont\large\bfseries\center}{\thesection.}{1em}{}
       
\makeatletter
\newcommand*{\rom}[1]{\expandafter\@slowromancap\romannumeral #1@}
\makeatother

\newtheorem{theorem}{Theorem}[section]
\newtheorem{corollary}[theorem]{Corollary}
\newtheorem{lemma}[theorem]{Lemma}
\newtheorem{proposition}[theorem]{Proposition}

\theoremstyle{definition}
\newtheorem{definition}[theorem]{Definition}
\newtheorem{remark}[theorem]{Remark}

\numberwithin{equation}{section}
\newtheorem{thmx}{Theorem}

\def\bR{\mathbb R}

\def\bO{\mathbb O}
\def\bP{\mathbb P}

\def\bRP{\mathbb {RP}}

\newcommand{\rT}{{\rm T}}

\def\wtd{\widetilde}

\def\pt{\partial}
\def\del{\nabla}
\def\G2{\mathrm{G}_2}
\def\g2{\varphi}
\def\S7{\mathrm{Spin}(7)}
\def\s7{\Phi}
\def\ddt{\frac{d}{dt}}

\def\fg{\mathfrak{g}}
\def\fh{\mathfrak{h}}
\def\fm{\mathfrak{m}}
\def\ufm{\underline{\mathfrak{m}}}
\def\fgl{\mathfrak{gl}}
\def\fso{\mathfrak{so}}

\def\cA{\mathcal{A}}

\def\cF{\mathcal{F}}

\def\cH{\mathcal{H}}
\def\cI{\mathcal{I}}

\def\cL{\mathcal{L}}

\def\cT{\mathcal{T}}

\def\cV{\mathcal{V}}

\newcommand{\sX}{\mathscr{X}}

\def\Spin7{\mathrm{Spin(7)}}

\def\Riem{\mathrm{R}}
\def\SO{\mathrm{SO}}
\def\GL{\mathrm{GL}}

\def\Cl{\mathrm{Cl}}
\def\Fr{\mathrm{Fr}}
\def\rg{\mathrm{g}}

\DeclareMathOperator\Crit{Crit}
\DeclareMathOperator\Diff{Diff}
\DeclareMathOperator\Div{div}
\DeclareMathOperator\Dom{Dom}

\DeclareMathOperator\vol{vol_g}
\DeclareMathOperator\tr{tr}
\DeclareMathOperator{\Ima}{Im}
\DeclareMathOperator{\Spa}{span}
\DeclareMathOperator{\Stab}{Stab}
\DeclareMathOperator\End{End}

\newcommand{\qandq}{\quad\text{and}\quad}

\newcommand{\qforq}{\quad\text{for}\quad}
\usepackage{cancel}
\usepackage[all]{xy}
\usepackage{multicol}

\begin{document}

\title{\vspace{-3cm}
Harmonic flow of $\S7$-structures}
\author{Shubham Dwivedi, Eric Loubeau \& Henrique S{\'{a}} Earp}
\date{\today}

\maketitle

\begin{abstract}
    We formulate and study the isometric flow of $\mathrm{Spin}(7)$-structures on compact $8$-manifolds, as an instance of the harmonic flow of geometric structures. Starting from a general perspective, we establish Shi-type estimates and a correspondence between harmonic solitons and self-similar solutions for arbitrary isometric flows of $H$-structures. We then specialise to $H=\mathrm{Spin}(7)\subset\mathrm{SO}(8)$, obtaining conditions for long-time existence, via a monotonicity formula along the flow, which leads to an $\varepsilon$-regularity theorem. Moreover, we prove Cheeger-Gromov and Hamilton-type compactness theorems for the solutions of the harmonic flow, and we characterise Type-$\mathrm{I}$ singularities as being modelled on shrinking solitons. We also establish a Bryant-type description of isometric $\mathrm{Spin}(7)$-structures, based on squares of spinors, which may be of independent interest. 
\end{abstract}

\begin{adjustwidth}{0.95cm}{0.95cm}
    \tableofcontents
\end{adjustwidth}

\section{Introduction}

We formulate and study the harmonic flow of $\S7$-structures on a compact, oriented and spinnable $8$-manifold. Since the flow runs among isometric $\S7$-structures, this problem fits naturally in the recent thread of activity on isometric/harmonic flows of $\mathrm{G}_2$-structures \cites{Grigorian2017,Grigorian2019,Bagaglini2019,dgk-isometric} and of almost complex structures \cites{He2019,He2019a}. These can be understood as instances of an abstract Dirichlet gradient flow of sections of a natural homogeneous fibre bundle of $H$-structures, according to the general theory developed in \cite{loubeau-saearp}. 

By contrast, to our knowledge, there have yet been no studies exploring \emph{any} flow of $\S7$-structures, ever since the fundamental reference \cite{karigiannis-spin7}.
We organise therefore the exposition around the harmonic $\S7$-flow, even though several analytic results are actually obtained for arbitrary $H$-structures, and can thus be seen as advances in the general theory of \emph{harmonic section flows} formulated in \cite{loubeau-saearp}. Namely, in \textsection\ref{sec: H-structures}, we describe deformations of $H$-structures, along the lines of \cite{karigiannis-spin7}. This leads to a general definition of harmonic $H$-solitons, as well as their correspondence to self-similar solutions in Proposition~\ref{def-sssol-GENERAL}, which generalises \cite[\textsection 2.5]{dgk-isometric}. Moreover, we compute Weitzenb\"ock and Bochner formulas for harmonic sections [\textsection\ref{WnB}], and use them to establish Shi-type estimates [Proposition~\ref{Prop:Shi-estimates}] for the gradient flow of the Dirichlet energy functional in \eqref{eq: Dirichlet energy}, which are instrumental in the study of the long-time behaviour of the harmonic $\S7$-flow.
While our approach builds upon previous broad-ranging results on the harmonic flow of geometric structures from \cite{loubeau-saearp}, at times it may be effortlessly formulated in the even ampler class of sections of Riemannian submersions with totally geodesic fibres.  

Readers especially interested in the $\S7$ narrative may wish to skip directly to  \textsection\ref{sec:prelims}. Throughout the text, a \emph{$\S7$-structure manifold $(M^8,\Phi,g)$} consists of a smooth oriented and spinnable $8$-manifold $M$, endowed with a $\S7$-structure $\s7\in\Omega^4(M)$ which induces the Riemannian metric $g$. 

We begin by reviewing a number of well-known facts about irreducible $\S7$-representations within differential forms and spinors. In \textsection\ref{sec: proof of Spin(7) Bryant formula}, we establish a Bryant-type formula for isometric $\S7$-structures as sections of a $\bRP^7$-bundle; this formula emulates a well-known result for isometric $\mathrm{G}_2$-structures on a $7$-manifold \cite[(3.6)]{Bryant2006}, yet it does not seem to appear anywhere in the literature in the following explicit form:

\begin{thmx}
\label{prop: Spin(7) Bryant formula}
    Let $(M^8,\s7_0,g_0)$ be a $\S7$-structure manifold. Then any other $\S7$-structure inducing the same Riemannian metric $g_0$ can be parametrised by a function $f\in C^\infty(M)$  and a vector field $X\in\sX(M)$ such that $f^2 + |X|^2 = 1$: 
$$
\Phi_{(f,X)} = (f^2- |X|^2) \Phi_0 + 2 f \Theta_{0,X} + 8(X\wedge (X\lrcorner\Phi_0)),
$$
    where the $4$-form $\Theta_{0,X}(a,b,c,d) := \Phi_0 (a \cdot X,b,c,d)$ is defined by spinorial multiplication.
\end{thmx}

In \textsection \ref{sec: Basics of the Spin(7)-flow}, we deduce the specific form of the harmonic $\S7$-flow, as an instance of the general harmonic flow of geometric structures. Consider the negative gradient flow of the functional $16E(\s7)=8\Vert T_\s7\Vert^2_{L^2(M)}$ [cf.  \eqref{energyfuncdefn}], restricted to the isometry class $\llbracket \s7_0 \rrbracket$, in which $T_\s7$ denotes the full torsion tensor. By Lemma \ref{lemma:E1var}, the evolution of $\s7$ is determined, somewhat similarly to the $\mathrm{G}_2$ case, by the divergence of torsion:

\begin{definition}[Harmonic $\S7$-flow]
\label{iflowdefn}
    Let $(M^8, \s7_0,g)$ be a compact $\S7$-structure manifold. The \emph{harmonic $\S7$-flow} is the following initial value problem:
\begin{align} 
\label{eq: Har Spin(7) Flow} 
 \left\{\begin{array}{rl} 
      & \dfrac{\pt \s7}{\pt t} = \Div T \diamond \s7 \\
      & \s7(0) =\s7_0
      \tag{HF}
   \end{array}\right.,
\end{align}
    where $\diamond$ denotes the infinitesimal action of $\End{TM}$ on $\Omega^4(M)$ [Definition \ref{def: diamond}]. 
\end{definition}

As a consequence, we obtain short-time existence and uniqueness of solutions effortlessly from the general theory in \cite[Theorem 1]{loubeau-saearp}. Moreover, after some elementary inquiry into deformations of $H$-structures, we obtain a general correspondence between harmonic $H$-solitons and self-similar solutions [\textsection\ref{sec: Spin(7)-solitons}],  which in the case of $\S7$ reads as follows:
\newpage
\begin{thmx}
\label{cor: iso soliton => self-sim sol}
    Suppose that a $\S7$-structure $\hat\Phi$ on $M$, a vector field $\hat X \in\sX(M)$ and $c\in\{-1,0,1\}$ define a harmonic soliton [Definition \ref{def: harmonic Spin(7)-soliton}]. Then there exist:
\begin{itemize}
    \item 
    a function $\rho\in C^\infty(\bR,\bR^+)$ such that $\rho(0)=1$,
    \item
    a function $\alpha\in C^\infty(\Dom(\alpha),\bR)$, an instant $\hat t \in \Dom(\alpha)$ such that $\alpha(\hat t )=1$, and an interval $\hat I \subset\Dom(\alpha)$ containing both $\{\hat t \}$ and $\{{\hat t} .\infty\}$,
\end{itemize}   
    such that the $1$-parameter family $\{f_t\}\subset\Diff{M}$ defined by
\begin{align}
\left\{\begin{array}{rl}
\ddt f_t &= \alpha(t)\hat X\circ f_t  \\
     f_{\hat t }&= \operatorname{Id}_M 
\end{array}\right.,
\quad  t\in \hat I,
\end{align}
    induces a self-similar solution [Definition \ref{def-sssol}] of the \eqref{eq: Har Spin(7) Flow} by $\Phi_t:=\rho^4(t)f_t^*\hat\Phi$. Moreover, the family of vector fields $\{X_t:=\alpha(t)\hat X \}\subset\sX(M)$ satisfies the \emph{harmonic $\S7$-soliton} condition [Lemma \ref{lemma: vec fields for self-sim sols}]:
$$
\Div{T_t}= X_t\lrcorner T_t + \nabla_7 (X_t),
\quad \textup{for\ all}\  t\in \hat I.
$$ 
    Finally, the data $(\rho, \alpha, \hat t)$  can be explicitly constructed, cf. Remark \ref{rem: explicit self-sim sols-GENERAL}.
\end{thmx}

The analytic core of the paper is developed in \textsection\textsection\ref{sec: Basics of the Spin(7)-flow} and \ref{sec: analysis of evolution}, following closely the study of the harmonic $\G2$-flow in \cite{dgk-isometric} and the methods therein. We prove a general Cheeger--Gromov-type compactness for $\S7$-structures in Theorem \ref{thm:cptthm}, and then use it to prove a Hamilton-type compactness theorem for solutions of \eqref{eq: Har Spin(7) Flow} in Theorem~\ref{thm:flowcptthm}. Furthermore, we establish a monotonicity formula along the flow [Theorem~\ref{thm:almostmon}], which leads to $\varepsilon$-regularity [Theorem \ref{thm:epsilonreg}]. 
We introduce moreover an \emph{entropy} functional for the solutions of \eqref{eq: Har Spin(7) Flow}. Then indeed small initial torsion implies long-time existence and convergence to a critical point and, using Corollary~\ref{cor:enttor}, we show that small initial entropy eventually implies pointwise small torsion.
In summary, the following is proved in Theorems \ref{thm: smalltorconv} and \ref{thm:smallent}:

\begin{thmx}[small torsion and entropy]
\label{thm:small_tor_ent}
    In the context of Definition \ref{iflowdefn}, the following hold:
\begin{enumerate}[(i)]
    \item
    For every $\varepsilon>0$ and $\sigma >0$, there exists a constant  $\lambda_\varepsilon=\lambda_\varepsilon(g,\sigma) >0$ such that, if the entropy \eqref{eq:entropyeqn} satisfies
\begin{align}
\label{smallenteq}
    \lambda(\s7_0, \sigma) < \lambda_\varepsilon,
\end{align}
    then $|T_{\Phi(\tau)}| < \varepsilon$, where $\tau$ is the maximal existence time of \eqref{eq: Har Spin(7) Flow}.
    
    \item 
    For every $\delta >0$, there exists $\varepsilon (\delta, g) >0$ such that, if $|T_{\Phi_0}| < \varepsilon$, then \eqref{eq: Har Spin(7) Flow} exists for all time and converges subsequentially smoothly to a $\S7$-structure $\Phi_{\infty}$ satisfying
$$
\Div T_{\Phi_{\infty}} =0, 
\quad
    |T_{\Phi_{\infty}}| < \delta,
\quad
    |\nabla^k T_{\Phi_{\infty}}| < C_k, \quad\forall\  k\geq 1,  
$$
    for constants $C_k < +\infty$ depending only on $(M,g)$.
\end{enumerate}    
    
\end{thmx}

Moving on to the analysis of singularities for the flow in \textsection\ref{sec: singsize}, we obtain an upper bound on the ``size'' of the singular set $S$, which is closed:
\begin{thmx}
\label{thm: singsize}
    In the context of Definition \ref{iflowdefn}, let 
\begin{align}
\label{eq:singsize}
    E(\s7_0)= \frac 12\int_M |T_{\s7_0}|^2 \vol \leq E_0.  
\end{align}
    Suppose that the maximal smooth harmonic $\S7$-flow $\{\s7(t)\}_{t\in [0,\tau)}$ starting at $\s7_0$ blows up at time $\tau< +\infty$. Then, as $t\rightarrow \tau$, \eqref{eq: Har Spin(7) Flow} 
    converges smoothly to a $\S7$-structure $\s7_{\tau}$ away from a closed set $S$, with finite $6$-dimensional Hausdorff measure satisfying 
\begin{align*}
    \mathcal{H}^6(S) \leq CE_0,  
\end{align*}
    for some constant $C<\infty$ depending on $g$. In particular, the Hausdorff dimension of $S$ is at most 6.
\end{thmx}

Finally, in \textsection\ref{sec: type_I sing}, we discuss Type-\rom{1} singularities and prove that they are modelled on shrinking harmonic solitons, cf. \ref{sec: Spin(7)-solitons}.

\bigskip

\noindent
\textbf{Notation and Conventions.} We use the symbol $*$ to denote various tensor contractions, the precise form of which is unimportant, and thus we use $\star$ instead for the Hodge star operator. The symbol $C$ denotes some positive real constant, which may change from line to line in the derivation of an estimate. We frequently use Young's inequality $ab \leq \tfrac{1}{2\varepsilon} a^2 + \tfrac{\varepsilon}{2} b^2$, for any $\varepsilon, a, b > 0$.

Throughout the paper, we compute in a local orthonormal frame $\left\{ \frac{\pt}{\pt x^{k}}\right\}$, all indices are subscripts and repeated indices are summed from $1$ to $8$. The symbol $\Delta$ always denotes the \emph{analyst's Laplacian} $\Delta = \nabla_k \nabla_k$ which is the opposite of the \emph{rough Laplacian} $\nabla^* \nabla$.

The Riemann curvature and the Ricci tensor are given respectively by
\begin{align*}
    R_{ijkm} \frac{\pt}{\pt x^{m}} 
    &= (\nabla_{i} \nabla_{j} - \nabla_{j} \nabla_{i})\frac{\pt}{\pt x^{k}},\\
    R_{jk} 
    &= R_{ljkl},
\end{align*}
and the Ricci identity is
\begin{equation} 
\label{eq: ricciidentity}
    \nabla_{k} \nabla_{i} X_l - \nabla_{i} \nabla_{k} X_l 
    = - R_{kilm} X_m.
\end{equation}
Schematically, for any tensor $S$, the Ricci identity implies that
\begin{equation} 
\label{eq: riccischematic}
    \del^m \Delta S - \Delta \del^m S 
    = \sum_{i=0}^{m}\del^{m-i}S*\del^i\Riem. 
\end{equation}
The Riemannian second Bianchi identity is
\begin{equation*}
\nabla_i R_{jkab} + \nabla_j R_{kiab} + \nabla_k R_{ijab} = 0,
\end{equation*}
which contracts in $i,a$ to
\begin{equation} 
\label{eq: riem2ndBid}
    \nabla_i R_{ibjk} 
    = \nabla_k R_{jb} - \nabla_j R_{kb}.
\end{equation}

\noindent\textbf{Acknowledgements:} 
We thank Robert Bryant for suggesting the approach via squares of spinors, leading to Theorem \ref{prop: Spin(7) Bryant formula}. We are also grateful to BIRS-CMO, for providing the excellent working environment where this project originated.
SD would like to thank Spiro Karigiannis and Panagiotis Gianniotis, for discussions on various aspects of geometric flows, and Thomas Walpuski, for discussions on the spinorial description of $\S7$-structures.

EL and HSE benefited from an ongoing CAPES-COFECUB bilateral collaboration (2018-2021), granted by the Brazilian Coordination for the Improvement of Higher Education Personnel (CAPES) – Finance Code 001 [88881.143017/2017-01], and COFECUB [MA 898/18], and from a CAPES-MathAmSud (2021-2022) grant [88881.520221/2020-01]. 

HSE has also been funded by the São Paulo Research Foundation (Fapesp)  \mbox{[2018/21391-1]} and the Brazilian National Council for Scientific and Technological Development (CNPq)  \mbox{[307217/2017-5]}.

\newpage
\section{$H$-Structures, homogeneous sections and harmonic flow}
\label{sec: general theory of H-str}

In many respects, $\S7$-structures should be seen as a particular occurrence of $H$-structures and some of the analytical properties of the harmonic flow are best formulated in the general setting of sections of Riemannian submersions with totally geodesic fibres. 

In \textsection\ref{sec: H-structures}, we briefly review the unifying abstract framework  required to define the harmonicity of $H$-structures, by mediation of their corresponding sections, describe their deformation theory and classify the solitons of the harmonic section flow. For further details, the reader is kindly referred to \cite{loubeau-saearp} and references therein.
Subsequently, in \textsection\ref{sec: Homogeneous sections}, we adopt the wider perspective of sections of Riemannian submersions in order to prove a very general Shi-type estimate, yielding fractional-polynomial time-bounds on all derivatives of sections along the flow.

\subsection{$H$-Structures}
\label{sec: H-structures}

A \emph{geometric structure} defined by an $H$-invariant tensor $\xi\in\cT^{p,q}(M)$ on a manifold $M$ can be viewed as a section of the homogeneous fibre bundle $\pi:N:=P/H\to M$, by reduction of the frame bundle $P:=\Fr(M)$ and a one-to-one correspondence [see \eqref{eq: Xi correspondence}]:
$$
\{\xi:M\to \cT^{p,q}(M)\} 
\quad\leftrightarrow\quad 
\{\sigma:M\to N\}.
$$
Assuming $M$ compact, a Riemannian metric $g$ on $M$ defines a suitable fibre metric $\eta$ on $N$, thus assigning a Dirichlet energy to such sections  $\sigma\in\Gamma(N)$: 
\begin{equation}
\label{eq: Dirichlet energy}
    E(\sigma):=\frac{1}{2}\int_M |d^\cV\sigma|_\eta^2 \vol,
\end{equation}
where the \emph{vertical torsion} $d^\cV\sigma$ is the projection of $d\sigma$ onto the distribution $\cV=\ker\pi_*\subset TN$. An $H$-structure is said to be \emph{harmonic} if $\sigma\in\Crit({E})$, and \emph{torsion-free} if $d^\cV\sigma=0$. 
The critical set is the vanishing locus of the \emph{vertical tension} field $\tau^\cV(\sigma):=\tr_g\nabla d^\cV\sigma$, where $\nabla^\cV$ is the vertical part of the Levi-Civita connection of $(N,\eta)$. The explicit form of $\tau^\cV$, in each particular context, defines a natural geometric PDE.

A sufficient setup for the development of this general theory consists of the following data:
\begin{equation}
\label{eq: conditions of HSF}
\tag{$\mathcal{A}$}
\text{\parbox{.90\textwidth}
    {$\bullet$ $G$ a compact semi-simple Lie group, with natural bi-invariant metric $\eta$;\\
    $\bullet$ $H\subset (G,\eta)$ a normal reductive Lie subgroup;\\
    $\bullet$ $P$ a principal $G$-bundle over a closed Riemannian manifold $(M^n,g)$;  \\
    $\bullet$ $\eta$ an $H$-invariant fibre metric on $P$, constructed from any compatible metric on $G$. 
}}
\end{equation}    
\begin{definition}
    Under the assumptions \eqref{eq: conditions of HSF}, the \emph{harmonic section flow} is the negative gradient flow associated to the Dirichlet energy \eqref{eq: Dirichlet energy}:
\begin{equation}
\label{eq: HSF}
\tag{HSF}
    \left\{\begin{array}{rl}
    \displaystyle \partial_t \sigma_t &= \tau^\cV (\sigma_t)\\
    \sigma_t|_{t=0} &= \sigma \in \Gamma(\pi)
    \end{array}\right.,
\quad\text{on}\quad 
    M_{\rT}:=M\times \left[0,\rT\right[. 
\end{equation}
\end{definition}
We will see below that, by mediation of the \emph{universal section} $\Xi$ [cf. \eqref{eq: Xi correspondence}], to each family $\{\sigma_t\}$ solving \eqref{eq: HSF}, there corresponds a family of $H$-structures   $\{\xi_t:=\sigma_t^*\Xi\}$.

\subsubsection{The geometric setting of $H$-structures}

We describe $H$-structures as sections of an ad-hoc homogeneous fibre bundle.
If $p:P\to M$ is a principal $G$-bundle, with Lie group $G$, and $H\subseteq G$, denote by $q: P\to N$ the corresponding principal $H$-bundle and $\pi : N:=P/H \to M$ the projection. We assume $H$ to be naturally reductive, i.e. the existence of an orthogonal complement $\fm$ satisfying
$$
\fg =\fh\oplus\fm 
\qandq 
\mathrm{Ad}_G(H)\fm\subseteq\fm.
$$
A connection $\omega\in\Omega^1(P,\fg)$  on $P$ (e.g. if $P=\Fr(M)$), will induce a splitting
$$
TN=\cV \oplus \cH
$$
with
$
\cV:=\ker\pi_*=q_*(\ker p_*) 
$
and
$
\cH:=q_*(\ker\omega).
$

\begin{multicols}{2}
\label{figure 1}
Let $\ufm\to N$ be the vector bundle associated to $q$ with fibre $\fm$, points of which are the $H$-equivalence classes defined by the infinitesimal action of $w\in \fm$ on $z\in P$:
$$
z\bullet w:=[(z,w)]_{H}=\cI(q_*(w_z^*)),
$$
where $w_z^*$ is a fundamental left-invariant vector field.
\columnbreak

$$\xymatrix{
\llap{$z\in$ } P 
	\ar[dr]^q 
    \ar[dd]_p
    & & {\ufm} \ar[dl] \rlap{ $:=P\times_H \fm \ni  z\bullet w $}\\ 
  &\llap{$y\in$ } N 
  \ar[dl]_\pi 
  	& \ar[l]\cV
  	\ar[u] ^\cI  \\ 
\llap{$x\in$ }  M 
& & }$$
\end{multicols}
This defines a vector bundle isomorphism
\begin{equation}
\label{eq: isomorphism I:V->m}
\begin{array}{rcccl}
    \cI&:&\cV&\tilde{\rightarrow}&\ufm\\
    &&q_*(w_z^*)&\mapsto&z\bullet w 
\end{array}
\end{equation}
and the $\fm$-component $\omega_\fm \in \Omega^1(P,\fm)$ of the connection is $H$-equivariant and $q$-horizontal, so it projects to a \emph{homogeneous connection form} $f \in\Omega^1(N,\ufm)$ defined by:
\begin{equation}
\label{eq: homog connection form f}
    f (q_*(Z)):= z \bullet \omega_\fm(Z) 
    \qforq
    Z\in T_zP.
\end{equation}
\begin{multicols}{2}
On $\pi$-vertical vectors, $f $ coincides with the canonical isomorphism (\ref{eq: isomorphism I:V->m}), while $\pi$-horizontal vectors are in the kernel: 
$$
f (v_y)=\cI(v_y^\cV), 
\qforq
v_y\in T_yN.
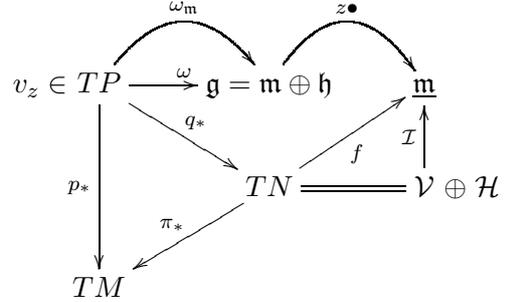
$$
\columnbreak
$$\xymatrix{
\llap{$v_z\in$ }TP 
	\ar[dr]^{q_*} 
    \ar[dd]_{p_*} 
    \ar[r]^\omega 
    \ar@/^2pc/[r]^{\omega_\fm}
    & \fg = \fm \ar@/^2pc/[r]^{z\bullet} \oplus \fh & {\ufm} \\ 
  &TN 
  	\ar[dl]_{\pi_*} 
    \ar@{=}[r]
    \ar[ur]_f 
    & \cV \rlap{ $\oplus\;\cH$}
  		\ar[u] ^\cI  \\ 
TM & &     
}$$
\end{multicols}

Let $\cF$ be a rank $r$ sub-bundle of $\bigoplus \cT^{p,q}(M)$, with fibre $V=\bR^r$. 
When $P=\Fr(M)$, we have a natural monomorphism of principal bundles
$\rho: P\hookrightarrow \Fr(\cF)$, which identifies, at each $x\in M$, the element $z_x\in P_x$ with a frame of $\cF_x$, i.e., with a linear isomorphism onto the typical fibre:
$$
\rho(z_x):\cF_x \tilde\to \;V.
$$
A section $\xi\in\Gamma(\cF)$ is a \emph{geometric structure}, modelled on a fixed element $\xi_0 \in V$, if, for any $x\in M$, there is a frame of $T_x M$ identifying $\xi(x)$ and $\xi_0$.
Suppose now $H\subseteq G=\SO^g(n)$ fixes the model structure $\xi_0$:
\begin{equation}
\label{eq: H=Stab(xi)}
    H=\Stab(\xi_0).    
\end{equation}
\begin{multicols}{2}
In view of \eqref{eq: H=Stab(xi)}, a \emph{universal section} $\Xi\in\Gamma(N,\pi^*(\cF))$ is well-defined by
\begin{align}
\label{eq: universal section}
    \Xi(y):=y^*\xi_0.
\end{align}
Explicitly, one assigns to $y\in N$ the vector of $\cF_{\pi(y)}$ whose coordinates are given by $\xi_0$ in any frame $\rho(z_{\pi(y)})$.
\columnbreak
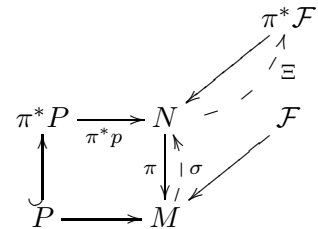
$$\xymatrix{
	&&\pi^*\cF \ar[dl]
    \\
\pi^*P 
	\ar[r]_{\pi^*p}
    &N \ar[d]_\pi \ar@{-->}@/_1pc/[ur]_(0.75)\Xi& \cF \ar[dl]\\
P \ar@{^{(}->}[u] \ar[r]&M \ar@{-->}@/_/[u]_\sigma &
}$$
\end{multicols}
Now to each section $\sigma\in\Gamma(M,N)$ one associates a  geometric structure $\xi\in\Gamma(M,\cF)$ modelled on $\xi_0$ by
\begin{align}
\label{eq: Xi correspondence}
    \xi_\sigma:=\sigma^*\Xi
    =\Xi\circ \sigma,
\end{align}
and, conversely, to a given geometric structure $\xi\in\Gamma(M,\cF)$, stabilised by $H$, one associates, at $x\in M$, an $H$-class of frames of $T_x M$ which, in turn, identifies an element of $\pi^{-1}(x)$, i.e. an element $\sigma (x)$ in the fibre of $\pi : N \to M$ over $x\in M$.

\subsubsection{Deformations of $H$-structures}

We establish some fundamental facts on the general evolution of an $H$-structure $\xi$ under $\mathrm{GL}(m,\mathbb{R})$-action. Notably, we identify the role of the stabiliser of $\xi$ and its implications on the evolution of the metric. A key result for the ensuing study of solitons is the action  on $\xi$ of the metric itself.

Separating symmetric (and traceless therein) and skew-symmetric parts of $(1,1)$-tensors, we have an orthogonal splitting
\begin{align*}
    \End(TM)=\Gamma(T^*M\otimes TM)=\Omega^0\oplus S_0\oplus \Omega^2.   
\end{align*}
Identifying the irreducible components $\fh \simeq \Omega^2_\fh$ and $\fm \simeq \Omega^2_\fm$ 
in  $\Omega^2 \simeq \fso(\dim M)$, we further decompose   $\Omega^2(M)$ as
\begin{align}
\label{eq:splitting TM* x TM-GENERAL}
 \Gamma(T^*M\otimes TM)=\Omega^0\oplus S_0 \oplus \Omega^2_\fm\oplus \Omega^2_\fh.  
\end{align}
With respect to this splitting, we can write $A=\frac{1}{\dim M} (\tr A)g+A_0+A_\fm+A_\fh$ where $A_0$ is a symmetric traceless $2$-tensor. 

In view of \eqref{eq: H=Stab(xi)} and \eqref{eq: Xi correspondence}, the orthogonal projection along $\Omega^2_\fh$ can be geometrically interpreted as selecting the non-trivial infinitesimal deformations of $\xi$ as an $H$-structure. To see that more clearly, write $\xi\in \cT^{p,q}$ in local coordinates $\{ x_1, \dots, x_m\}$: 
$$
    \xi = \sum \xi^{i_1 \dots i_p}_{j_1 \dots j_q} \frac{\partial}{\partial x_{i_1}}\otimes \dots\otimes\frac{\partial}{\partial x_{i_p}}\otimes dx_{j_1}\otimes\dots \otimes dx_{j_q} .
$$
A local change of frame $\rg\in \mathrm{GL}(m,\mathbb{R})$ acts on $T_pM$ and on $T^*_pM$, respectively, by
    $$
    \begin{array}{rrll}
        \rg: & T_p M&\to& T_{p} M\\
        & X &\mapsto& \rg^{-1}.X
    \end{array}
    \qandq
    \begin{array}{rrll}
        \rg: & T^*_p M&\to& T^*_{p} M\\
        & \alpha &\mapsto& \rg.\alpha .
    \end{array}
    $$
By extension, $\rg$ acts on $\cT^{p,q}$ by
\begin{align}
\label{eq: GL acts on T^p,q}
    \rg.\xi = \sum \xi^{i_1 \dots i_p}_{j_1 \dots j_q} \rg^{-1}.\frac{\partial}{\partial x_{i_1}}\otimes \dots\otimes \rg^{-1}.\frac{\partial}{\partial x_{i_p}}\otimes \rg.dx_{j_1}\otimes\dots \otimes \rg.dx_{j_q} .
\end{align}
Differentiating the action of $e^{tA}\in \GL (m,\bR)$ at the identity, with $A \in \mathfrak{gl}(m, \mathbb{R})$ and  $t\in\,]-\varepsilon,\varepsilon[$, we have:
\begin{align}
\label{eq: gl acts on T^p,q}
    \left.\ddt\right\vert_{t=0} e^{tA}.\xi = \sum \xi^{i_1 \dots i_p}_{j_1 \dots j_q} 
    \sum_{r,s=1}^{p,q} \Big\{
    &- \frac{\partial}{\partial x_{i_1}}\otimes \dots \otimes A.\frac{\partial}{\partial x_{i_r}}\otimes \dots\otimes \frac{\partial}{\partial x_{i_p}}\otimes dx_{j_1}\otimes\dots \otimes dx_{j_q}
    \nonumber\\
    &    +
    \frac{\partial}{\partial x_{i_1}}\otimes \dots\otimes \frac{\partial}{\partial x_{i_p}}\otimes dx_{j_1}\otimes\dots \otimes A.dx_{j_s}\otimes\dots\otimes dx_{j_q}
    \Big\}.
\end{align}

\begin{definition}
\label{def: diamond}
    Let $\xi\in\cT^{p,q}$ be a tensor on $M$ and $A\in\End(TM)$, and define the \emph{infinitesimal deformations} of $\xi$ by:
    $$
    \begin{array}{rcl}
        \diamond \, \xi: \quad \End(TM) &\to& \cF=\cT^{p,q}  \\
        A&\mapsto & A \diamond \xi:= \left.\ddt\right\vert_{t=0} e^{tA}.\xi. 
    \end{array}
    $$
    The \emph{net degree} of $\xi$ is defined as the difference $\ell:=q-p$ between its numbers of covariant and contravariant indices.
\end{definition}

\begin{proposition}
    Let $H\subset G=\SO(m)$. Let $A\in\fgl (m,\bR)$ act on an $H$-structure $\xi\in\cF=\cT^{p,q}$ by \eqref{eq: gl acts on T^p,q}.
    Then, in terms of the splitting \eqref{eq:splitting TM* x TM-GENERAL}, the following assertions hold:
\begin{enumerate}
    \item 
    $\ker (\diamond \, \xi) \supseteq \Omega^2_{\fh}$.
    
    \item 
    A general $GL(m,\bR)$-variation of $\xi$ can be written as $\ddt\xi=A\diamond\xi$, for some $A=B+C$, with $B\in S^2(M)$ and $C\in \Omega^2_{\fm}$.
    
    \item
    If $g_t$ is a co-evolving metric on $M$, and  $H\subset\SO^{g_t}(m)$ along the flow, then $\ddt g_t=2B$. Equivalently, any isometric variation of $\xi$ is characterised by $A=C\in \Omega^2_{\fm}$.
\end{enumerate}
\end{proposition}
\begin{proof}
Item 1. follows from Definition \ref{def: diamond}, and 2. is a trivial consequence of 1. and \eqref{eq:splitting TM* x TM-GENERAL}. The first assertion in 3. can be proved following  \cite[Prop 3.1]{karigiannis-spin7} for the general case; the second assertion is then a consequence of 2..
\end{proof}

In particular, if $A$ is the metric tensor, i.e. $A_{ij} = \delta_{ij}$ in normal coordinates at a given point, then
    $$
    A.\frac{\partial}{\partial x_{i_r}} = \frac{\partial}{\partial x_{i_r}} ; \quad A.dx_{j_s} = dx_{j_s}
    $$
    so indeed:
\begin{lemma}
\label{lem: homogeneity and diamond}
    Let $g$ be a Riemannian metric and $\ell$ be the net degree of $\xi$; then $$
    g\diamond\xi=\ell.\xi.
    $$
\end{lemma}

\begin{lemma}
The torsion $T=d^\cV\sigma$ is identified with an element of $\Omega^1(M,\sigma^*\cF_{\fm})$ by the relation
\begin{align}
\label{eq: nabla of xi-GENERAL}
    \nabla_Y\xi 
    &= (Y \lrcorner T) \diamond \xi, 
    \quad\forall\, Y \in \sX(M).
\end{align}    
\end{lemma}
\begin{proof}
    We know from \cite[Lemma 1]{loubeau-saearp} that $\nabla_Y\Xi=\mathrm{f}(Y).\Xi$, where $\mathrm{f} \in\Omega^1(N,\underline{\fm})$ is the homogeneous connection form associated to the connection $\omega\in\Omega^1(P,\mathfrak{g})$ and the splitting $TN=\cV \oplus \cH$:
\begin{equation*}
    \mathrm{f} (q_*(Z))= z \bullet \omega_\fm(Z) 
    \qforq
    Z\in T_zP , z\in P.
\end{equation*}
    Since $\xi = \Xi\circ\sigma$, we have
    \begin{align*}
        \nabla_Y\xi &= \nabla_Y (\Xi\circ\sigma) 
        = (\nabla_{d\sigma(Y)} \Xi)\circ\sigma \\
        &= f(d\sigma(Y)) . \Xi\circ\sigma 
        = \cI(d^\cV\sigma(Y)). \xi \\
        &= (Y \lrcorner T) \diamond \xi .
        \qedhere
    \end{align*}
\end{proof}

Given $Y\in \sX(M)$, in view of the splitting \eqref{eq:splitting TM* x TM-GENERAL} of $\End(TM)$, we may write 
$$
\nabla Y=\frac 12 \cL_Yg +\nabla_\fm (Y) +\nabla_\fh(Y),
$$
where $\nabla_k(Y):=\pi_k(\nabla Y)\in \Omega^2_k$, in particular  $\nabla_\fh Y\in\Omega^2_\fh \subset \ker(\diamond\, \xi)$.
\begin{lemma}
In terms of the torsion $T=d^\cV\sigma$, the Lie derivative of a $H$-structure is
\begin{align}
\label{eq: Lie of xi-GENERAL}
    \cL_Y \xi
    &= (Y\lrcorner T +\frac 12 \cL_Yg +\nabla_\fm (Y))\diamond \xi.
\end{align}    
\end{lemma}
\begin{proof}
    If $\xi$ is a $(p,q)$-form then, for $X_i \in \Gamma(TM)$ and $\sigma_j \in \Gamma(T^*M)$,
    \begin{align*}
        (\cL_Y \xi) (X_i, \sigma_j) 
        =&\  Y (\xi (X_i, \sigma_j)) - 
        \xi ( \cL_Y X_i, \sigma_j) - \xi ( X_i, \cL_Y \sigma_j) \\
        =&\  (\nabla_Y \xi) (X_i, \sigma_j) + \xi (\nabla_Y X_i, \sigma_j) + \xi (X_i, \nabla_Y \sigma_j) \\
        &- \xi (\nabla_Y X_i, \sigma_j) + \xi (\nabla_{X_i}Y, \sigma_j) 
        - \xi (X_i, \cL_Y \sigma_j)
    \end{align*}
    but
    $
    \cL_Y \sigma_j = \nabla_Y \sigma_j + \sigma_j (\nabla Y),
    $
    so
    \begin{align*}
        (\cL_Y \xi) (X_i, \sigma_j)
        &=  (\nabla_Y \xi) (X_i, \sigma_j) + \xi (\nabla_{X_i}Y, \sigma_j) 
        - \xi (X_i, \sigma_j (\nabla Y)).
    \end{align*}
    On the other hand, by \eqref{eq: nabla of xi-GENERAL},
    \begin{align*}
    (\nabla_Y \xi) (X_i, \sigma_j) &= ((Y \lrcorner T) \diamond \xi )(X_i, \sigma_j)
    \end{align*}
    while, using normal coordinates, 
    \begin{align*}
       & \xi (\nabla_{X_i}Y, \sigma_j) =
        \xi (\frac 12 (\cL_Y g)(X_i), \sigma_j) + \xi ((\nabla_{\fm}Y)(X_i), \sigma_j) + \xi ((\nabla_{\fh}Y)(X_i), \sigma_j) \\
      & =\frac 12 (\cL_Y g)\Big(X_i,\frac{\partial}{\partial x_{k}}\Big)\xi \Big(\frac{\partial}{\partial x_{k}},\sigma_j\Big) + 
      g\Big((\nabla_{\fm}Y)(X_i), \frac{\partial}{\partial x_{k}}\Big) \xi \Big(\frac{\partial}{\partial x_{k}}, \sigma_j\Big) 
      + g\Big((\nabla_{\fh}Y)(X_i), \frac{\partial}{\partial x_{k}}\Big) \xi \Big(\frac{\partial}{\partial x_{k}}, \sigma_j\Big)  
    \end{align*}
    and 
    \begin{align*}
       & \xi (X_i, \sigma_j (\nabla Y)) = \xi (X_i, \frac 12 (\cL_Y g)(\sigma_j)) + \xi (X_i,(\nabla_{\fm}Y)(\sigma_j)) +
       \xi (X_i,(\nabla_{\fh}Y)(\sigma_j)) \\
       & =\frac 12 (\cL_Y g)(\sigma_j, dx_{k})\xi (X_i,dx_{k} ) + 
       (\nabla_{\fm}Y)(\sigma_j,dx_{k})\xi (X_i,dx_{k}) +
       (\nabla_{\fh}Y)(\sigma_j,dx_{k})\xi (X_i,dx_{k}) .
    \end{align*}
    So putting these three terms together, and keeping in mind Definition \ref{def: diamond}:
    \begin{align*}
        \cL_Y \xi &= (Y \lrcorner T) \diamond \xi + \frac 12 (\cL_Y g)\diamond \xi + \nabla_{\fm}Y \diamond \xi + \nabla_{\fh}Y \diamond \xi.
    \end{align*}
    Now $H$ is the stabiliser of $\xi$, so the last term vanishes: $\nabla_{\fh}Y \diamond \xi = 0$.
\end{proof}

\subsubsection{Harmonic $H$-solitons and self-similar solutions}

The objective of this sub-section is to prove that solitons, i.e. generalised fixed points of the flow, are the same thing as self-similar solutions, in the sense that they are preserved by a family of diffeomorphisms parametrised by time.

\begin{definition}
\label{def-sssol-GENERAL} 
    Let $\{\sigma_t\}$ be a solution of the harmonic section flow \eqref{eq: HSF},
    and let $\{\xi_t:=\sigma_t^*\Xi\}$ be the corresponding family of $H$-structures. We say that it is a \emph{self-similar solution} if there exist a function $\rho(t)$, with $\rho(0)=1$, an $H$-structure $\xi:=\xi_\sigma$ and a family of diffeomorphisms $f_t:M\rightarrow M$ such that
$$
    \xi_t = \rho(t)^{\ell}{f_t}^*\xi,
    \quad\forall\  t\in \, [0,\rT[,
$$ 
    where $\ell$ is the net degree of $\xi$, cf. Definition \ref{def: diamond}.
    
    Denoting by $\{W_t\}\subset\sX(M)$ the infinitesimal generator of $\{f_t\}\subset\Diff{M}$, the \emph{stationary vector field} of a self-similar solution is defined by
\begin{equation}
\label{eq: stationary vf-GENERAL}
    X_t:=(f_t^{-1})_*W_t\in \sX(M),
    \quad\forall\  t\in \, [0,\rT[.
\end{equation}
\end{definition}

\begin{lemma}
    Along the harmonic section flow of a self-similar solution, as in Definition \ref{def-sssol-GENERAL}, we have $g=\rho(t)^2 f_t^*g$, hence
\begin{equation}
\label{eq: Lie of g}
    \cL_{X_t}g
    =- 2(\log\rho)'g.
\end{equation}
\begin{proof}
Asking that $\xi_t$ remains an $H$-structure, for $H\subset \SO^{g_t}(n)$, implies the metric must scale under the dilation prescribed by self-similarity:
$$
g_t=\rho(t)^2 f_t^*g.
$$
Since the \eqref{eq: HSF} is actually isometric, $g_t\equiv g$ and so:
\begin{align*}
    0&=\ddt g_t
    = 2(\log\rho)' g + \cL_{X_t}g.
    \qedhere
\end{align*}
\end{proof}
    
\end{lemma}

We can now prove that solitons of the harmonic section flow \eqref{eq: HSF} are exactly self-similar solutions. This bears close analogy to Lemmata 2.15 and 2.17 in \cite{dgk-isometric}, to the proof of which the reader is referred for details of the following computations. 

\begin{lemma}
\label{lemma: vec fields for self-sim sols-GENERAL}
    Given a self-similar solution $\{\xi_t\}_{t\in \, [0,\rT[,}$  of the harmonic section flow \eqref{eq: HSF}, the stationary vector field $X_t$ and the torsion $T_t$ satisfy the following relation:
    $$
    \Div{T_t}= X_t\lrcorner T_t + \nabla_\fm (X_t),
    $$
    where $\nabla_\fm :=\pi_\fm\circ\nabla:\sX(M)\to\Omega^2_\fm(M)$.
\end{lemma}
\begin{proof}

Using \eqref{eq: HSF}, \eqref{eq: Lie of g}, \eqref{eq: Lie of xi-GENERAL}, and the fact that $g_t\diamond\xi_t=\ell\xi_t$ [Lemma \ref{lem: homogeneity and diamond}],
\begin{align*}
    \Div T_t\diamond{\xi_t}=\ddt\xi_t 
    &= \ell(\log\rho)' \xi_t +\cL_{X_t}\xi_t\\
    &= 
    \ell(\log\rho)' \xi_t +(X_t\lrcorner T + \frac 12  \cL_{X_t}g_t +\nabla_\fm ({X_t}))\diamond \xi_t\\
    &=(X_t\lrcorner T +\nabla_\fm ({X_t}))\diamond \xi_t,
\end{align*}
where the first and third terms in the second row cancel out, by \eqref{eq: Lie of g}.
\end{proof}

\begin{definition}
\label{def: isometric H-soliton}
    A \emph{harmonic $H$-soliton} on a Riemannian manifold $(M,g)$ is a triple $(\hat\xi,\hat X,c)$, consisting of an $H$-structure  $\hat\xi$ with $H\subset \SO^g(n)$, a vector field $\hat X\in\sX(M)$, and $c\in\bR$, such that
\begin{align}
\left\{\begin{array}{rcl}
    \cL_{\hat X} g&=&c g\\
    \Div{\hat T}&=& \hat X\lrcorner \hat T + \nabla_\fm (\hat X).
\end{array}\right.
\end{align}
\end{definition}

Let us ponder in general whether a given harmonic $H$-soliton $(\hat\xi ,\hat X ,c)$ induces a  self-similar solution, i.e., $\xi_t = \rho(t)^{\ell}{f_t}^* \xi_0$, for some given function $\rho$ and $\xi_0=\hat\xi$. For a real function $\alpha$ and instant $\hat t   \in\bR$ to be determined, such that $\alpha(\hat t   )=1$, consider the $1$-parameter family of diffeomorphisms $\{f_t:M\to M\}$ defined by 
\begin{align*}
\left\{\begin{array}{rcl}
     \ddt f_t &=& \alpha(t)\hat X \circ f_t  \\
     f_{\hat t   }&=& \operatorname{Id}_M ,
\end{array}\right.
\end{align*}
and the vector field $W_t:=\alpha(t)\hat X $.
In particular,
$
\ddt f_t^* = \alpha f_t^*\circ \cL_{\hat X }.
$
Now $g_t:= \rho(t)^2 f_t^*g_0$ satisfies:
\begin{align*}
    0=\ddt g_t 
    &= \{(\log\rho^2)'+ \alpha c\} g_t\\
    \Leftrightarrow\qquad
    \alpha(t)&=-\frac 2c (\log\rho)',
    \qforq c\neq0.
\end{align*}
Imposing $\rho(0)=1=\alpha(\hat t   )$, this is solved, for instance, by
\begin{align*}
    \rho(t)=|t|^{\frac 12}&
    \qandq
    \alpha(t)= \mp\frac{1}{t}, \qforq
    c=\pm1=:-\hat t   ,\\
    \rho(t)=1&
    \qandq
    \alpha(t)= 1, \qforq
    c=0=:\hat t ,
\end{align*}
but also by more general rates (see Remark \ref{rem: explicit self-sim sols-GENERAL} below).

We now wish to determine the stationary family $\{X_t\}\subset\sX(M)$ predicted in Lemma \ref{lemma: vec fields for self-sim sols-GENERAL}:
$$
\Div{T_t}= X_t\lrcorner T_t + \nabla_\fm (X_t).
$$
We begin by observing that the vector field $\hat X$, which generates the soliton's translational symmetry, is itself constant in time: 
$$ 
\ddt (f_t)^{-1}_* \hat X  =\cL_{W_t}\hat X  =\cL_{\alpha \hat X  }\hat X  =0,
$$
so indeed $(f_t)^{-1}_*\hat X  =\hat X  $, for all $t$ along the flow. Therefore, we obtain the desired vector field by setting $X_t:=(f_t)^{-1}_*W_t=\alpha(t)(f_t)^{-1}_*\hat X  =\alpha(t)\hat X  $, as in Lemma \ref{lemma: vec fields for self-sim sols-GENERAL}. In summary:

\begin{proposition}
\label{prop: iso soliton => self-sim sol-GENERAL}
    Suppose that an $H$-structure $\hat\xi$ on $M$, a vector field $\hat X\in\sX(M)$ and $c\in\{-1,0,1\}$ define an harmonic soliton. Then there exist:
\begin{itemize}
    \item 
    a function $\rho\in C^\infty(\bR,\bR^+)$ such that $\rho(0)=1$,
    \item
    a function $\alpha\in C^\infty(\Dom(\alpha),\bR)$, an instant $\hat t \in \Dom(\alpha)$ such that $\alpha(\hat t )=1$, and an interval $\hat I \subset\Dom(\alpha)$ containing both $\{\hat t \}$ and $\{\hat t \cdot \infty\}$,
\end{itemize}   
    such that the $1$-parameter family $\{f_t\}\subset\Diff{M}$ defined by
\begin{align*}
\left\{\begin{array}{rcl}
     \ddt f_t &=& \alpha(t)\hat X\circ f_t  \\
     f_{\hat t }&=& \operatorname{Id}_M ,
 \quad  t\in \hat I ,
\end{array}\right.
\end{align*}    
    induces a self-similar solution of the \eqref{eq: HSF} by $\xi_t:=\rho(t)^\ell f_t^*\hat\xi$. Moreover, the family of vector fields $\{X_t:=\alpha(t)\hat X\}\subset\sX(M)$ satisfies the harmonic soliton condition of Lemma \ref{lemma: vec fields for self-sim sols-GENERAL}:
$$
\Div{T_t}= X_t\lrcorner T_t + \nabla_\fm (X_t),
\quad\forall\  t\in \hat I .
$$ 
\end{proposition}

\begin{remark}
\label{rem: explicit self-sim sols-GENERAL}
    The data $(\rho, \alpha, \hat t )$ in Proposition \ref{prop: iso soliton => self-sim sol-GENERAL} can be explicitly constructed in the following way:
\begin{description}
    \item[$c=\pm1$:]
    Take any $(\rho, \alpha, \hat t )$ solving
    $$
    \alpha(t)=-\frac 2c (\log\rho)'(t)
    \qandq 
    \rho(0)=1=\alpha(\hat t ).
    $$   
    For instance, given any $p>0$,
$$
\rho(t)=|t|^p,
    \quad
    \alpha(t)= \mp\frac{2p}{t}, 
    \quad
    \hat t =-2pc,
    \qandq
    \hat I=
    \begin{cases}
         ]-\infty,\hat t ],& c=1;\\
         [\hat t ,+\infty[,&c=-1,  
    \end{cases}
$$    
    \item[$c=0$:]
    take $\rho,\alpha \equiv 1$, $\hat t =0$ and $\hat I=\bR$.
\end{description}    
\end{remark}

\subsection{Harmonic sections}
\label{sec: Homogeneous sections}

In this sub-section, we take a slight detour through  Riemannian submersions $\pi : (N,\eta) \to (M,g)$ with totally geodesic fibres, and their sections $\sigma: M \to N$. The kernel of $d\pi$ defines a vertical distribution in $TN$ and the metric $\eta$ defines the horizontal sub-bundle as its orthogonal complement.

If $M$ is compact, the (vertical) energy is then the $L^2$-norm of the vertical component of $d\sigma$:
$$
E^\cV (\sigma) = \frac12 \int_M |d^\cV \sigma|^2 \, \vol .
$$
This defines a variational principle for sections of $\pi$ and its critical points, coined {\emph{harmonic sections}}, are characterised, following the same formal principles as in \textsection\ref{sec: H-structures}, by the Euler-Lagrange equation:
$
\tau^\cV (\sigma) =0.
$
Since $\pi$ is a Riemannian submersion with totally geodesic fibres, $\tau^\cV (\sigma)$ is actually the vertical part of the tension field $\tau (\sigma)$ from harmonic maps theory \cites{Wood1986,Wood1997}.
The analytic study of the corresponding heat equation
\begin{align} 
\label{HSecF}
    \frac{d \sigma_t}{dt} 
    = \tau^\cV (\sigma_t) 
\end{align}
aims at finding sufficient conditions for convergence towards a harmonic section.
This will require the local expression of the vertical second fundamental form, which will lead to a Weitzenb\"ock formula on the swapping of derivatives for the vertical tension field and, taking inner-products, an adapted Bochner formula. Under bounded geometry, we obtain a Bochner-type inequality along the harmonic flow.

As a consequence of these formulas, we prove wide-ranging Shi-type estimates on the derivatives of solutions of the harmonic flow. This result naturally holds for $H$-structures and will be exploited for $H=\S7\subset \SO{(8)}$ later on.

\subsubsection{Vertical second fundamental form}
\label{sec: vertical 2ndFF}

To compute the vertical second fundamental form of a section $\sigma \in \Gamma(\pi)$, choose a point $p\in M$ and local coordinates $x_1 , \dots,x_m$ on a domain $U$ around $p$ and $y_1,\dots , y_n$ local coordinates on $N$ such that $y_i = x_i \circ \pi$ for $i\leq m$ and for all $x\in U$, $y_{m+1}|_{\pi^{-1}(x)}, \dots , y_{n}|_{\pi^{-1}(x)}$ are local coordinates\footnote{This part greatly benefited from discussions with C.M.~Wood.} at $\pi^{-1}(x)$.

For $i\leq m$, let $\partial_i$ be the coordinate field for $x_i$ and $\tilde{\partial}_a$ the coordinate field for $y_a$ and for $1\leq i\leq m$, we have 
$d\pi(\tilde{\partial}_i) = {\partial}_i$, while $\tilde{\partial}_r$ is vertical for $m+1 \leq r \leq n$ (but $\tilde{\partial}_i$ is not, a priori, horizontal).
We assume summation on repeated indices and take the index convention that $1\leq i,j,k,l\leq m$ and $m+1 \leq u,v,s,p,r,t\leq n$.

In this coordinate system, the differential of the section $\sigma$ is
$    d\sigma \Big(\frac{\partial}{\partial x_i}\Big)
    = \frac{\partial \sigma^u}{\partial x_i} \tilde{\partial}_u + \tilde{\partial}_i
$,
and its vertical differential is therefore
\begin{align*}
    d^\cV \sigma \Big(\frac{\partial}{\partial x_i} \Big) &= \Big( \frac{\partial \sigma^u}{\partial x_i}  + \eta_{is}\eta^{su}\Big)\tilde{\partial}_u.
\end{align*}
An easy computation shows that
\begin{align*}
    |d^\cV \sigma|^2 &= g^{ij} \eta_{uv}  \frac{\partial \sigma^u}{\partial x_i} \frac{\partial \sigma^v}{\partial x_j} + 2 g^{ij} \eta_{iv}  \frac{\partial \sigma^v}{\partial x_j}  +  g^{ij} \eta_{is}\eta_{ju}\eta^{su} .
\end{align*}
The vertical second fundamental form is
\begin{align*}
\alpha(X,Y) &= (\nabla^\cV  d^\cV \sigma)(X,Y)= \nabla_{X}^{\sigma^{-1}\cV} (d^\cV \sigma (Y)) - (d^\cV \sigma)(\nabla^{M}_{X} Y)
\end{align*}
and, writing $\frac{\partial}{\partial x_i}=\partial_i$, we compute henceforth $\alpha(\partial_i, \partial_j)$:
\begin{align*}
    &\alpha(\partial_i, \partial_j) 
    = \Big( \frac{\partial^2 \sigma^u}{\partial x_i\partial x_j} + \frac{\partial}{\partial x_i}(\eta_{js}\eta^{su})\Big) \tilde{\partial}_u + 
   \Big(\frac{\partial \sigma^u}{\partial x_j} + \eta_{js}\eta^{su}\Big) \nabla_{d\sigma(\partial_i)}^{\cV} \tilde{\partial}_u - \Gamma^{k}_{ij} \Big( \frac{\partial \sigma^u}{\partial x_k} + \eta_{ks}\eta^{su}\Big)\tilde{\partial}_u .
\end{align*}
But $\tilde{\Gamma}^{k}_{uv}=dy^k (\nabla_{\tilde{\partial}_v}\tilde{\partial}_u)=0$, because the fibres are totally geodesic, so
\begin{align*}
    \nabla_{d\sigma(\partial_i)}^{\cV} \tilde{\partial}_u &=
    \frac{\partial \sigma^v}{\partial x_i} \tilde{\Gamma}^{w}_{uv} \tilde{\partial}_w
      + \Big(\tilde{\Gamma}^{k}_{iu} \eta_{ks}\eta^{sw} + \tilde{\Gamma}^{w}_{iu}\Big) \tilde{\partial}_w
\end{align*}
since $\tilde{\partial}^{\cV}_k= \eta_{ks}\eta^{sw}\tilde{\partial}_w$.
Now, denoting by $\tilde{\partial}^{z}_i$  the horizontal part of $\tilde{\partial}_i$,
\begin{align}
    \frac{\partial}{\partial x_i}(\eta_{js}\eta^{su}\circ\sigma) &=
(\tilde{\partial}^{z}_i)(\eta_{js}\eta^{su}) + (\eta_{is}\eta^{sv} + \frac{\partial \sigma^v}{\partial x_i}) \tilde{\partial}_v  (\eta_{js}\eta^{su}). \label{eqa}
\end{align}
Similarly,
\begin{align*}
   \tilde{\Gamma}^{r}_{ij}
    &= dy^r (\nabla_{\tilde{\partial}_i}\tilde{\partial}_j) =  dy^r (\nabla_{\tilde{\partial}^{z}_i}\tilde{\partial}^{z}_j) + \tilde{\partial}^{z}_i (\eta_{js}\eta^{sr}) + \eta_{js}\eta^{sv}\tilde{\Gamma}^{r}_{iv} - \eta_{js}\eta^{sv}\eta_{iw}\eta^{wu}\tilde{\Gamma}^{r}_{uv} + \eta_{is}\eta^{su}\tilde{\Gamma}^{r}_{ju} .
\end{align*}
By O'Neill's connection formulas for Riemannian submersions~\cite{oneill}, and denoting also by $\cV$ (resp. $\cH$) the projection onto the vertical (resp. horizontal) distribution, we have
\begin{align*}
    dy^r (\cH\nabla_{\tilde{\partial}^{z}_i}\tilde{\partial}^{z}_j) &=
    - \Gamma^{k}_{ij} \eta_{ks}\eta^{sr}, \quad 
\cV\nabla_{\tilde{\partial}^{z}_i}\tilde{\partial}^{z}_j = \frac12 \cV[\tilde{\partial}^{z}_i, \tilde{\partial}^{z}_j], \end{align*}
hence
$ dy^r (\nabla_{\tilde{\partial}^{z}_i}\tilde{\partial}^{z}_j) =  \frac12 dy^r (\cV[\tilde{\partial}^{z}_i, \tilde{\partial}^{z}_j]) - \Gamma^{k}_{ij} \eta_{ks}\eta^{sr}$.

For the other terms, first observe that
\begin{align*}
    \nabla_{\tilde{\partial}_u}\tilde{\partial}_j &= 
     \nabla_{\tilde{\partial}_u}\tilde{\partial}^{z}_j + \tilde{\partial}_u(\eta_{jv}\eta^{vw})\tilde{\partial}_w + \eta_{jv}\eta^{vw}\nabla_{\tilde{\partial}_u}\tilde{\partial}_w
\end{align*}
and applying $dy^w$ we get
\begin{align*}
   \tilde{\Gamma}^{w}_{uj} &= dy^w (\nabla_{\tilde{\partial}_u}\tilde{\partial}_j) 
    = dy^w (\nabla_{\tilde{\partial}_u}\tilde{\partial}^{z}_j) + \tilde{\partial}_u(\eta_{jv}\eta^{vw}) + \eta_{jv}\eta^{vt}
    \tilde{\Gamma}^{w}_{tu} .
\end{align*}
Since the fibres are totally geodesic, O'Neill's tensor $T$ vanishes and $\cV \nabla_{\tilde{\partial}_u}\tilde{\partial}^{z}_j =0$. 
Therefore, as $\tilde{\partial}^{z}_j$ is basic, one has
$\nabla_{\tilde{\partial}_u}\tilde{\partial}^{z}_j =  \tilde{\Gamma}^{k}_{ju} \tilde{\partial}^{z}_k
$ 
and, as we already know that $dy^w (\tilde{\partial}^{z}_k) = - \eta_{kv}\eta^{vw}$, we obtain
\begin{align*}
dy^w (\nabla_{\tilde{\partial}_u}\tilde{\partial}^{z}_j) &= - \eta_{kv}\eta^{vw}\tilde{\Gamma}^{k}_{ju} .
\end{align*}

In conclusion, the vertical second fundamental form is
 \begin{align}
    \alpha(\partial_i, \partial_j)  =&\
\Big\{ \frac{\partial^2 \sigma^u}{\partial x_i\partial x_j} 
- \Gamma^{k}_{ij} \frac{\partial \sigma^u}{\partial x_k}
+ \tilde{\Gamma}^{u}_{vw} 
\frac{\partial \sigma^v}{\partial x_i} 
\frac{\partial \sigma^w}{\partial x_j} 
\label{vsf} \\
& +\frac{\partial \sigma^v}{\partial x_j}
\Big(\tilde{\Gamma}^{u}_{iv} + \eta_{ks}\eta^{su} \tilde{\Gamma}^{k}_{iv} \Big) 
+ \eta_{js}\eta^{sv} \eta_{kw}\eta^{wu}\tilde{\Gamma}^{k}_{iv} \notag \\
&+ \frac{\partial \sigma^v}{\partial x_i}
\Big(\tilde{\Gamma}^{u}_{jv} 
+ \eta_{ks}\eta^{su}\tilde{\Gamma}^{k}_{jv}\Big)
+ \eta_{is}\eta^{sv} \eta_{kw}\eta^{wu}\tilde{\Gamma}^{k}_{jv} \notag\\
&+ \Big(\tilde{\Gamma}^{u}_{ij} 
    -\frac12 dy^u (\cV[\tilde{\partial}^{z}_i, \tilde{\partial}^{z}_j]) \Big)
 \Big\}\tilde{\partial}_u \notag
\end{align}
with
$
-2  g_{kj} \eta^{uv}\tilde{\Gamma}^{k}_{iu} = dy^v (\cV [\tilde{\partial}^{z}_i , \tilde{\partial}^{z}_j]).
$

\begin{remark}
The vertical second fundamental form of a submersion, unlike its usual counterpart, is not symmetric. In some cases, the non-symmetric part of $\alpha$ can be expressed in terms of curvature, e.g. if $\cH$ is a $\mathrm{GL}(n)$-connection on a vector bundle, then it is equal to $\frac12 (\pi^* R)$, see \cite[9.53, 9.54 and 9.65]{Besse2008} and \cite[\textsection5]{oneill}. When $\pi: N \to M$ is a homogeneous fibre bundle, the skew-symmetric part of vertical second fundamental form is $\frac12 \sigma^* \Phi$, where $\Phi$ is the curvature two-form on $N$ 
\cite[Thm 3.5]{Wood2003}.
\end{remark}

Since $\pi$ is a Riemannian submersion, we have that $\tau^\cV (\sigma) = \tr{\alpha}$, so tracing the above expression yields the vertical tension field of $\sigma$:
\begin{align*}
    \tau^\cV (\sigma) =&\ g^{ij} \Big\{ \frac{\partial^2 \sigma^u}{\partial x_i\partial x_j} + \tilde{\Gamma}^{u}_{rs} \frac{\partial \sigma^r}{\partial x_i}\frac{\partial \sigma^s}{\partial x_j} - \Gamma^{k}_{ij}\frac{\partial \sigma^u}{\partial x_k} \\
    &+2 \Big( \tilde{\Gamma}^{u}_{ir} \frac{\partial \sigma^r}{\partial x_j} 
    + \tilde{\Gamma}^{k}_{ir} \Big(\frac{\partial \sigma^r}{\partial x_j} + 
    \eta_{js}\eta^{sr}\Big) \eta_{kw}\eta^{wu}\Big) + \tilde{\Gamma}^{u}_{ij} \Big\}\tilde{\partial}_{u},
\end{align*}
where $1\leq i,j,k\leq m$ and $m+1 \leq r,s,u,w\leq n$.

\begin{remark}
    To obtain the equation of the harmonic flow, work with the Cartesian product $M\times \bR$ and write $\frac{\partial \sigma}{\partial t}$ as $d\sigma(\partial_t)$. Then, since each $\sigma_t$ is a section and $\sigma^l$ is constant, 
\begin{align*}
   \frac{\partial \sigma}{\partial t} &= d\sigma(\partial_t)
   = \frac{\partial \sigma^u}{\partial t} \tilde{\partial}_u + \frac{\partial \sigma^l}{\partial t} \tilde{\partial}_l = \frac{\partial \sigma^u}{\partial t} \tilde{\partial}_u 
   = d^\cV \sigma(\partial_t).
\end{align*}
\end{remark}

\subsubsection{Weitzenb\"ock and Bochner formulas} \label{WnB}

We assume the coordinates $\{x_i\}$ are normal at the point $p$, so that $\partial_k (g^{ij})|_p = \frac{\partial g^{ij}}{\partial x_k}|_p =0$, and evaluate all quantities at $p$. Given a section $\sigma:M\to N$, we denote by $\sigma^{-1}\cV$ the pullback of the vertical distribution to the base manifold $M$.
By standard arguments, we have a Weitzenb\"ock formula for the vertical tension:
\begin{align*}
   \nabla_{\partial_k}^{\sigma^{-1}\cV} \tau^\cV  (\sigma) 
   =&\ \nabla_{\partial_k}^{\sigma^{-1}\cV} \Big( g^{ij} \alpha(\partial_i,\partial_j)\Big) \\
  =&\ g^{ij} \Big\{R^{\sigma^{-1}\cV} (\partial_k,\partial_i) (d^\cV \sigma(\partial_j)) 
 +  (d^\cV \sigma)(R^{M} (\partial_i,\partial_k)\partial_j) \\
 &+  \nabla_{\partial_i}^{\sigma^{-1}\cV} \nabla_{\partial_j}^{\sigma^{-1}\cV} (d^\cV \sigma(\partial_k)) 
 +  \nabla_{\partial_i}^{\sigma^{-1}\cV} \Big( \alpha(\partial_k,\partial_j) - \alpha(\partial_j,\partial_k)\Big) - 
 (d^\cV \sigma)(\nabla_{\partial_i}\nabla_{\partial_j}\partial_k)\Big\} .
\end{align*}
Recall from \eqref{vsf} that 
\begin{align*}
    \alpha(\partial_k,\partial_j) - \alpha(\partial_j,\partial_k) 
    &= 
    \cV([\tilde{\partial}^{z}_j, \tilde{\partial}^{z}_k]  ) = 2  g_{ij} \eta^{uv}\tilde{\Gamma}^{i}_{ku} \tilde{\partial}_{v},
\end{align*}
so that
$
\nabla_{\partial_i}^{\sigma^{-1}\cV} \Big( \alpha(\partial_k,\partial_j) - \alpha(\partial_j,\partial_k)\Big)
$
is a curvature term.

Let us now compute the Laplacian of the energy density:
\begin{align*}
   - \Delta e(\sigma)& = -g^{ij} \frac{\partial^2}{\partial x_i \partial x_j}\frac{|d^\cV \sigma|^2}{2} \\
    &= -\frac{1}{2}g^{ij} \Bigg[
    g^{kl}(\nabla^{2}_{\partial_i ,\partial_j} \eta)(d^\cV \sigma(\partial_k),d^\cV \sigma(\partial_l)) +
    4 g^{kl}(\nabla_{\partial_j} \eta)(\nabla^{\sigma^{-1}\cV}_{\partial_i}(d^\cV \sigma(\partial_k)),d^\cV \sigma(\partial_l)) \\
    & \quad + \frac{\partial^2 g^{kl}}{\partial x_i \partial x_j}\eta(d^\cV \sigma(\partial_k),d^\cV \sigma(\partial_l)) 
    \\
    & \quad + 2 g^{kl} 
    \eta(\nabla^{\sigma^{-1}\cV}_{\partial_i}(d^\cV \sigma(\partial_k)),\nabla^{\sigma^{-1}\cV}_{\partial_j}(d^\cV \sigma(\partial_l)))
    + 2 g^{kl} 
    \eta(\nabla^{\sigma^{-1}\cV}_{\partial_i}\nabla^{\sigma^{-1}\cV}_{\partial_j}(d^\cV \sigma(\partial_k)),d^\cV \sigma(\partial_l))
    \Bigg]
\end{align*}
The covariant derivative of the target metric $\eta$ by the pull-back connection is
\begin{align*}
    \nabla_{\partial_j} \eta 
    &=  \frac{\partial \sigma^u}{\partial x_j} \nabla_{\tilde{\partial}_u} \eta +  \nabla_{\tilde{\partial}_j} \eta
\end{align*}
and its second covariant derivatives are
\begin{align*}
    \nabla^{2}_{\partial_i , \partial_j} \eta 
    =&\ \frac{\partial \sigma^u}{\partial x_j} 
  \frac{\partial \sigma^v}{\partial x_i} \nabla^{2}_{\tilde{\partial}_v ,\tilde{\partial}_u} \eta
  + \frac{\partial \sigma^u}{\partial x_j}\nabla^{2}_{\tilde{\partial}_i ,\tilde{\partial}_u} \eta
  + \frac{\partial \sigma^u}{\partial x_i} \nabla^{2}_{\tilde{\partial}_u ,\tilde{\partial}_j} \eta
  +  \nabla^{2}_{\tilde{\partial}_i ,\tilde{\partial}_j} \eta \\
  &+ \Big(
  \frac{\partial^2 \sigma^w}{\partial x_i \partial x_j} +  \frac{\partial \sigma^u}{\partial x_j} 
  \frac{\partial \sigma^v}{\partial x_i} \tilde{\Gamma}^{w}_{uv} + \frac{\partial \sigma^u}{\partial x_j} \tilde{\Gamma}^{w}_{ui} + \frac{\partial \sigma^u}{\partial x_i}
  \tilde{\Gamma}^{w}_{uj} + \tilde{\Gamma}^{w}_{ij} 
  \Big)\nabla_{\tilde{\partial}_w} \eta \\
  &+ \Big(
  \frac{\partial \sigma^u}{\partial x_j}\tilde{\Gamma}^{k}_{ui} +
   \frac{\partial \sigma^u}{\partial x_i}\tilde{\Gamma}^{k}_{uj} + \tilde{\Gamma}^{k}_{ij}
  \Big)\nabla_{\tilde{\partial}_k} \eta .
\end{align*}
On the other hand
\begin{align*}
    &\nabla^{\sigma^{-1}\cV}_{\partial_i}(d^\cV \sigma(\partial_k)) =\nabla^{\sigma^{-1}\cV}_{\partial_i}\Big[ \Big(\frac{\partial \sigma^u}{\partial x_k}  + \eta_{ks}\eta^{su}\Big)\tilde{\partial}_u\Big] \\
    &= \Big[\frac{\partial^2 \sigma^u}{\partial x_i \partial x_k} + \frac{\partial}{\partial x_i}
    (\eta_{ks}\eta^{su}) +
    \Big( \frac{\partial \sigma^w}{\partial x_k}  + \eta_{ks}\eta^{sw}\Big) 
    \Big(
    \frac{\partial \sigma^v}{\partial x_i} \tilde{\Gamma}^{u}_{wv}
    + \tilde{\Gamma}^{u}_{wi}\Big)
    \Big]\tilde{\partial}_u +
    \Big( \frac{\partial \sigma^w}{\partial x_k}  + \eta_{ks}\eta^{sw}\Big)\tilde{\Gamma}^{k}_{wi} \tilde{\partial}_k.
\end{align*}
Therefore, using the Weitzenb\"ock formula and properties of the Levi-Civita connection yields a Bochner formula for $\sigma$:
\begin{align*}
     -\Delta e(\sigma) =&\
 -\frac{1}{2}g^{ij}g^{kl} \Big( \frac{\partial \sigma^t}{\partial x_k}  + \eta_{ks}\eta^{st}\Big)\Big( \frac{\partial \sigma^r}{\partial x_l}  + \eta_{ls}\eta^{sr}\Big)\Bigg[
    \frac{\partial \sigma^u}{\partial x_j} 
  \frac{\partial \sigma^v}{\partial x_i} \nabla^{2}_{\tilde{\partial}_v ,\tilde{\partial}_u} \eta
  + \frac{\partial \sigma^u}{\partial x_j}\nabla^{2}_{\tilde{\partial}_i ,\tilde{\partial}_u} \eta
  + \frac{\partial \sigma^u}{\partial x_i} \nabla^{2}_{\tilde{\partial}_u ,\tilde{\partial}_j} \eta \\
  &+  \nabla^{2}_{\tilde{\partial}_i ,\tilde{\partial}_j} \eta 
    \Bigg](\tilde{\partial}_t,\tilde{\partial}_r)
    \\
    & -\frac{1}{2}g^{ij}\frac{\partial^2 g^{kl}}{\partial x_i \partial x_j}
    \Big( \frac{\partial \sigma^t}{\partial x_k}  + \eta_{ks}\eta^{st}\Big)\Big( \frac{\partial \sigma^r}{\partial x_l}  + \eta_{ls}\eta^{sr}\Big)
    \eta_{tr} - |\nabla^\cV  d^\cV \sigma|^2\\
    &+ g^{ij} g^{kl} 
    \eta \Big(d^\cV \sigma(\partial_l) , 
    R^{\sigma^{-1}\cV} (\partial_k,\partial_i) (d^\cV \sigma(\partial_j)) 
 +  (d^\cV \sigma)(R^{M} (\partial_i,\partial_k)\partial_j) \\
 & +  \nabla_{\partial_i}^{\sigma^{-1}\cV} \Big( 2  g_{ij} \eta^{uv}\tilde{\Gamma}^{i}_{ku} \tilde{\partial}_{v}\Big) - 
\Big( \frac{\partial \Gamma^{m}_{jk}}{\partial x_i} +  \Gamma^{p}_{jk} \Gamma^{m}_{ip}\Big) d^\cV \sigma(\partial_m ) \Big) - g^{kl} \eta( \nabla_{\partial_k}^{\sigma^{-1}\cV} \tau^\cV  (\sigma),d^\cV \sigma(\partial_l)).
\end{align*}

Then, still under no assumption on $\sigma$ being harmonic, we obtain a Bochner-type inequality:

\begin{lemma}
Let $\sigma:M\to N$ be a section of a Riemannian submersion $\pi : (N,\eta) \to (M,g)$ between compact manifolds, with totally geodesic fibres, then there exists a universal constant  $C>0$ depending only on  $(M,g)$ and $(N,\eta)$ such that
\begin{align*}
    & - \Delta e(\sigma) \leq - |\nabla^\cV  d^\cV \sigma|^2 + C\Big( 
    e^2 (\sigma) + e(\sigma) + 1\Big) - (\nabla \tau^\cV  (\sigma),d^\cV \sigma) .
\end{align*}
\end{lemma}

\subsubsection{Shi-type estimates along \eqref{eq: HSF}}

We obtain global derivative estimates along the harmonic flow, with a technique used by Shi~\cite{Shi1989} for the Ricci flow, Lotay--Wei~\cite{lotay-wei-gafa} for the Laplacian flow of closed ${\rm G}_2$-structures and Dwivedi \emph{et al.}~\cite{dgk-isometric} for the isometric flow of ${\rm G}_2$-structures. This will later on allow us to establish long-time existence conditions of the harmonic $\S7$-flow.

\begin{proposition} 
\label{Prop:Shi-estimates}
    Let $\{\sigma_t\}_{t\in [0, t_1[}$ be a solution of the harmonic section flow \eqref{eq: HSF}, and assume that there exist $K \geq 1$ and constants $B_j$, $0\leq j\in\mathbb{Z}$, and $D\geq0$, such that 
$$
    |T| \leq K,\quad |\nabla g|,\ |\nabla \eta| \leq D K,
    \qandq
    |\nabla^j R^M|,\ |\nabla^j R^N| \leq B_j K^{j+2},\quad\forall\  j\geq 0.
$$
    Then, for all $m\in \mathbb{N}$, there exist constants $C_m=C_m(g, \eta)$ such that 
$$
    |\nabla^{m} T| \leq C_mK t^{-m/2},
    \quad\forall\  t\in [0, \frac{1}{K^4}].
$$
\end{proposition}

\begin{proof}
Locally the Weitzenb\"ock formula for $H$-structures becomes
\begin{align*}
   \nabla_{\partial_k}^{\sigma^{-1}\cV} \tau^\cV  (\sigma) 
   =&\ g^{ij} \Big\{R^{\sigma^{-1}\cV} (\partial_k,\partial_i) (d^\cV \sigma(\partial_j)) 
 +  (d^\cV \sigma)(R^{M} (\partial_i,\partial_k)\partial_j) +  \nabla_{\partial_i}^{\sigma^{-1}\cV} \nabla_{\partial_j}^{\sigma^{-1}\cV} (d^\cV \sigma(\partial_k))\\
 & + 2 \nabla_{\partial_i}^{\sigma^{-1}\cV} \Big( g_{lj} \eta^{uv}\tilde{\Gamma}^{l}_{ku} \tilde{\partial}_{v}\Big) - 
 (d^\cV \sigma)(\nabla_{\partial_i}\nabla_{\partial_j}\partial_k)\Big\} + \frac{\partial g^{ij}}{\partial x_k} \alpha(\partial_i,\partial_j) - g^{ij}(\nabla^\cV d^\cV  \sigma)(\partial_k,\nabla_{\partial_i}\partial_j) .
\end{align*}
Along the harmonic flow, we have
\begin{align*}
   \nabla_{\partial_k}^{\sigma^{-1}\cV} \frac{\partial \sigma}{\partial t} 
   =&\ g^{ij} \Big\{R^{\sigma^{-1}\cV} (\partial_k,\partial_i) (d^\cV \sigma(\partial_j)) 
 +  (d^\cV \sigma)(R^{M} (\partial_i,\partial_k)\partial_j) +  \nabla_{\partial_i}^{\sigma^{-1}\cV} \nabla_{\partial_j}^{\sigma^{-1}\cV} (d^\cV \sigma(\partial_k))\\
 & + 2 \nabla_{\partial_i}^{\sigma^{-1}\cV} \Big( g_{lj} \eta^{uv}\tilde{\Gamma}^{l}_{ku} \tilde{\partial}_{v}\Big) - 
 (d^\cV \sigma)(\nabla_{\partial_i}\nabla_{\partial_j}\partial_k)\Big\} + \frac{\partial g^{ij}}{\partial x_k} \alpha(\partial_i,\partial_j) - g^{ij}(\nabla^\cV d^\cV  \sigma)(\partial_k,\nabla_{\partial_i}\partial_j) .
\end{align*}
Since
\begin{align*}
       \nabla_{\partial_k}^{\sigma^{-1}\cV} \frac{\partial \sigma}{\partial t}
       &= \Big[(\nabla^\cV d^\cV  \sigma)(\partial_k ,\partial_t) -(\nabla^\cV d^\cV  \sigma)(\partial_t ,\partial_k)\Big] + 
        \frac{\partial}{\partial t} (d^\cV \sigma(\partial_k))
   \end{align*}
and
   \begin{align*}
   &\Big[(\nabla^\cV d^\cV  \sigma)(\partial_k ,\partial_t) -(\nabla^\cV d^\cV  \sigma)(\partial_t ,\partial_k)\Big] =
   \frac{\partial \sigma^u}{\partial t} \nabla^{N}_{\tilde{\partial}_k}\tilde{\partial}_u - \frac{\partial}{\partial t} \Big( \eta_{ks}\eta^{su} \Big) \tilde{\partial}_u - \eta_{ks}\eta^{su} \frac{\partial \sigma^v}{\partial t} \nabla^{N}_{\tilde{\partial}_v}\tilde{\partial}_u ,
   \end{align*}
the flow equation for $d^\cV  \sigma$ is
   \begin{align*}
   & \frac{\partial}{\partial t} (d^\cV \sigma(\partial_k)) = -\frac{\partial \sigma^u}{\partial t} \nabla^{N}_{\tilde{\partial}_k}\tilde{\partial}_u + \frac{\partial}{\partial t} \Big( \eta_{ks}\eta^{su} \Big) \tilde{\partial}_u + \eta_{ks}\eta^{su} \frac{\partial \sigma^v}{\partial t} \nabla^{N}_{\tilde{\partial}_v}\tilde{\partial}_u \\
     & + g^{ij} \Big\{R^{\sigma^{-1}\cV} (\partial_k,\partial_i) (d^\cV \sigma(\partial_j)) 
 +  (d^\cV \sigma)(R^{M} (\partial_i,\partial_k)\partial_j) +  \nabla_{\partial_i}^{\sigma^{-1}\cV} \nabla_{\partial_j}^{\sigma^{-1}\cV} (d^\cV \sigma(\partial_k))\\
 & + 2 \nabla_{\partial_i}^{\sigma^{-1}\cV} \Big( g_{lj} \eta^{uv}\tilde{\Gamma}^{l}_{ku} \tilde{\partial}_{v}\Big) - 
 (d^\cV \sigma)(\nabla_{\partial_i}\nabla_{\partial_j}\partial_k)\Big\} + \frac{\partial g^{ij}}{\partial x_k} \alpha(\partial_i,\partial_j) - g^{ij}(\nabla^\cV d^\cV  \sigma)(\partial_k,\nabla_{\partial_i}\partial_j)   .
   \end{align*}
   
 We will only indicate the main steps of the induction, using the Ricci identity \eqref{eq: riccischematic}.
   Write $T$ for $d^\cV  \sigma$ and use a schematic expression for the flow of $T$. In particular, note that the curvature term $R^{\sigma^{-1}\cV}$ of the pull-back connection will contribute instances of $T$.
\begin{align}
\label{eq3.1}
    \frac{\partial}{\partial t} T &= \nabla T * \nabla \eta + T^3 * R^N + T * R^M + \Delta T + \nabla T * \nabla g +  \nabla g*  \nabla \eta + T * R^N
\end{align}
so the schematic expression for the flow of $\nabla T$ is
\begin{align}
    \frac{\partial}{\partial t} \nabla T 
    =& \nabla^2 T * \nabla \eta + \nabla T *T* R^N + \nabla T * T^2 * R^N + T^4* \nabla R^N + \nabla T * R^M + T * \nabla R^M \\
    & + \Delta (\nabla T) + \nabla^2 T * \nabla g +  R^M * \nabla \eta + \nabla g * T * R^N + \nabla T * R^N + T^2 * \nabla R^N
\end{align}
since $\nabla ( R^N \circ \sigma) = T * \nabla R^N$, $\nabla (\Delta T) =  \Delta ( \nabla T) + \nabla T * R^N + T * \nabla R^M$ and $\nabla (\frac{\partial}{\partial t} T) = \frac{\partial}{\partial t} (\nabla T) + R^N * T$.
   
Therefore the evolution equation for $|T|^2$ is 
\begin{align*}
    \frac{\partial}{\partial t} |T|^2 
    =&\ 2 \langle \nabla T * \nabla \eta, T \rangle + 2 \langle T^3 * R^N, T \rangle + 2 \langle T * R^M, T \rangle + 2 \langle \Delta T,T \rangle + 2 \langle \nabla T * \nabla g , T \rangle \\
        & + 2 \langle \nabla g * \nabla \eta ,T \rangle + 2 \langle T * R^N, T \rangle ,
\end{align*}
so 
\begin{align*}
    \frac{\partial}{\partial t} |T|^2 
    \leq&\  \Delta |T|^2  - 2 |\nabla T|^2 + C |\nabla \eta| |T| |\nabla T| + C |T|^4 |R^N| + C |T|^2 |R^M| + C |\nabla g| |T| |\nabla T| \\
       &+ C |\nabla g||\nabla \eta| |T| + C  |R^N||T|^2
\end{align*}
and the evolution equation for $\nabla T$ implies
\begin{align*}
       \frac{\partial}{\partial t} |\nabla T|^2 \leq&  \Delta |\nabla T|^2  - 2 |\nabla^2 T|^2 
   + C |\nabla \eta| |\nabla T | |\nabla^2 T | + C |R^N||T|| \nabla T|^2 + C |R^N | | T|^2 |\nabla T|^2   \\
   & + C |\nabla R^N| | T|^4 |\nabla T| + C |R^M| |\nabla T|^2 + C |\nabla R^M| |T | |\nabla T| \\
   & + C |\nabla g||\nabla T ||\nabla^2 T| + C| R^M||\nabla \eta||\nabla T| + C|\nabla g||R^N||T| |\nabla T| 
   +  C | R^N| |\nabla T |^2 \\
   & + C |\nabla R^N| | T|^2 |\nabla T | .
\end{align*}
The hypotheses imply that
\begin{align*}
    \frac{\partial}{\partial t} |T|^2 
    &\leq  \Delta |T|^2  - 2 |\nabla T|^2 + C K |T| |\nabla T| +  C K^2 |T|^4 
       + C K^2 |T|^2 + C K^2  |T|
\end{align*}
and
\begin{align*}
    \frac{\partial}{\partial t} |\nabla T|^2 \leq&  \Delta |\nabla T|^2  - 2 |\nabla^2 T|^2 
    + C K|\nabla T | |\nabla^2 T | + C K^2 |T|| \nabla T|^2 + C K^2 | T|^2 |\nabla T|^2   \\
    & + C K^3 | T|^4 |\nabla T| + C K^2 |\nabla T|^2 + C K^3 |T | |\nabla T|  +  C K^3 |\nabla T | + C K^3 |T|^2 |\nabla T | .
\end{align*}
   
   \underline{Base step}: Let $f = t | \nabla T|^2 + \beta |T|^2$, for a constant $\beta$, then
   \begin{align*}
        \frac{\partial f}{\partial t} =& | \nabla T|^2 + t \frac{\partial | \nabla T|^2 }{\partial t} + \beta
\frac{\partial | T|^2 }{\partial t} \\
 \leq& | \nabla T|^2 + t \Big( 
\Delta |\nabla T|^2  - 2 |\nabla^2 T|^2 
   + C K|\nabla T | |\nabla^2 T | + C K^2 |T|| \nabla T|^2 + C K^2 | T|^2 |\nabla T|^2   \\
   & + C K^3 | T|^4 |\nabla T| + C K^2 |\nabla T|^2 + C K^3 |T | |\nabla T|  +  C K^3 |\nabla T | + C K^3 |T|^2 |\nabla T |
\Big) \\
&+ \beta \Big(
\Delta |T|^2  - 2 |\nabla T|^2 + C K^2 |\nabla T| + C K^6
\Big).
\end{align*}
By Young's inequality and $|T|\leq K$,
   \begin{align*}
        \frac{\partial f}{\partial t}  \leq & | \nabla T|^2 + \Delta f - (2 -\varepsilon) t |\nabla^2 T|^2 
   + t\Big( C K^4 |\nabla T |^2 + C K^{10}  \Big) \\
&+ \beta \Big( - (2 -\varepsilon) |\nabla T|^2 + C K^4 + C K^6 \Big) \\
\leq & \Delta f + (C- \beta(2 -\varepsilon)) |\nabla T|^2 + C \beta K^6 ,
   \end{align*}
   so, for $\beta$ large enough,
   $$
   \frac{\partial f}{\partial t}  \leq \Delta f + C \beta K^6 .
$$
Since $f(x,0)\leq \beta K^2$, by the maximum principle and the assumption $0\leq t\leq 1/K^4$,
\begin{align*}
    \sup f(x,t) 
    &\leq \beta K^2 + C\beta t K^6 \leq c K^2.
\end{align*}
Therefore $t |\nabla T|^2 \leq C K^2$, that is $|\nabla T| \leq C K t^{-1/2}$, and this proves the base step.

\underline{Induction step}: From \eqref{eq3.1}, we obtain the schematic expression for the flow of $\nabla^m T$ based on
\begin{equation*}
   \begin{aligned}[l]
       a) & \nabla^m \frac{\partial}{\partial t} T =  \frac{\partial}{\partial t} \nabla^m T + \nabla^{m-1} (T* R^N)\\
       b) & \nabla^m(\Delta T) = \Delta (\nabla^m T) + \sum_{i=0}^{m} \nabla^{m -i} T * \nabla^i R^M \\
       c) & \nabla^m(T * R^N) = \sum_{a+b=m} \nabla^{a} T * \nabla^{b} (R^N \circ\sigma) \\
       d) & \nabla^m(T^3 * R^N) = \sum_{i=0}^{m} \nabla^{i} (T^3) * \nabla^{m-i} (R^N \circ\sigma) \\
       h) & \nabla^m(\nabla T * \nabla \eta) = \nabla^{m+1}  T * \nabla \eta + \sum_{i=0}^{m-1} \nabla^{i+1} T * \nabla^{m-i}(\nabla \eta\circ\sigma)
       \end{aligned}
       \hspace{-.5in}
       \begin{aligned}[l]
       e) & \nabla^m(T * R^M) = \sum_{a+b=m} \nabla^{a} T * \nabla^{b} R^M \\
       f) & \nabla^m(\nabla T * \nabla g) = \sum_{a+b=m} \nabla^{a+1} T * \nabla^{b+1} g \\
       g) & \nabla^m(\nabla g*  \nabla \eta) = \sum_{a+b=m} \nabla^{a+1} g * \nabla^{b} (\nabla \eta \circ\sigma) 
       \end{aligned}
   \end{equation*}
and the Fa{\`a}-di Bruno formula~\cite{FdB}:
$$
\nabla^{b} (R^N \circ\sigma) = \sum_{\sum_{j=1}^{n} jb_j = b}
\Big(\nabla^{b_1 + \dots + b_j} R^N\Big) * 
\prod_{j=1}^{n} \Big( \nabla^{j-1} T\Big)^{b_j}
$$
where $\prod$ is a product of $*$'s. Then 
\begin{align*}
    \frac{\partial}{\partial t} (|\nabla^m T|^2) =& 
     2 \langle \nabla^m(\nabla T * \nabla \eta) , \nabla^m T\rangle
    + 2 \langle \nabla^m(T^3 * R^N), \nabla^m T\rangle
    + 2 \langle \nabla^m(T * R^M), \nabla^m T\rangle\\
&    + 2 \langle \nabla^m(\Delta T), \nabla^m T\rangle
    + 2 \langle \nabla^m(\nabla T * \nabla g), \nabla^m T\rangle
    + 2 \langle \nabla^m(\nabla g*  \nabla \eta), \nabla^m T\rangle \\
&    + 2 \langle \nabla^m(T * R^N), \nabla^m T\rangle
    - 2 \langle \nabla^{m-1} (T* R^N), \nabla^m T\rangle .
\end{align*}
A combination of the chain rule, the above formulas and the hypotheses yields
\begin{align*}
    |\nabla^{b} (R^N \circ\sigma)| &\leq C K^2 t^{-b/2} ;\\
    2 \langle \nabla^m(\Delta T), \nabla^m T\rangle 
    &\leq \Delta |\nabla^m T|^2 -2|\nabla^{m+1} T|^2 + C K^2 |\nabla^m T|^2  + C  K^{3} t^{-m/2}  |\nabla^m T| ;\\
   2 \langle \nabla^m(T * R^M), \nabla^m T\rangle  &\leq C K^2 |\nabla^m T|^2 + C K^3  t^{-m/2} |\nabla^m T| ;\\
    2 \langle \nabla^m(T * R^N), \nabla^m T\rangle 
     &\leq CK^2 |\nabla^{m} T|^2 + C K^3 t^{-m/2} |\nabla^m T| ;\\
   2 \langle \nabla^m(\nabla T * \nabla g), \nabla^m T\rangle 
&\leq CK |\nabla^{m+1} T||\nabla^{m} T| + CK^2 |\nabla^{m} T|^2 + CK^2 t^{-(m+1)/2} |\nabla^{m} T| ;\\
    2 \langle \nabla^m(T^3 * R^N), \nabla^m T\rangle 
    &\leq CK^4 |\nabla^{m} T|^2 + C K^5  t^{-m/2}|\nabla^m T| ;\\
    2 \langle \nabla^m(\nabla T * \nabla \eta) , \nabla^m T\rangle 
     &\leq CK |\nabla^{m}  T| |\nabla^{m+1}  T| + C K^2 |\nabla^{m}  T|^2 + CK^3t^{-m/2} |\nabla^{m}  T| ;\\
    2 \langle \nabla^m(\nabla g *  \nabla \eta), \nabla^m T\rangle 
    &\leq C K^{3} t^{-m/2}|\nabla^{m}  T| + CK^3 t^{-(m-1)/2} |\nabla^{m}  T| ;\\
    2 \langle \nabla^{m-1} (T* R^N), \nabla^m T\rangle 
    &\leq C K t^{-m/2} |\nabla^m T| .
\end{align*}
Therefore
\begin{align*}
    \frac{\partial}{\partial t} (|\nabla^m T|^2) 
    &\leq \Delta |\nabla^m T|^2  - 2 |\nabla^{m+1} T|^2 + C K |\nabla^m T| |\nabla^{m+1} T| + CK^2(1+K^2) |\nabla^m T|^2\\
    & + C K^2 (K + K^3) t^{-m/2} |\nabla^{m} T|  + CK^2 t^{-(m+1)/2} |\nabla^{m} T| + CK^3 t^{-(m-1)/2} |\nabla^{m} T|.
\end{align*}
Now, by Young's inequality, for $\varepsilon$ small enough, and using $K\geq 1$ and $t\leq 1$, we have
\begin{align*}
    \frac{\partial}{\partial t} (|\nabla^m T|^2) 
    \leq& \Delta |\nabla^m T|^2  - |\nabla^{m+1} T|^2 + CK^2 |\nabla^m T|^2 + C K^2 t^{-1/2} |\nabla^m T|^2 + C K^5 t^{-m/2} |\nabla^{m} T| \\
    &+ CK^2 t^{-(m+1)/2} |\nabla^{m} T| 
\end{align*}
This remains true replacing $k$ by $m-k$ so that
\begin{align*}
    \frac{\partial}{\partial t} (|\nabla^{m-k} T|^2) 
    \leq& \Delta |\nabla^{m-k} T|^2  - |\nabla^{m-k+1} T|^2 + CK^4 t^{-(m-k)} + CK^4 t^{-1/2} t^{-(m-k)} \\
    & + C K^6 t^{-(m-k)/2} t^{-(m-k)/2} + CK^3 t^{-(m-k+1)/2} t^{-(m-k)/2} \\
    \leq& \Delta |\nabla^{m-k} T|^2  - |\nabla^{m-k+1} T|^2 + CK^4 t^{-(m-k)} + CK^4 t^{-1/2} t^{-(m-k)} .
\end{align*}
Let $f_m = t^m |\nabla^{m} T|^2 + \beta_m \sum_k \alpha^m_k t^{m-k} |\nabla^{m-k} T|^2$.

Combining the previous estimates and Young's inequality again,
 \begin{align*}
    \frac{\partial f_m}{\partial t} \leq& \Delta f_m + \Big( CK^2 t^m +CK^6 t^m + CK^2 t^{(2m-1)/2} + (C+m) t^{m-1} - \beta_m \alpha^m_1 t^{m-1} \Big) |\nabla^m T|^2 \\
    & + \beta_m \sum_{k=1}^{m-1} ((m-k)\alpha^m_k - \alpha^m_{k+1})   t^{m-k-1} |\nabla^{m-k} T|^2 \\
& + C\beta_m \sum_{k=1}^{m} \alpha^m_k (K^4 + K^4 t^{-1/2}) +CK^4 .
\end{align*}
But $(m-k)\alpha^m_k - \alpha^m_{k+1} =0$, for $1\leq k \leq m-1$, and for $\beta_m$ large enough, we have
\begin{align*}
    \frac{\partial f_m}{\partial t} &\leq \Delta f_m  + C K^4 t^{-1/2} +CK^4 .
\end{align*}
 Since $f_m (0) = \beta_m \alpha^m_{m} |T|^2 \leq \beta_m \alpha^m_{m} K^2$, we invoke the maximum principle to conclude that
 \begin{align*}
     \sup f_m &\leq \beta_m \alpha^m_{m} K^2 + C K^4 t^{1/2} +CK^4 t \\
     &\leq CK^2.
 \end{align*}
    Hence $t^m |\nabla^{m} T|^2 \leq CK^2$, that is the $m$-induction step $|\nabla^{m} T| \leq CK t^{-m/2}$, provided $0\leq t\leq 1/K^4$. 
 \end{proof}
 
 \begin{remark}
As described in \textsection\ref{sec: vertical 2ndFF}, $H$-structures are in one-to-one correspondence with sections of a Riemannian homogeneous fibre bundle, so  Proposition~\ref{Prop:Shi-estimates} applies in particular to $H$-structures.
\end{remark}

\section{Preliminaries on $\S7$-structures}
\label{sec:prelims}

\subsection{Definitions and basic properties}

In this section, we review the notion of a $\S7$-structure on an $8$-dimensional manifold $M$ and the associated decomposition of differential forms. We also discuss the torsion tensor of a $\S7$-structure. More details can be found in \cite[Chapter 10]{joycebook} and \cite{karigiannis-spin7}. In particular, we closely follow \cite{karigiannis-spin7} in the discussion below.

A $\S7$-structure on $M^8$ is a reduction of the structure group of the frame bundle $\Fr(M)$  to the Lie group $\S7\subset \SO(8)$. From the point of view of differential geometry, a $\S7$-structure on $M$ is a $4$-form $\s7$ on $M$. The existence of such a structure is a \emph{topological condition}. The space of $4$-forms which determine a $\S7$-structure on $M$ is a subbundle $\cA$ of $\Omega^4(M)$, called the bundle of \emph{admissible} $4$-forms. This is \emph{not} a vector subbundle and it is not even an open subbundle, unlike the case for $\G2$-structures.

A $\S7$-structure $\s7$  determines a Riemannian metric and an orientation on $M$ in a nonlinear way. Let $p\in M$ and extend  a non-zero tangent vector $v\in T_pM$  to a local frame $\{v, e_1, \cdots , e_7\}$ near $p$. Define
\begin{align*}
    B_{ij}(v)&=((e_i\lrcorner v\lrcorner \s7)\wedge (e_j\lrcorner v\lrcorner \s7)\wedge (v\lrcorner \s7))(e_1, \cdots , e_7),\\
    A(v)&=((v\lrcorner \s7)\wedge \s7)(e_1, \cdots, e_7).
\end{align*}
Then the metric induced by $\s7$ is given by
\begin{align}
\label{eq:metric}
    (g_{\s7}(v,v))^2=-\frac{7^3}{6^{\frac{7}{3}}}\frac{(\textup{det}\ B_{ij}(v))^{\frac 13}}{A(v)^3}.    
\end{align}
The metric and the orientation determine a Hodge star operator $\star$, and the $4$-form is \emph{self-dual}, i.e., $\star \s7=\s7$. 

\begin{definition}
    Let $\del$ be the Levi-Civita connection of the metric $g_{\s7}$. The pair $(M^8, \s7)$ is a \emph{$\S7$-manifold} if $\del \s7=0$. This is a non-linear partial differential equation for $\s7$, since $\del$ depends on $g$, which in turn depends non-linearly on $\s7$. A $\S7$-manifold has Riemannian holonomy contained in the subgroup $\S7\subset \SO(8)$. Such a parallel $\S7$-structure is also called \emph{torsion-free}. 
\end{definition}

\subsection{Decomposition of the space of forms}
\label{subsec:formdecomp}

The existence of a $\S7$-structure $\s7$ induces a decomposition of the space of differential forms on $M$  into irreducible $\S7$ representations. We have the following orthogonal decomposition, with respect to $g_\s7$:
\begin{align*}
\Omega^2=\Omega^2_{7}\oplus \Omega^2_{21},\ \ \ \ \ \ \ \Omega^3=\Omega^3_{8}\oplus \Omega^3_{48}, \ \ \ \ \ \ \ \ \  \Omega^4=\Omega^4_{1}\oplus \Omega^4_{7}\oplus \Omega^4_{27}\oplus \Omega^4_{35},
\end{align*}
where $\Omega^k_l$ has pointwise dimension $l$. Explicitly, $\Omega^2$ and $\Omega^3$ are described as follows:
\begin{align}
    \Omega^2_7&=\{\beta \in \Omega^2 \mid \star(\s7 \wedge \beta)=-3\beta\}, \label{eq:27decomp1}\\
    \Omega^2_{21}&=\{ \beta\in \Omega^2 \mid \star(\s7\wedge \beta)=\beta\}, \label{eq:221decomp1}
\end{align}
and
\begin{align}
    \Omega^3_8&=\{ X\lrcorner \s7 \mid X\in \Gamma(TM)\}, \label{eq:38decomp1}\\
    \Omega^3_{48}&=\{ \gamma \in \Omega^3\mid \gamma \wedge \s7 =0\}. \label{eq:348decomp1}
\end{align}
For our computations, it is useful to describe the spaces of forms in local coordinates. For $\beta\in \Omega^2(M)$,
\begin{align}
    \beta_{ij}\in \Omega^2_7 \iff \beta_{ab}\s7_{abij}&=-6\beta_{ij},\label{eq:27decomp2}\\
    \beta_{ij}\in \Omega^2_{21} \iff \beta_{ab}\s7_{abij}&=2\beta_{ij}.\label{eq:221decomp2}
\end{align}
Similarly, for $\gamma\in \Omega^3(M)$,
\begin{align}
    \gamma_{ijk}\in \Omega^3_8 &\iff \gamma_{ijk}=X_l\s7_{ijkl}\ \ \ \ \textup{for\ some}\ X\in \Gamma(TM), \label{eq:38decomp2}\\
    \gamma_{ijk}\in \Omega^3_{48} &\iff \gamma_{ijk}\s7_{ijkl}=0. \label{eq:348decomp2}
\end{align}
If $\pi_7$ and $\pi_{21}$ are the projection operators on $\Omega^2$, it follows from \eqref{eq:27decomp2} and \eqref{eq:221decomp2} that 
\begin{align}
\pi_7(\beta)_{ij}&=\frac 14\beta_{ij}-\frac 18\beta_{ab}\s7_{abij}, \label{eq:pi7}\\
\pi_{21}(\beta)_{ij}&=\frac 34\beta_{ij}+\frac 18\beta_{ab}\s7_{abij}. \label{eq:pi21}
\end{align}
Finally, for $\beta_{ij}\in \Omega^2_{21}$,
\begin{align}
    \beta_{ab}\s7_{bpqr}&=\beta_{pi}\s7_{iqra}+\beta_{qi}\s7_{irpa}+\beta_{ri}\s7_{ipqa}, \label{eq:221prop}   
\end{align}
which can be used to show that $\Omega^2_{21}\equiv \mathfrak{so}(7)$ is the Lie algebra of $\S7$, with the commutator of matrices
\begin{align*}
    [\alpha, \beta]_{ij}=\alpha_{il}\beta_{lj}-\alpha_{jl}\beta_{li}.
\end{align*}

To describe $\Omega^4$ in local coordinates, we use the $\diamond$ operator described in Definition~\ref{def: diamond}, as in \cite{dgk-isometric} for the $\G2$ case. Given $A\in \Gamma(T^*M\otimes TM)$, define
\begin{align}
\label{eq:diadefn1}
    A\diamond \s7= \frac{1}{24}(A_{ip}\s7_{pjkl}+A_{jp}\s7_{ipkl}+A_{kp}\s7_{ijpl}+A_{lp}\s7_{ijkp})dx^i\wedge dx^j\wedge dx^k\wedge dx^l,    
\end{align}
and hence 
\begin{align}
\label{eq:diadefn2}
    (A\diamond \s7)_{ijkl}=  A_{ip}\s7_{pjkl}+A_{jp}\s7_{ipkl}+A_{kp}\s7_{ijpl}+A_{lp}\s7_{ijkp}. 
\end{align}
Recall that $
    \Gamma(T^*M\otimes TM)=\Omega^0\oplus S_0\oplus \Omega^2$,
and $\Omega^2$ further splits orthogonally into \eqref{eq:27decomp1} and \eqref{eq:221decomp1}, so
\begin{align}
\label{eq:splitting TM* x TM}
 \Gamma(T^*M\otimes TM)=\Omega^0\oplus S_0 \oplus \Omega^2_7\oplus \Omega^2_{21}.   
\end{align}
With respect to this splitting, we can write $A=\frac 18 (\tr A)g+A_0+A_7+A_{21}$ where $A_0$ is a symmetric traceless $2$-tensor. 
The diamond contraction \eqref{eq:diadefn2} defines a linear map $A\mapsto A\diamond \s7$, from $\Omega^0\oplus S_0 \oplus \Omega^2_7\oplus \Omega^2_{21}$ to $\Omega^4(M)$. 
The following proposition is proved  in \cite[Prop. 2.3]{karigiannis-spin7}.

\begin{proposition}\label{prop:diaproperties1}
The kernel of the map $A\mapsto A\diamond \s7$ is isomorphic to the subspace $\Omega^2_{21}$. The remaining three summands $\Omega^0,\ S_0$ and $\Omega^2_7$ are mapped isomorphically onto the subspaces $\Omega^4_1,\ \Omega^4_{35}$ and $\Omega^4_7$ respectively.
\end{proposition}

We also record the following properties of the $\diamond$ operator.

\begin{proposition}\label{prop:diaproperties2}
Let $(M, \s7)$ be a manifold with a $\S7$-structure. Then
\begin{enumerate}
\item The Hodge star of $A\diamond \s7$ is
\begin{align}\label{diapropertieseqn1}
\star(A\diamond \s7)&=(\Bar{A}\diamond \s7)\ \  \textup{where}\ \  \Bar{A}=\frac 14 (\tr A)g_{ij}-A_{ji}.    
\end{align}

\item If $A, B\in \Gamma(T^*M\otimes TM)$ and $A_0,\  B_0$ denote the traceless symmetric parts of $A$ and $B$ respectively and $A_7,\ B_7$ denote their $\Omega^2_7$ component. Then
\begin{align}\label{diapropertieseqn2}
g(A\diamond \s7, B\diamond \s7)&=\frac 72(\tr A)(\tr B)+4\tr (A_0B_0) -16\tr(A_7B_7) 
\end{align}
where $A_0B_0$ and $A_7B_7$ are matrix multiplications.
\end{enumerate}
\end{proposition}

To understand $\Omega^4_{27}$, we need another characterization of the space of $4$-forms using the $\S7$-structure. Following \cite{karigiannis-spin7}, we adopt the following:

\begin{definition}
    On $(M, \s7)$, define a $\S7$-equivariant linear operator $\Lambda_{\s7}$ on $\Omega^4$ as follows. Let $\sigma\in \Omega^4(M)$ and let $(\sigma \cdot \s7)_{ijkl}=\sigma_{ijmn}\s7_{mnkl}$. Then 
\begin{align}
\label{eq:lambdamapdefn}
    (\Lambda_{\s7}(\sigma))_{ijkl}=(\sigma \cdot \s7)_{ijkl}+(\sigma \cdot \s7)_{iklj}+(\sigma \cdot \s7)_{iljk}+(\sigma \cdot \s7)_{jkil}+(\sigma \cdot \s7)_{jlki}+(\sigma \cdot \s7)_{klij}. 
\end{align}
\end{definition}

\begin{proposition}
\label{prop:427decomp}
    The spaces $\Omega^4_1$, $\Omega^4_7,\ \Omega^4_{27}$ and $\Omega^4_{35}$ are all eigenspaces of $\Lambda_{\s7}$ with distinct eigenvalues:
\begin{align}
\label{eq:4formes}
  \begin{aligned}
   \Omega^4_1&=\{\sigma\in \Omega^4 \mid \Lambda_{\s7}(\sigma)=-24\sigma\}, \\       \Omega^4_{27}&=\{\sigma\in \Omega^4 \mid \Lambda_{\s7}(\sigma)=4\sigma\},
  \end{aligned}
  &&
  \begin{aligned}
   \Omega^4_7&=\{\sigma\in \Omega^4 \mid \Lambda_{\s7}(\sigma)=-12\sigma\}, \\       \Omega^4_{35}&=\{\sigma\in \Omega^4 \mid \Lambda_{\s7}(\sigma)=0\}.
  \end{aligned}
 \end{align}
 Moreover, the decomposition of $\Omega^4(M)$ into self-dual and anti-self-dual parts is
 \begin{align}\label{eq:sd&asd}
\Omega^4_+=\{\sigma \in \Omega^4\mid \star \sigma = \sigma\}=\Omega^4_1\oplus \Omega^4_7\oplus \Omega^4_{27},\ \ \ \ \ \Omega^4_-=\{\sigma\in \Omega^4\mid \star \sigma=-\sigma\}=\Omega^4_{35} .    
 \end{align}
\end{proposition}

\medskip

Before we discuss the torsion of a $\S7$-structure, we note some contraction identities involving the $4$-form $\s7$. In local coordinates $\{x^1, \cdots, x^8\}$, the $4$-form $\s7$ is
\begin{align*}
\s7=\frac{1}{24}\s7_{ijkl}\ dx^i\wedge dx^j\wedge dx^k\wedge dx^l    
\end{align*}
where $\s7_{ijkl}$ is totally skew-symmetric. We have the following identities
\begin{align}
    \s7_{ijkl}\s7_{abcl}
    &=g_{ia}g_{jb}g_{kc}+g_{ib}g_{jc}g_{ka}+g_{ic}g_{ja}g_{kb}\nonumber \\
    & \quad -g_{ia}g_{jc}g_{kb}-g_{ib}g_{ja}g_{kc}-g_{ic}g_{jb}g_{ka} \nonumber \\
    & \quad -g_{ia}\s7_{jkbc}-g_{ib}\s7_{jkca}-g_{ic}\s7_{jkab} \nonumber \\
    & \quad -g_{ja}\s7_{kibc}-g_{jb}\s7_{kica}-g_{jc}\s7_{kiab} \nonumber \\
    & \quad -g_{ka}\s7_{ijbc}-g_{kb}\s7_{ijca}-g_{kc}\s7_{ijab} \label{eq:impiden1} \\
    \s7_{ijkl}\s7_{abkl}&=6g_{ia}g_{jb}-6g_{ib}g_{ja}-4\s7_{ijab} \label{eq:impiden2} \\
    \s7_{ijkl}\s7_{ajkl}&=42g_{ia} \label{eq:impiden3} \\
    \s7_{ijkl}\s7_{ijkl}&=336 . \label{eq:impiden4}
\end{align}
We also have contraction identities involving $\del \s7$ and $\s7$
\begin{align}
(\del_m\s7_{ijkl})\s7_{abkl}&=-\s7_{ijkl}(\del_m\s7_{abkl})-4\del_m\s7_{ijab} \label{eq:impiden5}\\
(\del_m\s7_{ijkl})\s7_{ajkl}&=-\s7_{ijkl}(\del_m\s7_{ajkl}) \label{eq:impiden6} \\
(\del_m\s7_{ijkl})\s7_{ijkl}&=0. \label{eq:impiden7}
\end{align}

We now describe the \emph{torsion} of a $\S7$-structure. Given $X\in \Gamma(TM)$, we know from \cite[Lemma 2.10]{karigiannis-spin7} that $\del_X\s7$ lies in the subbundle $\Omega^4_7\subset\Omega^4$. 
\begin{definition}
    The \emph{torsion tensor} of a $\S7$-structure $\s7$ is the element of $\Omega^1_8\otimes \Omega^2_7$ defined by expressing $\del \s7$ in the light of Proposition \ref{prop:diaproperties1}:
\begin{align}
\label{Tdefneqn}
    \del_m\s7_{ijkl}=(T_m\diamond \s7)_{ijkl}=T_{m;ip}\s7_{pjkl}+T_{m;jp}\s7_{ipkl}+T_{m;kp}\s7_{ijpl}+T_{m;lp}\s7_{ijkp}    
\end{align}
    where $T_{m;ab}\in\Omega^2_7$, for each fixed $m$. 
\end{definition}

Directly in terms of $\del \s7$, the torsion $T$ is given by
\begin{align}
\label{eq:Texpress}
    T_{m;ab} 
    =\frac{1}{96}(\del_m\s7_{ajkl})\s7_{bjkl}    
\end{align}

\begin{theorem}
\cite{fernandez-spin7}
    The $\S7$-structure $\s7$ is torsion-free if, and only if, $d\s7=0$. Since $\star \s7=\s7$, this is equivalent to $d^*\s7=0$.
\end{theorem}

Finally, the torsion  satisfies a `Bianchi-type identity'. This was first proved in \cite[Theorem 4.2]{karigiannis-spin7}, using the diffeomorphism invariance of the torsion tensor. We give a different proof below, using the Ricci identity \eqref{eq: riccischematic}.

\begin{theorem}\label{thm:spin7bianchi}
The torsion tensor $T$ satisfies the following `Bianchi-type identity'
\begin{align}
\label{spin7bianchi}
    \del_iT_{j;ab}-\del_jT_{i;ab}=2T_{i;am}T_{j;mb}-2T_{j;am}T_{i;mb}+\frac 14R_{jiab}-\frac 18R_{jimn}\s7_{mnab}.  
\end{align}
\end{theorem}
\begin{proof}
As remarked in \cite{karigiannis-spin7}, relation \eqref{spin7bianchi} can be proved by using the Ricci identity \eqref{eq: riccischematic} and the contraction identities \eqref{eq:impiden1}-\eqref{eq:impiden4}. We provide the details for completeness.
From \eqref{eq:Texpress}, we have
\begin{align*}
    \del_nT_{m;ab}=\frac{1}{96}(\del_n\del_m\s7_{ajkl})\s7_{bjkl}+\frac{1}{96}\del_m\s7_{ajkl}\del_n\s7_{bjkl}    
\end{align*}
which gives
\begin{align}
    \del_nT_{m;ab}-\del_mT_{n;ab}
    &=\frac{1}{96}(\del_n\del_m-\del_m\del_n)\s7_{ajkl}\s7_{bjkl}+\frac{1}{96}(\del_m\s7_{ajkl}\del_n\s7_{bjkl}-\del_n\s7_{ajkl}\del_m\s7_{bjkl})\nonumber \\
    &=\frac{1}{96}(-R_{nmas}\s7_{sjkl}-R_{nmjs}\s7_{askl}-R_{nmks}\s7_{ajsl}-R_{nmls}\s7_{ajks})\s7_{bjkl} \nonumber \\
    & \quad +\frac{1}{96}(T_{m;ap}\s7_{pjkl}+T_{m;jp}\s7_{apkl}+T_{m;kp}\s7_{ajpl}+T_{m;lp}\s7_{ajkp}) \nonumber \\
    & \qquad \qquad  (T_{n;bs}\s7_{sjkl}+T_{n;js}\s7_{bskl}+T_{n;ks}\s7_{bjsl}+T_{n;ls}\s7_{bjks}) \nonumber \\
& \quad +\frac{1}{96}(T_{n;ap}\s7_{pjkl}+T_{n;jp}\s7_{apkl}+T_{n;kp}\s7_{ajpl}+T_{n;lp}\s7_{ajkp}) \nonumber \\
& \qquad \qquad  (T_{m;bs}\s7_{sjkl}+T_{m;js}\s7_{bskl}+T_{m;ks}\s7_{bjsl}+T_{m;ls}\s7_{bjks}) \label{spinbpf1}
\end{align}
where we used the Ricci identity \eqref{eq: riccischematic} on the first line of the second equality and \eqref{Tdefneqn} to write the remaining terms of the second equality. Now using the contraction identities \eqref{eq:impiden1}-\eqref{eq:impiden4}, an easy albeit tedious computation gives \eqref{spin7bianchi}.
\end{proof}

Using the Riemannian Bianchi identity, we see  that
\begin{align*}
R_{ijkl}\s7_{ajkl}=-(R_{jkil}+R_{kijl})\s7_{ajkl}=-R_{iljk}\s7_{aljk}-R_{ikjl}\s7_{akjl},    
\end{align*}
hence 
\begin{align}
\label{rmspin7}
    R_{ijkl}\s7_{ajkl}=0.    
\end{align}
Using this and contracting \eqref{spin7bianchi} on $j$ and $b$ gives the expression for the Ricci curvature of a metric induced by a $\S7$-structure:
\begin{align}
\label{ricci}
    R_{ij}=4\del_iT_{a;ja}-4\del_aT_{i;ja}-8T_{i;jb}T_{a;ba}+8T_{a;jb}T_{i;ba}.   
\end{align}
This also proves that the metric of a torsion-free $\S7$-structure is Ricci-flat, a result originally due to Bonan. Taking the trace of \eqref{ricci} gives the  scalar curvature $R$:
\begin{align}
\label{scalar}
    R=4\del_iT_{a;ia}-4\del_aT_{i;ia}+8|T|^2+8T_{a;jb}T_{j;ba}.    
\end{align}

\subsection{Proof of Theorem \ref{prop: Spin(7) Bryant formula}}
\label{sec: proof of Spin(7) Bryant formula}

This section addresses the proof of the Bryant-type formula for isometric $\S7$-structures, which are in bijection with unit spinors \cite[\textsection 2]{martin-merchan}. We derive a $\S7$ counterpart of the well-known formula for $\G2$-structures \cite[(3.6)]{Bryant2006}.

Fix a $\S7$-structure $\Phi_0$ and identify its corresponding unit spinor with the identity ${\mathbf 1}_+$. As the twistor space is an $\mathbb{R}\mathrm{P}^7$-bundle, any other isometric $\S7$-structure can be identified with a unit spinor $\psi$ and decomposed into 
\begin{equation}
    \psi=f {\mathbf 1}_+\oplus X,
    \quad\text{with}\quad 
    X\in\Gamma(\mathbf{7})
    \quad\text{and}\quad 
    |f|^2+|X|^2=1,
\end{equation}
in which $\mathbf{7}$ is the orthogonal complement to the span of ${\mathbf 1}_+$ (later on identified with $\Ima\bO$).
Then its square is
$$
\psi\otimes\psi
    = f^2{\mathbf 1}_+ \otimes {\mathbf 1}_+ + f({\mathbf 1}_+ \otimes X + X\otimes {\mathbf 1}_+) + X\otimes X
    \in\Cl_8 \simeq \Lambda \bR^8,
$$
and we are only interested in its $\Lambda^4 \bR^8$-component $[\psi\otimes\psi]_4$.

It is well-known that $[{\mathbf 1}_+ \otimes {\mathbf 1}_+]_4 = \Phi_0$ 
\cite[p. 355]{lawson-michelsohn}. In order to compute the cross-term $[{\mathbf 1}_+ \otimes X + X\otimes {\mathbf 1}_+]_4$, 
we take its inner-product with a general quadruple $e_i \wedge e_j \wedge e_k \wedge e_l \in \Lambda^4 \bR^8$, following eg. the proof of \cite[Prop 2.1]{D-H}):
\begin{align*}
    & 16 \langle {\mathbf 1}_+ \otimes X + X\otimes {\mathbf 1}_+ , e_i \wedge e_j \wedge e_k \wedge e_l\rangle \\
    &= 16 \langle {\mathbf 1}_+ , (e_i \wedge e_j \wedge e_k \wedge e_l) X \rangle + 16 \langle  X , (e_i \wedge e_j \wedge e_k \wedge e_l) {\mathbf 1}_+\rangle.
\end{align*}
With the identifications $\Cl_8 \simeq \bR(16) \simeq \End(\bP) \simeq \bP \otimes \bP^*$ and $\bP = S^+ \oplus S^-$, so that $\bP \simeq \bO \oplus \bO$, with $S_+ \simeq \bO\times \{0\}$ and $S_- \simeq \{0\} \times \bO$, a general spinor $s$ satisfies:
$$
(e_i \wedge e_j \wedge e_k \wedge e_l) s = (e_i(e_j(e_k(e_l s))))
$$
where the right-hand side is a composition of left-multiplications in $\bO$.

Take $e_i,e_j,e_k,e_l$ to be distinct unit orthonormal vectors ordered such that $7\geq l>k>j>i\geq 0$.
By standard properties of the octonions \cite[(5.6), (5.21), (5.31)]{salamon-walpuski},
we have
\begin{align*}
    & \langle {\mathbf 1}_+ \otimes X + X\otimes {\mathbf 1}_+ , e_i \wedge e_j \wedge e_k \wedge e_l\rangle \\
    &= \langle {\mathbf 1}_+ , (e_i(e_j(e_k(e_l X))))  \rangle + \langle  X , (e_i(e_j(e_k(e_l {\mathbf 1}_+)))) \rangle .
\end{align*}
Identifying ${\mathbf 1}_+$ with the identity of $\bO$ and defining the $\S7$-structure $\Phi_0$ by \cite[Thm 5.20]{salamon-walpuski}
we obtain
\begin{align*}
& \langle {\mathbf 1}_+ \otimes X + X\otimes {\mathbf 1}_+ , e_i \wedge e_j \wedge e_k \wedge e_l\rangle = 2 \Phi_0(e_i X,e_j,e_k, e_l).
\end{align*}

To compute the term $[X \otimes X]_4$, we again take its inner-product with $e_i \wedge e_j \wedge e_k \wedge e_l \in \Lambda^4 \bR^8$:
$$
\langle X \otimes X , e_i \wedge e_j \wedge e_k \wedge e_l \rangle= 
\varepsilon (X,e_l (e_k (e_j (e_i X)))),
$$
where $\epsilon$ is an inner-product on $\bP$.
If $\Phi_0(e_i,e_j,e_k,e_l)\neq 0$, then 
$$
e_l (e_k (e_j (e_i X)))= \begin{cases*}
    \phantom{-} \Phi_0 (e_i, e_j, e_k,e_l)X, \text{ if } X = e_i, e_j, e_k,e_l \\
    - \Phi_0 (e_i, e_j, e_k,e_l)X\ \  \text{otherwise}.
    \end{cases*}
$$
Decomposing $X= X_{ijkl} + X_*$, with $X_{ijkl} \in \Spa(e_i,e_j,e_k,e_l)$, we obtain
\begin{align*}
\varepsilon (X,e_l (e_k (e_j (e_i X)))) &= \Phi_0 (e_i, e_j, e_k,e_l)( |X_{ijkl}|^2 - |X_*|^2) \\
&= \Phi_0 (e_i, e_j, e_k,e_l)( 2|X_{ijkl}|^2 - |X|^2).
\end{align*}
On the other hand, we have
 \begin{align*}
    4(X\wedge (X\lrcorner\Phi_0))(e_i \wedge e_j\wedge e_k \wedge e_l) &=
    |X_{ijkl}|^2 \Phi_0(e_i , e_j, e_k , e_l),
\end{align*}
and, at least for such quadruples, the following relation holds:
$$
\langle X \otimes X , e_i \wedge e_j \wedge e_k \wedge e_l \rangle= 
(8(X\wedge (X\lrcorner\Phi_0)) - |X|^2 \Phi_0) (e_i , e_j, e_k , e_l).
$$
This formula remains valid for all the other quadruples, so long as  $X$ is pure imaginary.

Summing up the three computations, we have:

$$
[X \otimes X]_4 = (f^2- |X|^2) \Phi_0 + 2 f \Theta_{0,X} + 8(X\wedge (X\lrcorner\Phi_0)),
$$
where $\Theta_{0,X}(a,b,c,d) = \Phi_0 (a \cdot X,b,c,d)$ is a four-form.


\section{Harmonic $\S7$-flow}
\label{sec: Basics of the Spin(7)-flow}

\subsection{Derivation of the flow}
\label{sec:ifintro}
In this section we define the harmonic flow of $\S7$-structures and prove that it is a negative gradient flow.

\begin{definition}[Isometric $\S7$-structures.]
Two $\S7$-structures $\s7_1$ and $\s7_2$ on $M$ are called \emph{isometric} if they induce the same Riemannian metric, that is, if $g_{\s7_1}=g_{\s7_2}$. We will denote by $\llbracket \s7\rrbracket$ the space of $\S7$-structures that are isometric to a given $\S7$-structure $\s7$.
\end{definition}

\begin{definition}
Let $\s7_0$ be a fixed initial $\S7$-structure on $M$. The \emph{energy functional} $E$ on the set $\llbracket \s7_0 \rrbracket$ is
\begin{align}
\label{energyfuncdefn}
    E(\s7)=\frac 12 \int_M |T_{\s7}|^2 \vol_{g_{\s7}}     
\end{align}
where $T_{\s7}$ is the torsion of $\s7$.
\end{definition}
At this point it is worth  noticing that, whereas the energy functional \eqref{energyfuncdefn} could be considered over \emph{all} $\S7$-structures, in the present setting it is being considered only over isometric $\S7$-structures. The natural question here, inspired by harmonic map theory, is whether, given a $\S7$-structure $\s7_0$, there is a  ``best'' $\S7$-structure in $\llbracket \s7_0 \rrbracket$. An obvious way to study this question is to consider the negative gradient flow of $E$.

\subsubsection{Direct approach: variational calculus}
 
The most general flow of $\S7$-structures \cite{karigiannis-spin7} is given by
\begin{align}\label{gflowdefn}
\frac{\pt \s7}{\pt t}=A\diamond \s7 = (h+X)\diamond \s7   
\end{align}
where $A_{ij}=h_{ij}+X_{ij}$ with $h_{ij}\in S^2$ and $X_{ij}\in \Omega^2_7$. Thus, the most general flow of $\S7$-structures is given by a time-dependent family of symmetric $2$-tensors and a time-dependent family of $2$-forms lying in the subspace $\Omega^2_7$. In this case the evolution of the metric is given by 
\begin{align}\label{gflowmetricevol}
\frac{\pt g}{\pt t}=2h_{ij}.    
\end{align}
We start by considering the variation of the torsion $T$ with respect to variations of the $\S7$-structures that preserve the metric.

\begin{proposition}\label{prop:1var}
Let $(\s7_t)_{t\in (-\varepsilon, \varepsilon)}$ be a smooth family of $\S7$-structures in the class $\llbracket \s7_0 \rrbracket$. By \eqref{gflowdefn} and \eqref{gflowmetricevol}, we can write $\left.\frac{\pt}{\pt t}\right|_{t=0}\s7_t=X\diamond \s7$ for some $X\in \Omega^2_7$. If $T_t$ is the torsion of $\s7_t$ then
\begin{align}\label{1vareqn}
\left.\frac{\pt}{\pt t}\right|_{t=0}(T_{t})_{m;is}=X_{ip}T_{m;ps}-X_{sp}T_{m;pi}+\pi_7(\del_mX_{is}).  
\end{align}
\end{proposition}

\begin{proof}
As the $\S7$-structures are varying in an isometry class, $g_t=g_0$ for all $t\in (-\varepsilon, \varepsilon)$ and hence the covariant derivative $\del$ is independent of $t$ and so is the volume form. Since $\left.\frac{\pt}{\pt t}\right|_{t=0}\s7_t=X\diamond \s7$, by \cite[Lemma 3.3]{karigiannis-spin7}, $\left.\frac{\pt}{\pt t}\right|_{t=0}\del \s7_t=X\diamond \del \s7+\del X\diamond \s7$. Explicitly, we have
\begin{align}
\left.\frac{\pt}{\pt t}\right|_{t=0}(\s7_t)_{ijkl}&=X_{ip}\s7_{pjkl}+X_{jp}\s7_{ipkl}+X_{kp}\s7_{ijpl}+X_{lp}\s7_{ijkp}, \label{1varpf1} \\
\left.\frac{\pt}{\pt t}\right|_{t=0}(\del_m(\s7_{t})_{ijkl})&=X_{ip}\del_m\s7_{pjkl}+X_{jp}\del_m\s7_{ipkl}+X_{kp}\del_m\s7_{ijpl}+X_{lp}\del_m\s7_{ijkp} \nonumber \\
& \quad + (\del_mX_{ip})\s7_{pjkl}+(\del_mX_{jp})\s7_{ipkl}+(\del_mX_{kp})\s7_{ijpl}+(\del_mX_{lp})\s7_{ijkp}.
\end{align}
Using the above equations and \eqref{eq:Texpress}, we compute
\begin{align}
96\left. \frac{\pt}{\pt t}\right|_{t=0}(T_t)_{m;is}&=\left.\frac{\pt}{\pt t}\right|_{t=0}((\del_m\s7_{ijkl})\s7_{sjkl}) \nonumber \\
&= \s7_{sjkl}\left.\frac{\pt}{\pt t}\right|_{t=0}(\del_m(\s7_{t})_{ijkl}) +(\del_m\s7_{ijkl})\left.\frac{\pt}{\pt t}\right|_{t=0}(\s7_t)_{ijkl} \nonumber \\
&=\Big [X_{ip}\del_m\s7_{pjkl}+X_{jp}\del_m\s7_{ipkl}+X_{kp}\del_m\s7_{ijpl}+X_{lp}\del_m\s7_{ijkp} \nonumber \\
& \quad + (\del_mX_{ip})\s7_{pjkl}+(\del_mX_{jp})\s7_{ipkl}+(\del_mX_{kp})\s7_{ijpl}+(\del_mX_{lp})\s7_{ijkp}  \Big]\s7_{sjkl} \nonumber \\
& \quad + \del_{m}\s7_{ijkl}\Big[X_{ip}\s7_{pjkl}+X_{jp}\s7_{ipkl}+X_{kp}\s7_{ijpl}+X_{lp}\s7_{ijkp}  \Big] \nonumber .
\end{align}
Using the contraction identities \eqref{eq:impiden1}-\eqref{eq:impiden4}, the above becomes
\begin{align*}
96\left. \frac{\pt}{\pt t}\right|_{t=0}(T_t)_{m;is}&= 96X_{ip}T_{m;ps}-96X_{sp}T_{m;pi} +42\del_mX_{is}+\del_mX_{jp}(6g_{is}g_{pj}-6g_{ij}g_{ps}-4\s7_{ipsj})\\
& \quad +\del_mX_{kp}(6g_{is}g_{pk}-6g_{ik}g_{ps}-4\s7_{ipsk})+\del_mX_{lp}(6g_{is}g_{pl}-6g_{il}g_{ps}-4\s7_{ipsl}) \\
&=96X_{ip}T_{m;ps}-96X_{sp}T_{m;pi} +42\del_mX_{is}-18\del_mX_{is}-12\del_mX_{jp}\s7_{isjp}\\
&=96X_{ip}T_{m;ps}-96X_{sp}T_{m;pi} +24\del_mX_{is}-12\del_mX_{jp}\s7_{isjp}
\end{align*}
which, using \eqref{eq:pi7} for the last two terms, is precisely \eqref{1vareqn}.
\end{proof}

\medskip

Let $E$ be the energy functional from \eqref{energyfuncdefn}, restricted to the isometry class $\llbracket \s7 \rrbracket$ of $\S7$-structures. The following result motivates our formulation of the harmonic $\S7$-flow, cf. Definition \ref{iflowdefn}.

\begin{lemma}
\label{lemma:E1var}
    The gradient of $16 E:\llbracket \s7 \rrbracket \rightarrow \bR_{\geq 0}$ at the point $\s7$ is $-\Div T\diamond \s7$, where $T$ is the torsion of $\s7$. In other words, if $(\s7_t)_{t\in (-\varepsilon, \varepsilon)}$ is a smooth family of $\S7$-structures in $\llbracket \s7_0 \rrbracket$ and $\left. \frac{d}{dt}\right|_{t=0}\s7_t=\xi$, then
\begin{align*}
    \left.\frac{d}{dt}\right|_{t=0}16 E(\s7_t) 
    =-\int_M \langle \Div T\diamond \s7, \xi  \rangle \vol.  
\end{align*}
\end{lemma}

\begin{proof}
We use Proposition \ref{prop:1var} to compute
\begin{align*}
\left.\frac{d}{dt}\right|_{t=0} E(\s7_t)&=  \left.\frac{d}{dt}\right|_{t=0}\frac 12\int_M(T_t)_{m;is}(T_t)_{m;is}\vol\\
&=\int_M T_{m;is}(X_{ip}T_{m;ps}-X_{sp}T_{m;pi}+\pi_7(\del_mX_{is}))\vol\\
&=\int_M T_{m;is}(X_{ip}T_{m;ps}-X_{sp}T_{m;pi}+\frac 14\del_mX_{is}-\frac 18\del_mX_{jp}\s7_{isjp}) \vol.
\end{align*}
The first two terms vanish, as $T_{m;is}T_{m;ps}$ and $T_{m;is}T_{m;pi}$ are symmetric in $i,\ p$ and $s,\ p$, respectively, and $X\in \Omega^2_7$. We integrate by parts on the remaining two terms:
\begin{align*}
    \left.\frac{d}{dt}\right|_{t=0} E(\s7_t)&=\int_M -\frac{1}{4}\del_mT_{m;is}X_{is}+\frac{1}{8}\del_mT_{m;is}X_{jp}\s7_{isjp}\vol  \\
    &=\int_{M}-\frac{1}{4}\del_mT_{m;is}X_{is}-\frac{3}{4}\del_mT_{m;is}X_{is} \vol\\ 
    &=-\int_M \langle \Div T, X \rangle \vol.
\end{align*}
Using Proposition \ref{prop:diaproperties2}--$2.$, we see that $A, B\in \Omega^2_7$ satisfy $\langle A\diamond \s7, B\diamond \s7 \rangle=16\langle A, B\rangle $, so the above equation becomes
\begin{align}\label{E1varpf1}
\left.\frac{d}{dt}\right|_{t=0} 16E(\s7_t)=-16\int_M \langle \Div T, X\rangle \vol=\int_M \langle \Div T\diamond \s7, X\diamond \s7 \rangle \vol. 
\end{align}
The result now follows from the splitting $
\Omega^4=\Omega^4_1\oplus \Omega^4_7\oplus \Omega^4_{27}\oplus \Omega^4_{35}
$,
with $\Omega^4_7=\{Y\diamond \s7 \mid Y\in \Omega^2_7\}$ and $T_{m;is}\in \Omega^2_7$, for each $m$.
\end{proof}

Comparing \eqref{gflowdefn} and \eqref{eq: Har Spin(7) Flow}, we see that $h\equiv 0$ for the harmonic flow, and hence it is indeed a flow of isometric $\S7$-structures.

\subsubsection{Alternative approach: representation theory}

The harmonicity condition for $\S7$-structures can be inferred as an instance of the general framework in \textsection\ref{sec: H-structures}, stemming from the representation theory of $\S7$.
The key ingredient is the connection form $f$, which identifies the vertical of the tangent bundle of the twistor space with $\mathfrak{m}$, the (naturally reductive) complement of $\mathfrak{so}(7)= \mathfrak{spin}(7))\subset\mathfrak{so}(8)$. Sections of this space correspond to $\S7$-structures and restricting ourselves to the vertical part means we only look at variations through $\S7$-structures.

If $\tilde{\s7}$ is the universal $\S7$-structure [cf. \eqref{eq: universal section}], then the connection form is characterised by
$$
\nabla_A \tilde{\s7} = f(A).\tilde{\s7} .
$$
Identifying $\mathfrak{so}(8)\simeq\Omega^2$, the reductive complement $\mathfrak{m}$ is then identified with $\Omega^2_7$ and $f(A).\tilde{\s7}$ should be understood as the the natural $\mathrm{GL}(8)$-action encoded in the diamond operator \eqref{eq:diadefn1}. We invert the $\diamond$ operator with a triple contraction $\lrcorner_3$:
$$
(\beta \diamond \s7) \lrcorner_3 \s7 = 96 \beta .
$$
The connection form is then given by $96 f(A) = \nabla_A \tilde{\s7} \lrcorner_3 \tilde{\s7}$ and,  pulling back by a section $\sigma$ of the twistor space, we obtain:
$$
f(d\sigma(X)) = \tfrac{1}{96} (\nabla_{X}\s7) \lrcorner_3 \s7 .
$$
Everything now  follows from this expression,
which is precisely the torsion $T(X)$ of \eqref{eq:Texpress}. 
Then the (vertical) energy density is $|d^\cV \sigma|^2 = |T|^2$, so the energy is precisely the $L^2$-norm of the torsion. The vertical tension field is 
$$
I(\tau^\cV (\sigma)) = \sum_1^8 \nabla_{e_i} (T(e_i)) - T(\nabla_{e_i} e_i) = \Div T
$$
and the flow $\frac{d\sigma_t}{dt} = \tau^\cV (\sigma_t)$ is equivalent to $I(\frac{d\sigma_t}{dt}) = I(\tau^\cV (\sigma_t))$
where $I$ plays the role of an extended connection $f$ on $M\times \bR$.
We know that $I(\tau^\cV (\sigma_t))= \Div T_t$ and, generalising to $M\times \bR$ (or at least an interval) all the previous objects, we have  
$$
I\left(\frac{d\sigma_t}{dt}\right) = \tfrac{1}{96} \frac{d{\s7}_t}{dt} \lrcorner_3 {\s7}_t .
$$
On the other hand, since $
\Div T_t = \tfrac{1}{96} (\Div T_t \diamond {\s7}_t) \lrcorner_3 {\s7}_t
\in\Omega^2_7
$,
and $\lrcorner_3 {\s7}_t$ is an isomorphism on $\Omega^2_7$ (its kernel being $\Omega^2_{21}$), we recover indeed
$$
\frac{d{\s7}_t}{dt} = \Div T_t \diamond {\s7}_t .
$$

\begin{remark}
This flow is the vertical part of the harmonic map flow, so short time existence of \eqref{eq: Har Spin(7) Flow} follows from general properties of parabolic equations, cf. \cite{loubeau-saearp}.
\end{remark}

\subsection{Immediate consequences from the harmonic $H$-flow}

Let us collect from the outset a number of analytic properties of the harmonic $\S7$-flow which can be effortlessly deduced from the general theory in \textsection\ref{sec: general theory of H-str}.

\subsubsection{Harmonic $\S7$-solitons and self-similar solutions}
\label{sec: Spin(7)-solitons}

The specific notion of self-similar solution to \eqref{eq: Har Spin(7) Flow}, derived from Definition \ref{def-sssol-GENERAL}, is the following: 
\begin{definition}
\label{def-sssol}
    Let $\{\s7_t\}_{t\in (a, b)}$ be a solution of the harmonic flow \eqref{eq: Har Spin(7) Flow} where $0\in (a, b)$. We say that it is a \emph{self-similar solution} if there exist a function $\rho(t)$ with $\rho(0)=1$, a $\S7$-structure $\s7_0$ and a family of diffeomorphisms $f_t:M\rightarrow M$ such that
\begin{align*}
\s7_t = \rho(t)^{4}{f_t}^*\s7_0    
\end{align*}
for all $t\in (a, b)$. 
\end{definition}

For $Y\in \Gamma(TM)$, in view of the splitting \eqref{eq:splitting TM* x TM} of $TM^*\otimes TM$ and Proposition \ref{prop:diaproperties1}, we may write 
$$
\nabla Y=\frac 12 \cL_Yg +\nabla_7 (Y) +\nabla_{21}(Y),
$$
where $\nabla_k(Y):=\pi_k(\nabla Y)\in \Omega^2_k$. In particular  $\nabla_{21}Y\in\Omega^2_{21}=\ker(\diamond\,\Phi)$. We now use \eqref{eq: Lie of xi-GENERAL} to get 
\begin{equation}
\label{eq: Lie of Phi}
 \cL_Y \s7= (Y\lrcorner T +\frac 12 \cL_Yg +\nabla_7 (Y))\diamond \Phi.
\end{equation}    
    
The stationary condition for harmonic $\S7$-solitons is obtained from Lemma \ref{lemma: vec fields for self-sim sols-GENERAL}:
\begin{lemma}
\label{lemma: vec fields for self-sim sols}
    Given a self-similar solution $\{\s7_t\}_{t\in (a, b)}$  of the harmonic flow \eqref{eq: Har Spin(7) Flow}, there is a family  of vector fields $\{X_t\}\subset\sX(M)$ such that
    $$
    \Div{T_t}= X_t\lrcorner T_t + \nabla_7 (X_t),
    $$
    where $\nabla_7 :=\pi_7\circ\nabla:\sX(M)\to\Omega^2_7(M)$. 
\end{lemma}
    
\begin{definition}
\label{def: harmonic Spin(7)-soliton}
    A \emph{harmonic $\S7$-soliton} on a Riemannian manifold $(M^8,g)$ is a triple $(\hat\Phi,\hat X ,c)$, in which $\hat\Phi$ is a $\S7$-structure with induced metric $g$, $\hat X \in\sX(M)$ and $c\in\bR$ are such that
\begin{align}
\begin{split}
    \cL_{\hat X }g&=cg,\\
    \Div\hat{T}&= \hat X \lrcorner \hat T + \nabla_7 (\hat X ).
\end{split}
\end{align}
\end{definition}

We obtain the family of stationary vector fields $\{X_t\}\subset\sX(M)$ predicted in Lemma \ref{lemma: vec fields for self-sim sols}, by setting [cf. \eqref{eq: stationary vf-GENERAL}]
$$
X_t:=(f_t)^{-1}_*W_t=\alpha(t)(f_t)^{-1}_*\hat X =\alpha(t)\hat X.
$$ 
Finally, the equivalence between harmonic $\S7$-solitons and self-similar solutions to \eqref{eq: Har Spin(7) Flow} is summarised in Theorem \ref{cor: iso soliton => self-sim sol}, as an instance of Proposition \ref{prop: iso soliton => self-sim sol-GENERAL}.

\subsubsection{Doubling-time estimates}

As an immediate consequence of the general theory of harmonic section flows \cite[Thm.2]{loubeau-saearp}, we obtain a doubling-time estimate for \eqref{eq: Har Spin(7) Flow} which says that if the torsion is bounded at the initial time then it will remain bounded until some subsequent time.

\begin{corollary}[Doubling-Time Estimate]
\label{dteforif} 
    Let $\Phi_t$ be a solution of the harmonic flow on a compact $8$-manifold $M$ for $t\in [0,\tau]$. Then there exists $\delta >0$ such that
\begin{equation}
    \cT(t) \leq 2 \cT(0)
\end{equation}
    for all $0\leq t \leq \delta$, where 
$$
\cT(t) = \sup_{M} |T(x,t)| .
$$
\end{corollary}

\subsubsection{Shi-type estimates for the harmonic $\S7$-flow}

For $H=\S7$, the homogeneous fibre bundle $N = P/H \to M^8$ of \textsection\ref{sec: H-structures} is compact, with fibres isometric to $\SO(8)/\S7\simeq\bRP^7$. In this case, O'Neill's curvature formulas~\cite{oneill} enable us to reduce the hypotheses of Proposition~\ref{Prop:Shi-estimates} to geometric bounds over the base manifold $M$ only.

\begin{corollary}
\label{shitypeest}
    Let $\s7_t$ be a solution of the harmonic flow \eqref{eq: Har Spin(7) Flow} and assume that there exist $K \geq 1$ and constants $B_j$, $0\leq j\in\mathbb{Z}$, such that 
\begin{align*}
    |T| \leq K,\ \ |\nabla^j \Riem| \leq B_j K^{j+2},\ \text{for all} \ j\geq 0.
\end{align*}
    Then, for all $m\in \mathbb{N}$, there exist constants $C_m=C_m(g)$ such that for all $t\in [0, \frac{1}{K^4}]$, we have
\begin{align*}
    |\nabla^{m} T| \leq C_mK t^{-m/2}.  
\end{align*}    
\end{corollary}

\subsection{New analytic preliminaries for the harmonic $\S7$-flow}

We now derive some properties which are context-specific, and cannot heretofore be obtained directly as instances of the general theory for $H$-flows.

\subsubsection{Parabolic rescaling}
\label{sec:pscale}

We discuss the `parabolic rescaling' for \eqref{eq: Har Spin(7) Flow}. The natural parabolic rescaling for a geometric evolution equation involves scaling the time variable $t$ by $c^2t$, when the space variable scales by $c$. We will use this notion frequently in the paper.

\begin{lemma}
\label{lemma:pscale}
    Let $c>0$ be a constant. If $\{\s7(t)\}$ is a solution of the harmonic $\S7$-flow \eqref{eq: Har Spin(7) Flow} with $\s7(0)=\s7$, then $\wtd{\s7}(\wtd{t})=c^4\s7(c^2t)$ is still a solution of \eqref{eq: Har Spin(7) Flow} with $\wtd{\s7}(0)=c^4\s7$. In particular,
$$
\wtd{T}=c^2T \qandq \Div_{\widetilde{g}}\wtd{T}=\Div_{g}T.
$$
\end{lemma}
\begin{proof}
    Define a new $\S7$-structure $\wtd{\s7}=c^4\s7$. It follows from \eqref{eq:metric} and the expressions of $B_{ij}$, $A$ in \textsection \ref{sec:prelims} that $\wtd{g}=c^2g$, as well as $\wtd{g}^{-1}=c^{-2}g^{-1}$ and $\wtd{\vol}=c^8\vol$. Recall that we are suppressing the writing of $g^{-1}$ terms because we are using an orthonormal frame. Hence \eqref{eq:Texpress} yields $\wtd{T}=c^2T$. Thus, as a $2$-form, $\Div_{\widetilde{g}}\wtd{T}=\Div_{g}T$, and therefore  $(\Div_{\wtd{g}}\wtd{T})\diamond \wtd{\s7}=c^{-2}c^4(\Div_{g}T\diamond \s7)=c^2(\Div_g T)\diamond \s7$. The main assertion now follows from \eqref{eq: Har Spin(7) Flow}.
\end{proof}

Since $\wtd{\del}=\del$, we note, for later use that if $\wtd{\s7}=c^4\s7$ then
\begin{align}\label{pscaletor}
|\wtd{\del}^j\wtd{\Riem}|_{\wtd{g}}=c^{-(2+j)}|\del^j\Riem|_g,\ \ \ \ \ \ \ \  |\wtd{\del}^j\wtd{T}|_{g}=c^{-(1+j)}|\del T|_g \cdot  
\end{align}

\subsubsection{Torsion evolution}

We determine the evolution equations of the torsion and its norm under the harmonic $\S7$-flow, following a process similar to the $\G2$ case \cite[Lemma 3.1]{dgk-isometric}.

Let $\{ \partial_t, e_1, \dots , e_7\}$ be an orthonormal (geodesic) frame and use $t$ and $i$ as shorthands for $\partial_t$ and $e_i$. First we use the formula
$$
(R(i,j)T)(a,b,c) = - T(R(i,j)a,b,c) - T(a,R(i,j)b,c) - T(a,b,R(i,j)c)
$$
to derive a formula for the Laplacian of the torsion of a $\S7$-structure.

\begin{lemma} Let $\Delta = +\tr{\nabla_i\nabla_i}$ be the Laplacian operator, then
 \begin{align*}
   (\Delta T)_{m;ab} =& \nabla_m \nabla_i T_{i;ab} - T_{q;ab}R_{imiq} - T_{i;qb}R_{imaq} - T_{i;aq}R_{imbq} + 2 \nabla_i T_{i;ap}T_{m;bp} + 2 T_{i;ap}\nabla_i T_{m;bp} \\
    &  - 2 \nabla_i T_{m;ap}T_{i;bp} - 2 T_{m;ap}\nabla_i T_{i;pb} + \tfrac14 \nabla_i R_{miab} - \tfrac18 \nabla_i R_{mipq} \Phi_{pqab} - \tfrac18 R_{mipq} \nabla_i \Phi_{pqab}.
\end{align*}
\end{lemma}

\begin{proof}
Use the $\S7$ Bianchi-type identity \eqref{spin7bianchi} to compute:
\begin{align*}
    \nabla_i \nabla_i  T_{m;ab} 
    =&\ \nabla_i (\nabla_i  T_{m;ab}) \\
    =&\ \nabla_i ( \nabla_m  T_{i;ab} + 2 T_{i;ap} T_{m;pb} - 2 T_{m;ap}T_{i;pb} + \tfrac14 R_{miab} - \tfrac18 R_{mipq} \Phi_{pqab} )\\
    =&\ \nabla_i \nabla_m  T_{i;ab} + 2 \nabla_i T_{i;ap} T_{m;pb} + 2 T_{i;ap} \nabla_i T_{m;pb} - 2 \nabla_i T_{m;ap}T_{i;pb} - 2  T_{m;ap} \nabla_i T_{i;pb} \\
    & + \tfrac14 \nabla_i R_{miab} - \tfrac18 \nabla_i R_{mipq} \Phi_{pqab} - \tfrac18 R_{mipq} \nabla_i\Phi_{pqab} 
\end{align*}
and then use 
\begin{align*}
    \nabla_i \nabla_m  T_{i;ab} =  \nabla_m \nabla_i  T_{i;ab} - T_{q;ab}R_{imiq} - T_{i;qb}R_{imaq} - T_{i;aq}R_{imbq}.
\qedhere
\end{align*}
\end{proof}

\begin{proposition}
Let $\{\s7_t\}$ be a solution of the harmonic $\S7$-flow \eqref{eq: Har Spin(7) Flow}; then its torsion evolves according to the equation
\begin{align*}
  4\frac{\partial}{\partial t} T_{m;is} 
  =&\ 
 4(\Delta T)_{m;is} \\
  &+ \nabla_a T_{m;bc}\Big(4T_{a;bp}\s7_{pcis}+T_{a;ip}\s7_{bcps}+T_{a;sp}\s7_{bcip} \Big) + T_{m;bc} \nabla_a T_{a;bp} \s7_{pcis} \\
&+ 3\nabla_a T_{a;ip} T_{m;ps} + \nabla_a T_{a;sp} T_{m;pi} - 2 T_{a;ip}\nabla_a T_{m;sp} + 2 \nabla_a T_{m;ip}T_{a;sp} \\
 & + T_{m;bc} T_{a;bp} \Big(T_{a;pq}\s7_{qcis}+ T_{a;cq}\s7_{pqis}+2T_{a;iq}\s7_{pcqs}+
2T_{a;sq}\s7_{pciq} \Big) \\
&+  \tfrac12 T_{m;bc}T_{a;ip} \Big(
T_{a;pq}\s7_{bcqs}+
2T_{a;sq}\s7_{bcpq} \Big) 
 + \tfrac12 T_{m;bc} T_{a;sp}
T_{a;pq}\s7_{bciq} \\
& + 4 T_{q;is}R_{amaq}
- (\nabla_a R_{mais} - \tfrac12 \nabla_a R_{mapq} \Phi_{pqis})\\
&+ T_{a;qs}R_{amiq} + T_{a;iq}R_{amsq}  
  + \tfrac18 R_{mapq} \nabla_a \Phi_{pqis}  - T_{a;qc}R_{ambq}\Phi_{bcis} 
  - \tfrac{1}{16} R_{mapq} \nabla_a \Phi_{pqbc}\Phi_{bcis}.
\end{align*}
\end{proposition}

\begin{proof}
    The general expression for the flow is
$$
\frac{\partial}{\partial t} T_{m;is} = X_{ip} T_{m;ps} - X_{sp} T_{m;pi} + \pi_7 (\nabla_m X_{is}),
$$
where $(\pi_7 \beta)_{ij} = \tfrac14 \beta_{ij} - \tfrac18 \beta_{ab}\Phi_{abij}$. Take $X_{is} = (\Div T)_{is} = \nabla_a T_{a;is}$ for the harmonic $\S7$-flow, so
\begin{align}
\label{eq: dT/dt}
    \frac{\partial}{\partial t} T_{m;is} = \nabla_r T_{r;ip} T_{m;ps} - \nabla_r T_{r;sp} T_{m;pi} + \pi_7 (\nabla_m (\nabla_r T_{r;is})) .
\end{align}
But
\begin{align*}
    \nabla_m \nabla_r T_{r;is} 
    =&\ (\Delta T)_{m;is} + T_{q;is}R_{amaq} + T_{a;qs}R_{amiq} + T_{a;iq}R_{amsq} - 2 \nabla_a T_{a;ip}T_{m;sp} - 2 T_{a;ip}\nabla_a T_{m;sp} \\
    &  + 2 \nabla_a T_{m;ip}T_{a;sp} + 2 T_{m;ip}\nabla_a T_{a;ps} - \tfrac14 \nabla_a R_{mais} + \tfrac18 \nabla_a R_{mapq} \Phi_{pqis} + \tfrac18 R_{mapq} \nabla_a \Phi_{pqis}
\end{align*}
and note that $T_{q;is}R_{amaq}$ and $-\tfrac14 (\nabla_a R_{mais} - \tfrac12 \nabla_a R_{mapq} \Phi_{pqis})= - \pi_7 (\nabla_a R_{mais})$ are in $\Omega^2_7$, 
so
\begin{align}
\label{eq-proj}
    4\pi_7 (\nabla_m \nabla_r T_{r;is}) 
    =&\ 4\pi_7 ((\Delta T)_{m;is}) + 4 T_{q;is}R_{amaq}
- (\nabla_a R_{mais} - \tfrac12 \nabla_a R_{mapq} \Phi_{pqis}) \\
&+ T_{a;qs}R_{amiq} + T_{a;iq}R_{amsq} - 2 \nabla_a T_{a;ip}T_{m;sp} 
- 2 T_{a;ip}\nabla_a T_{m;sp} \notag \\
&  + 2 \nabla_a T_{m;ip}T_{a;sp} + 2 T_{m;ip}\nabla_a T_{a;ps}  + \tfrac18 R_{mapq} \nabla_a \Phi_{pqis} \notag\\
    &  - \tfrac12 T_{a;qc}R_{ambq}\Phi_{bcis} - \tfrac12 T_{a;bq}R_{amcq}\Phi_{bcis} 
    +  \nabla_a T_{a;bp}T_{m;cp}\Phi_{bcis} \notag \\
    & +  T_{a;bp}\nabla_a T_{m;cp} \Phi_{bcis} -  \nabla_a T_{m;bp}T_{a;cp} \Phi_{bcis} 
  -  T_{m;bp}\nabla_a T_{a;pc}\Phi_{bcis} 
  - \tfrac{1}{16} R_{mapq} \nabla_a \Phi_{pqbc}\Phi_{bcis}.\notag
\end{align}
Since $T \in \Omega^2_7$, we have 
\begin{align*}
  T_{q;bc} \s7_{bcis} &= -6 T_{q;is} ;\\
  (96)^2  T_{m;rp} T_{n;sp} &= 6  \nabla_m \Phi_{rjkl} \nabla_n \Phi_{sjkl}
    -9 \nabla_m \Phi_{rjkl} \nabla_n \Phi_{sjbc}\Phi_{klbc} ; \\
    96  T_{a;qc} \Phi_{bcis} &= - 6 \nabla_a \Phi_{qbis} + 3 \nabla_a \Phi_{qbkl} \Phi_{klis} + 
    3 \nabla_a \Phi_{qikl} \Phi_{klsb} + 3 \nabla_a \Phi_{qskl} \Phi_{klbi}.
\end{align*}
The first two terms on line 4 of  \eqref{eq-proj} combine to give
$$
- \tfrac12 T_{a;qc}R_{ambq}\Phi_{bcis} - \tfrac12 T_{a;bq}R_{amcq}\Phi_{bcis} = - T_{a;qc}R_{ambq}\Phi_{bcis},
$$
the third term of line 4 and the third term on line 5 of  \eqref{eq-proj} give
$$
\nabla_a T_{a;bp}T_{m;cp}\Phi_{bcis}-  T_{m;bp}\nabla_a T_{a;pc}\Phi_{bcis} = 0.
$$
The first two terms on line 5 of  \eqref{eq-proj} combine to give
$$
T_{a;bp}\nabla_a T_{m;cp} \Phi_{bcis} 
  -  \nabla_a T_{m;bp}T_{a;cp} \Phi_{bcis} = 2 T_{a;bp}\nabla_a T_{m;cp} \Phi_{bcis},
$$
so
\begin{align*}
4\pi_7 (\nabla_m \nabla_r T_{r;is}) 
=&\ 4\pi_7 ((\Delta T)_{m;is}) + 4 T_{q;is}R_{amaq}
- (\nabla_a R_{mais} - \tfrac12 \nabla_a R_{mapq} \Phi_{pqis})\\
&+ T_{a;qs}R_{amiq} + T_{a;iq}R_{amsq} - 2 \nabla_a T_{a;ip}T_{m;sp} 
- 2 T_{a;ip}\nabla_a T_{m;sp} \\
&  + 2 \nabla_a T_{m;ip}T_{a;sp} + 2 T_{m;ip}\nabla_a T_{a;ps}  + \tfrac18 R_{mapq} \nabla_a \Phi_{pqis}\\
    &  - T_{a;qc}R_{ambq}\Phi_{bcis} + 2 T_{a;bp}\nabla_a T_{m;cp} \Phi_{bcis} 
  - \tfrac{1}{16} R_{mapq} \nabla_a \Phi_{pqbc}\Phi_{bcis}
\end{align*}
and 
\begin{align*}
    4\frac{\partial}{\partial t} T_{m;is} 
    =&\ 4\nabla_a T_{a;ip} T_{m;ps} - 4\nabla_a T_{a;sp} T_{m;pi} + 4\pi_7 (\nabla_m \nabla_r T_{r;is}) \\
    =&\ 4\nabla_a T_{a;ip} T_{m;ps} - 4\nabla_a T_{a;sp} T_{m;pi} + 4\pi_7 ((\Delta T)_{m;is}) + 4 T_{q;is}R_{amaq}
- (\nabla_a R_{mais} - \tfrac12 \nabla_a R_{mapq} \Phi_{pqis})\\
    &+ T_{a;qs}R_{amiq} + T_{a;iq}R_{amsq} - 2 \nabla_a T_{a;ip}T_{m;sp} 
- 2 T_{a;ip}\nabla_a T_{m;sp} \\
    &  + 2 \nabla_a T_{m;ip}T_{a;sp} + 2 T_{m;ip}\nabla_a T_{a;ps}  + \tfrac18 R_{mapq} \nabla_a \Phi_{pqis}\\
    &  - T_{a;qc}R_{ambq}\Phi_{bcis} + 2 T_{a;bp}\nabla_a T_{m;cp} \Phi_{bcis}
  - \tfrac{1}{16} R_{mapq} \nabla_a \Phi_{pqbc}\Phi_{bcis} \\
    =&\ 6\nabla_a T_{a;ip} T_{m;ps} - 2\nabla_a T_{a;sp} T_{m;pi} + 4\pi_7 ((\Delta T)_{m;is}) + 4 T_{q;is}R_{amaq}
- (\nabla_a R_{mais} - \tfrac12 \nabla_a R_{mapq} \Phi_{pqis})\\
    &+ T_{a;qs}R_{amiq} + T_{a;iq}R_{amsq}  
- 2 T_{a;ip}\nabla_a T_{m;sp} + 2 \nabla_a T_{m;ip}T_{a;sp}  + \tfrac18 R_{mapq} \nabla_a \Phi_{pqis}\\
    &  - T_{a;qc}R_{ambq}\Phi_{bcis} + 2 T_{a;bp}\nabla_a T_{m;cp} \Phi_{bcis} 
  - \tfrac{1}{16} R_{mapq} \nabla_a \Phi_{pqbc}\Phi_{bcis} .
\end{align*}

We compute $4\pi_7 ((\Delta T)_{m;is})$:
\begin{align*}
    4\pi_7 ((\Delta T)_{m;is}) &= (\Delta T)_{m;is} -\tfrac12 (\Delta T)_{m;bc}\s7_{bcis} \\
    &= (\nabla_a \nabla_a T)_{m;is} -\tfrac12 (\nabla_a \nabla_a T)_{m;bc}\s7_{bcis} .
\end{align*}
Differentiating $ T_{m;bc} \s7_{bcis} = -6 T_{m;is}$  twice, we have
$$
(\nabla_a \nabla_a T)_{m;bc} \s7_{bcis} + 2 \nabla_a T_{m;bc} \nabla_a \s7_{bcis} +
T_{m;bc} (\nabla_a \nabla_a \s7)_{bcis} = -6 (\nabla_a \nabla_a T)_{m;is},
$$
so
\begin{align*}
4\pi_7 ((\Delta T)_{m;is}) &=  (\Delta T)_{m;is} -\tfrac12 (\nabla_a \nabla_a T)_{m;bc}\s7_{bcis} \\
&= (\Delta T)_{m;is} -\tfrac12 \Big( -6 (\nabla_a \nabla_a T)_{m;is} - 2 \nabla_a T_{m;bc} \nabla_a \s7_{bcis} - T_{m;bc} (\nabla_a \nabla_a \s7)_{bcis} \Big)\\
&= 4(\Delta T)_{m;is} + \nabla_a T_{m;bc} \nabla_a \s7_{bcis} + \tfrac12 T_{m;bc} (\nabla_a \nabla_a \s7)_{bcis}, 
\end{align*}
and the third term expands as follows:
\begin{align*}
    (\nabla_a \nabla_a \s7)_{bcis} =&\ 
    \nabla_a T_{a;bp} \s7_{pcis} + \nabla_a T_{a;cp} \s7_{bpis} +\nabla_a T_{a;ip} \s7_{bcps} +\nabla_a T_{a;sp} \s7_{bcip} \\
    & +T_{a;bp} \nabla_a\s7_{pcis} +  T_{a;cp} \nabla_a\s7_{bpis} + T_{a;ip} \nabla_a\s7_{bcps} + T_{a;sp} \nabla_a\s7_{bcip}.
\end{align*}
So
\begin{align*}
    4\pi_7 ((\Delta T)_{m;is}) 
    =&\ 4(\Delta T)_{m;is} + \nabla_a T_{m;bc} \nabla_a \s7_{bcis} \\
    &+ \tfrac12 T_{m;bc} \Big( \nabla_a T_{a;bp} \s7_{pcis} + \nabla_a T_{a;cp} \s7_{bpis} +\nabla_a T_{a;ip} \s7_{bcps} +\nabla_a T_{a;sp} \s7_{bcip} \\
    & +T_{a;bp} \nabla_a\s7_{pcis} +  T_{a;cp} \nabla_a\s7_{bpis} + T_{a;ip} \nabla_a\s7_{bcps} + T_{a;sp} \nabla_a\s7_{bcip}\Big) \\
    =&\ 4(\Delta T)_{m;is} + \nabla_a T_{m;bc} \nabla_a \s7_{bcis} \\
    &+ \tfrac12 T_{m;bc} \Big( \nabla_a T_{a;bp} \s7_{pcis} + \nabla_a T_{a;cp} \s7_{bpis}  \\
    & +T_{a;bp} \nabla_a\s7_{pcis} +  T_{a;cp} \nabla_a\s7_{bpis} + T_{a;ip} \nabla_a\s7_{bcps} + T_{a;sp} \nabla_a\s7_{bcip}\Big) \\
    & +\tfrac12 T_{m;bc} \Big(\nabla_a T_{a;ip} \s7_{bcps} +\nabla_a T_{a;sp} \s7_{bcip}\Big)\\
    =&\ 4(\Delta T)_{m;is} + \nabla_a T_{m;bc} \nabla_a \s7_{bcis} \\
    &+ \tfrac12 T_{m;bc} \Big( \nabla_a T_{a;bp} \s7_{pcis} + \nabla_a T_{a;cp} \s7_{bpis}  \\
    & +T_{a;bp} \nabla_a\s7_{pcis} +  T_{a;cp} \nabla_a\s7_{bpis} + T_{a;ip} \nabla_a\s7_{bcps} + T_{a;sp} \nabla_a\s7_{bcip}\Big) \\
    & -3 T_{m;ps} \nabla_a T_{a;ip} -3 T_{m;ip} \nabla_a T_{a;sp}. 
\end{align*}

In conclusion,
\begin{align*}
    4\frac{\partial}{\partial t} T_{m;is} 
    =&\ 
    4(\Delta T)_{m;is} + \nabla_a T_{m;bc} \nabla_a \s7_{bcis} \\
    &+ \tfrac12 T_{m;bc} \Big( \nabla_a T_{a;bp} \s7_{pcis} + \nabla_a T_{a;cp} \s7_{bpis}  \\
    & +T_{a;bp} \nabla_a\s7_{pcis} +  T_{a;cp} \nabla_a\s7_{bpis} + T_{a;ip} \nabla_a\s7_{bcps} + T_{a;sp} \nabla_a\s7_{bcip}\Big) \\
    & -3 T_{m;ps} \nabla_a T_{a;ip} -3 T_{m;ip} \nabla_a T_{a;sp} \\
    & + 6\nabla_a T_{a;ip} T_{m;ps} - 2\nabla_a T_{a;sp} T_{m;pi} + 4 T_{q;is}R_{amaq}
- (\nabla_a R_{mais} - \tfrac12 \nabla_a R_{mapq} \Phi_{pqis})\\
    &+ T_{a;qs}R_{amiq} + T_{a;iq}R_{amsq}  
- 2 T_{a;ip}\nabla_a T_{m;sp} + 2 \nabla_a T_{m;ip}T_{a;sp}  + \tfrac18 R_{mapq} \nabla_a \Phi_{pqis}\\
    &  - T_{a;qc}R_{ambq}\Phi_{bcis} + 2 T_{a;bp}\nabla_a T_{m;cp} \Phi_{bcis} 
  - \tfrac{1}{16} R_{mapq} \nabla_a \Phi_{pqbc}\Phi_{bcis},
  \end{align*}
  so
  \begin{align*}
  4\frac{\partial}{\partial t} T_{m;is} 
  =&\ 
4(\Delta T)_{m;is} + \nabla_a T_{m;bc} \nabla_a \s7_{bcis} \\
    &+ \tfrac12 T_{m;bc} \Big( \nabla_a T_{a;bp} \s7_{pcis} + \nabla_a T_{a;cp} \s7_{bpis}  \\
    & +T_{a;bp} \nabla_a\s7_{pcis} +  T_{a;cp} \nabla_a\s7_{bpis} + T_{a;ip} \nabla_a\s7_{bcps} + T_{a;sp} \nabla_a\s7_{bcip}\Big) \\
    & + 3\nabla_a T_{a;ip} T_{m;ps} + \nabla_a T_{a;sp} T_{m;pi} + 4 T_{q;is}R_{amaq}
- (\nabla_a R_{mais} - \tfrac12 \nabla_a R_{mapq} \Phi_{pqis})\\
    &+ T_{a;qs}R_{amiq} + T_{a;iq}R_{amsq}  
- 2 T_{a;ip}\nabla_a T_{m;sp} + 2 \nabla_a T_{m;ip}T_{a;sp}  + \tfrac18 R_{mapq} \nabla_a \Phi_{pqis}\\
    &  - T_{a;qc}R_{ambq}\Phi_{bcis} + 2 T_{a;bp}\nabla_a T_{m;cp} \Phi_{bcis} 
  - \tfrac{1}{16} R_{mapq} \nabla_a \Phi_{pqbc}\Phi_{bcis}\\
    =&\ 4(\Delta T)_{m;is} + \nabla_a T_{m;bc} \nabla_a \s7_{bcis} \\
    &+ \tfrac12 T_{m;bc} \Big( 2\nabla_a T_{a;bp} \s7_{pcis}  + 2T_{a;bp} \nabla_a\s7_{pcis} + T_{a;ip} \nabla_a\s7_{bcps} + T_{a;sp} \nabla_a\s7_{bcip}\Big) \\
    & + 3\nabla_a T_{a;ip} T_{m;ps} + \nabla_a T_{a;sp} T_{m;pi} + 4 T_{q;is}R_{amaq}
- (\nabla_a R_{mais} - \tfrac12 \nabla_a R_{mapq} \Phi_{pqis})\\
    &+ T_{a;qs}R_{amiq} + T_{a;iq}R_{amsq}  
- 2 T_{a;ip}\nabla_a T_{m;sp} + 2 \nabla_a T_{m;ip}T_{a;sp}  + \tfrac18 R_{mapq} \nabla_a \Phi_{pqis}\\
    &  - T_{a;qc}R_{ambq}\Phi_{bcis} + 2 T_{a;bp}\nabla_a T_{m;cp} \Phi_{bcis} 
  - \tfrac{1}{16} R_{mapq} \nabla_a \Phi_{pqbc}\Phi_{bcis}.
\end{align*}
Inserting
\begin{align*}
\del_m\s7_{ijkl}=T_{m;ip}\s7_{pjkl}+T_{m;jp}\s7_{ipkl}+T_{m;kp}\s7_{ijpl}+T_{m;lp}\s7_{ijkp}    
\end{align*}
into first order terms, we get
\begin{align*}
  4\frac{\partial}{\partial t} T_{m;is} 
    =&\ 
  4(\Delta T)_{m;is} + \nabla_a T_{m;bc} \nabla_a \s7_{bcis} \\
&+ \tfrac12 T_{m;bc} \Big( 2\nabla_a T_{a;bp} \s7_{pcis}  
 + 2T_{a;bp} \nabla_a\s7_{pcis} + T_{a;ip} \nabla_a\s7_{bcps} + T_{a;sp} \nabla_a\s7_{bcip}\Big) \\
& + 3\nabla_a T_{a;ip} T_{m;ps} + \nabla_a T_{a;sp} T_{m;pi} + 4 T_{q;is}R_{amaq}
- (\nabla_a R_{mais} - \tfrac12 \nabla_a R_{mapq} \Phi_{pqis})\\
&+ T_{a;qs}R_{amiq} + T_{a;iq}R_{amsq}  
- 2 T_{a;ip}\nabla_a T_{m;sp} + 2 \nabla_a T_{m;ip}T_{a;sp}  + \tfrac18 R_{mapq} \nabla_a \Phi_{pqis}\\
    &  - T_{a;qc}R_{ambq}\Phi_{bcis} + 2 T_{a;bp}\nabla_a T_{m;cp} \Phi_{bcis} 
  - \tfrac{1}{16} R_{mapq} \nabla_a \Phi_{pqbc}\Phi_{bcis}\\
    =&\ 4(\Delta T)_{m;is} \\
  &+ \nabla_a T_{m;bc} \nabla_a \s7_{bcis} 
+ T_{m;bc} \nabla_a T_{a;bp} \s7_{pcis} + 3\nabla_a T_{a;ip} T_{m;ps} + \nabla_a T_{a;sp} T_{m;pi} \\
    & - 2 T_{a;ip}\nabla_a T_{m;sp} + 2 \nabla_a T_{m;ip}T_{a;sp} + 2 T_{a;bp}\nabla_a T_{m;cp} \Phi_{bcis} \\
    & + \tfrac12 T_{m;bc} \Big( 2T_{a;bp} \nabla_a\s7_{pcis} + T_{a;ip} \nabla_a\s7_{bcps} + T_{a;sp} \nabla_a\s7_{bcip}\Big) \\
    & + 4 T_{q;is}R_{amaq}
- (\nabla_a R_{mais} - \tfrac12 \nabla_a R_{mapq} \Phi_{pqis})\\
    &+ T_{a;qs}R_{amiq} + T_{a;iq}R_{amsq}  
  + \tfrac18 R_{mapq} \nabla_a \Phi_{pqis}  - T_{a;qc}R_{ambq}\Phi_{bcis} 
  - \tfrac{1}{16} R_{mapq} \nabla_a \Phi_{pqbc}\Phi_{bcis}.
\end{align*}
Moreover,
\begin{align*}
i) \quad 
    & \nabla_a T_{m;bc} \nabla_a \s7_{bcis}
    + T_{m;bc} \nabla_a T_{a;bp} \s7_{pcis} + 3\nabla_a T_{a;ip} T_{m;ps} + \nabla_a T_{a;sp} T_{m;pi} \\
    & - 2 T_{a;ip}\nabla_a T_{m;sp} + 2 \nabla_a T_{m;ip}T_{a;sp} + 2 T_{a;bp}\nabla_a T_{m;cp} \Phi_{bcis} \\
    =&\ \nabla_a T_{m;bc}\Big(T_{a;bp}\s7_{pcis}+T_{a;cp}\s7_{bpis}+T_{a;ip}\s7_{bcps}+T_{a;sp}\s7_{bcip} \Big) + T_{m;bc} \nabla_a T_{a;bp} \s7_{pcis} \\
    &+ 3\nabla_a T_{a;ip} T_{m;ps} + \nabla_a T_{a;sp} T_{m;pi} - 2 T_{a;ip}\nabla_a T_{m;sp} + 2 \nabla_a T_{m;ip}T_{a;sp} + 2 T_{a;bp}\nabla_a T_{m;cp} \Phi_{bcis} \\
    =&\ \nabla_a T_{m;bc}\Big(4T_{a;bp}\s7_{pcis}+T_{a;ip}\s7_{bcps}+T_{a;sp}\s7_{bcip} \Big) + T_{m;bc} \nabla_a T_{a;bp} \s7_{pcis} \\
    &+ 3\nabla_a T_{a;ip} T_{m;ps} + \nabla_a T_{a;sp} T_{m;pi} - 2 T_{a;ip}\nabla_a T_{m;sp} + 2 \nabla_a T_{m;ip}T_{a;sp}  ;
\end{align*}
\begin{align*}
ii) \quad 
    & 4 T_{m;bc} T_{a;bp} \nabla_a\s7_{pcis} +  \tfrac12 T_{m;bc}T_{a;ip} \nabla_a\s7_{bcps} + \tfrac12 T_{m;bc} T_{a;sp} \nabla_a\s7_{bcip} \\
    =&\ T_{m;bc} T_{a;bp} \Big(T_{a;pq}\s7_{qcis}+T_{a;cq}\s7_{pqis}+T_{a;iq}\s7_{pcqs}+T_{a;sq}\s7_{pciq} \Big) \\
    &+  \tfrac12 T_{m;bc}T_{a;ip} \Big(T_{a;bq}\s7_{qcps}+T_{a;cq}\s7_{bqps}+T_{a;pq}\s7_{bcqs}+T_{a;sq}\s7_{bcpq} \Big) \\
    & + \tfrac12 T_{m;bc} T_{a;sp} \Big(T_{a;bq}\s7_{qcip}+T_{a;cq}\s7_{bqip}+T_{a;iq}\s7_{bcqp}+T_{a;pq}\s7_{bciq} \Big) \\
    =&\ T_{m;bc} T_{a;bp} \Big(T_{a;pq}\s7_{qcis}+ T_{a;cq}\s7_{pqis}+2T_{a;iq}\s7_{pcqs}+
2T_{a;sq}\s7_{pciq} \Big) \\
    &+  \tfrac12 T_{m;bc}T_{a;ip} \Big(
T_{a;pq}\s7_{bcqs}+
2T_{a;sq}\s7_{bcpq} \Big) 
 + \tfrac12 T_{m;bc} T_{a;sp}
T_{a;pq}\s7_{bciq}.
\end{align*}
Gathering terms yields the Proposition.
\end{proof}

\begin{proposition}
\label{prop:evolnormT}
    If $\{\s7_t\}$ is a solution of the harmonic $\S7$-flow \eqref{eq: Har Spin(7) Flow} and $T$ is its torsion, then the evolution equation for $|T|^2$ is
\begin{align*}
  2\frac{\partial}{\partial t} |T|^2 &= 2\Delta |T|^2 - 4 |\nabla T|^2 
   + 16 T_{a;bp} T_{m;bc} T_{a;pq} T_{m;qc} 
   + 16 T_{a;bp} T_{m;bc} T_{a;cq} T_{m;pq} \\
  & \quad  + 16 T_{a;qs} T_{m;is} R_{amiq} 
  + 4 T_{q;is} T_{m;is} R_{amaq} 
  -4 T_{m;is} \nabla_a R_{mais}. 
\end{align*}
\end{proposition}

\begin{remark}
    Even though we obtained the doubling-time estimate in Corollary~\ref{dteforif} and the Shi-type estimates in Corollary~\ref{shitypeest} for \eqref{eq: Har Spin(7) Flow} as a consequence of the general theory for \eqref{eq: HSF}, we could also derive those estimates from Proposition~\ref{prop:evolnormT} and the maximum principle, as one does for the $\G2$ case in \cite[Prop. 3.2 \& Thm. 3.3]{dgk-isometric}. Hopefully the general theory will be able to spare people the effort of similar lengthy computations in the future. 
\end{remark}

We also have the following local derivative estimates for the torsion tensor along \eqref{eq: Har Spin(7) Flow}. The proof is similar to \cite[Theorem 3.7]{dgk-isometric} so we omit the details. We first define the parabolic cylinder of radius $r$ centred at $(x_0, t_0)$ as
\begin{align*}
P_r(x_0, t_0) = \{ (x,t)\in M\times \bR \mid d(x, x_0)\leq r,\ t_0-r^2\leq t\leq t_0\}.     
\end{align*}

\begin{corollary}[Local derivative estimates]\label{localderest}
Let $\s7(t)$ be a solution to \eqref{eq: Har Spin(7) Flow} on $M^8$. Let $x_0\in M$ and $t_0\geq 0$ such that $\s7(t)$ is defined at least up to $t_0$. There exist constants $s>0$ and $C_m$ for all $m\geq 1$ such that if $|T|\leq K$ and $|\del^j \Riem|\leq B_jK^{2+j}$ for all $j\geq 0$ in some parabolic cylinder $P_r(x_0, t_0)$ with $r\leq s$ and $K\geq \frac{1}{r^2}$, then we have
\begin{align}\label{eq:localderest}
|\del^mT|\leq C_mK^{m+1}    
\end{align}
in the smaller parabolic cylinder $P_{\frac{r}{2^m}}(x_0, t_0)$.
\end{corollary}

\subsection{Compactness}

\label{sec:compactness}
We want to prove a Hamilton-type compactness theorem for the solutions of \eqref{eq: Har Spin(7) Flow}. Apart from aiding in the understanding of singularity formation, it will also be useful in the applications of the monotonicity formula. 

\subsubsection{Cheeger-Gromov compactness  for $\S7$-structures}

We first prove a very general Cheeger-Gromov type compactness theorem for $\S7$-structures which, while believed by experts, has not appeared in the literature yet. 

\begin{definition}
\label{def:C-Gcomp}
    Let $\{M_i\}$ be a sequence of $8$-manifolds and let $p_i\in M_i$, for each $i$. Let $\s7_i$ be a $\S7$-structure on $M_i$ such that the associated Riemannian metric $g_i$ on $M_i$ are complete, for each $i$. Let $M$ be an $8$-manifold with $p\in M$ and $\s7$ a $\S7$-structure on $M$. We say that
\begin{align*}
(M_i, \s7_i, p_i) \rightarrow (M, \s7, p)\ \ \ \ \textup{as}\ i\rightarrow \infty    
\end{align*}
if there exists a sequence of compact subsets $\Omega_i\subset M$ exhausting $M$ with $p_i\in \textup{int}(\Omega_i)$ for each $i$ and a sequence of diffeomorphisms $F_i:\Omega_i\rightarrow F_i(\Omega_i)\subset M_i$ with $F_i(p)=p_i$ such that on every compact subset $K$ of $M$ and for every $\varepsilon >0$, there exists $i_0=i_0(\varepsilon)$ such that for $i\geq i_0$,
\begin{align*}
\underset{x\in K}{\textup{sup}}\ |\del^k(F_i^*\s7_i-\s7)|_{g_0} <\varepsilon,\ \ \ \ \ \ \ \ \textup{for}\ k\geq 0,   
\end{align*}
uniformly on $K$. Here the covariant derivative is with respect to any fixed metric $g_0$ on $M$. 
\end{definition}

The proof of the following compactness theorem is similar in spirit to the compactness theorems for complete pointed Riemannian manifolds \cite[Theorem 2.3]{hamilton-compactness}, and for $\G2$-structures by Lotay--Wei \cite[Theorem 7.1]{lotay-wei-gafa}.

\begin{theorem}
\label{thm:cptthm}
    Let $(M_i^8,\Phi_i,g_i)$ be a sequence of smooth manifolds with $\S7$-structures and with the induced Riemannian metrics, and suppose that, for each $i$, we have $p_i\in M_i$ and $g_i$ is complete. Suppose further, that
\begin{align}
\label{eq:cptpf1}
    \underset{i}{\textup{sup}}\ \underset{x\in M_i}{\textup{sup}}\ \left( |\del^k_{g_i}\Riem_{g_i}(x)|^2_{g_i} + |\del^{k+1}_{g_i}T_{i}(x)|^2_{g_{i}}   \right)^{\frac 12} < \infty,
    \quad\forall \ k\geq 0,
\end{align}
    where $T_i$, $\Riem _{g_i}$ are the torsion of $\s7_i$ and the curvature of $g_i$, respectively, and the injectivity radius of $(M_i, g_i)$ at $p_i$ satisfies
\begin{align*}
    \textup{inf}\ inj(M_i, g_i, p_i)>0,
    \quad\forall\  i.
\end{align*}
Then there exist a manifold $(M^8,\Phi)$ with a $\S7$-structure and a point $p\in M$ such that, subsequentially,
\begin{align*}
 (M_i, \s7_i, p_i) \rightarrow (M, \s7, p),
 \quad as\ i\rightarrow \infty.   
\end{align*}
\end{theorem}

\begin{proof}
Throughout the proof, we will continue to use the index $i$, even after taking a subsequence. By \eqref{eq:cptpf1}, we have a uniform bound on the Riemann curvature tensor, and the injectivity radius is positive, so the Cheeger--Gromov compactness theorem for complete pointed Riemannian manifolds \cite[Theorem 2.3]{hamilton-compactness} yields a complete Riemannian manifold $(M^8,g)$ and a point $p\in M$ such that, after passing to a subsequence,
\begin{align}
\label{eq:cptpf2}
    (M_i, g_i, p_i) \rightarrow (M, g, p),
    \quad \textup{as}\ i\rightarrow \infty.
\end{align}
Explicitly, there exist nested compact sets $\Omega_i\subset M$ exhausting $M$, with $p\in \textup{int}(\Omega_i)$ for each $i$, and diffeomorphisms $F_i:\Omega_i\rightarrow F_i(\Omega_i)\subset M_i$ with $F_i(p)=p_i$, such that, on any compact subset $K\subset M$  and for every $\varepsilon >0$, there exists $i_0=i_0(\varepsilon)$ for which the following hold uniformly on $K$:
\begin{align*}
    \underset{x\in K}{\textup{sup}}\ |\del^k(F_i^*g_i-g)|_{g_0} <\varepsilon,
    \qforq k\geq 0  
    \qandq i\geq i_0,
\end{align*}

Fixing $i$ sufficiently large, let us work on the compact subset $\Omega_i\subset M$. Since $\{\Omega_i\}$ is nested, for each  $j\geq 0$ we have $\Omega_i\subset \Omega_{i+j}$ and a diffeomorphism $F_{i+j}:\Omega_{i+j}\rightarrow F_{i+j}(\Omega_{i+j})\subset M_{i+j}$. Define the restricted diffeomorphism 
\begin{align*}
    F_{i,j}=\left.F_{i+j}\right|_{\Omega_i} :\Omega_i\rightarrow F_{i+j}(\Omega_i)\subset M_{i+j},
\quad\forall\  j\geq 0.
\end{align*}
The convergence in \eqref{eq:cptpf2} implies that the sequence $\{g_{i,j}=F_{i,j}^*g_{i+j} \}_{j=0}^{\infty}$ of Riemmanian metrics on $\Omega_i$ converges to $g_{i, \infty}=g$ on $\Omega_i$, as $j\rightarrow \infty$.

Let $\del,\ \del_{i,j}$ be the Levi-Civita connections of $g$ and $g_{i,j}$ on $\Omega_i$, respectively. Let $h=g-g_{i,j}$ and $A=\del-\del_{i,j}$ be the difference of the metrics and their connections on $\Omega_i$, respectively. Then locally,
\begin{align*}
A^c_{ab}=\frac{1}{2}(g_{i,j})^{cd}(\del_ah_{bd}+\del_bh_{ad}-\del_dh_{ab}).     
\end{align*}
Since $g_{i,j}\rightarrow g$ smoothly on $\Omega_i$, as $j\rightarrow \infty$, $g_{i,j}$ and $g$ are equivalent for large enough $j$,  and hence we can take the norms with respect to either of these metrics. Moreover, $|\del^kh|_g\rightarrow 0$ as $j\rightarrow \infty$, for all $k\geq 0$. Hence, $A$ is uniformly bounded with respect to $g$ for all large $j$. Since $\del g=0$,
\begin{align*}
    \del^kA^c_{ab}&=\frac{1}{2}\sum_{l=1}^k \del ^{(k+1-l)}(g_{i,j})^{cd}(\del^l\del_ah_{bd}+\del^l\del_bh_{ad}-\del^l\del_dh_{ab}) \\
&=-\frac 12 \sum_{l=1}^k \del ^{(k+1-l)}(g^{cd}-(g_{i,j})^{cd})(\del^l\del_ah_{bd}+\del^l\del_bh_{ad}-\del^l\del_dh_{ab}). 
\end{align*}
Thus there exist constants $c_k$ for $k\geq 0$ such that $|\del^kA|_g\leq c_k$ for all $j\geq 0$.

Using the diffeomorphisms $F_{i,j}$, we can define $\S7$-structures $\s7_{i,j}=F^*_{i,j}\s7_{i+j}$ on $\Omega_i$ by pulling back the $\S7$-structure $\s7_{i+j}$ on $M_{i+j}$. We next estimate $|\del^k\s7_{i,j}|_g$. Since $g$ and $g_{i,j}$ are equivalent for large $j$, $|\s7_{i,j}|_g\leq c_0|\s7_{i,j}|_{g_{i,j}}\leq 8c_0$ for some constant $c_0$. Next, observe trivially that
\begin{align*}
\del \s7_{i,j}=\del_{i,j}\s7_{i,j}+(\del-\del_{i,j})\s7_{i,j},    
\end{align*}
and since $A$ is uniformly bounded, there is a constant $c_1$ such that 
\begin{align*}
|\del \s7_{i,j}|_g\leq c_0|\del_{i,j}\s7_{i,j}|_{g_{i,j}}+C|A|_g |\s7_{i,j}|_g\leq c_1. 
\end{align*}
Similarly, for $\del^2{\s7_{i,j}}$, we have
\begin{align*}
\del^2\s7_{i,j}=\del^2_{i,j}\s7_{i,j}+(\del-\del_{i,j})\del_{i,j}\s7_{i,j}+(\del(\del-\del_{i,j}))\s7_{i,j},    
\end{align*}
and so, since $A$, $\del A$ are uniformly bounded, there is a constant $c_2$ such that
\begin{align*}
|\del^2\s7_{i,j}|_{g}\leq C|\del^2_{i,j}\s7_{i,j}|_{g_{i,j}}+C|A|_g|\del_{i,j}\s7_{i,j}|_g+C|\del A|_g|\s7|_g \leq c_2.    
\end{align*}
For $k\geq 2$, we have the estimate
\begin{align*}
|\del^k \s7_{i,j}|_g\leq C\sum_{l=0}^k|A|^l_g|\del^{k-l}_{i,j}\s7_{i,j}|_{g_{i,j}}+C\sum_{l=1}^{k-1}|\del^lA|_g|\del^{k-1-l}\s7_{i,j}|_g.     
\end{align*}
By the assumption \eqref{eq:cptpf1}, covariant derivatives of the torsion $T_i$ of all orders are uniformly bounded, so using \eqref{eq:Texpress} and the estimate $|\del^kA|_g\leq c_k$, by an induction argument, we can show the existence of constants $c_k$ for $k\geq 0$ such that $|\del^k\s7_{i,j}|_g\leq c_k$ on $\Omega_i$ for all $j, k\geq 0$. 

Since we have uniform bounds on covariant derivatives of all orders of $\s7_{i,j}$, the Arzel\`a-Ascoli Theorem (see eg. \cite[Corollary 9.14]{andrews-hopper}) implies that there exist a $4$-form $\s7_{i, \infty}$ and a subsequence of $\s7_{i,j}$ in $j$, which we still denote by $\s7_{i,j}$, that converges to $\s7_{i, \infty}$ smoothly on $\Omega_i$. In other words,
\begin{align}
\label{eq:cptpf3}
    |\del^k(\s7_{i,j}-\s7_{i, \infty})|_g\rightarrow 0, 
    \quad \textup{as}\ j\rightarrow \infty, 
\end{align}
uniformly on $\Omega_i$, for all $k\geq 0$.

We prove that $\s7_{i, \infty}$ is a $\S7$-structure on $\Omega_i$. Since each $\s7_{i,j}$ is a $\S7$-structure on $\Omega_i$ with associated metric $g_{i,j}$, we have the following relation. Let $u, v, w$ be any vector fields on $\Omega_i\subset M$, then
\begin{align*}
    (u\lrcorner v\lrcorner \s7_{i,j})\wedge (u\lrcorner w\lrcorner \s7_{i,j}) \wedge \s7_{i,j} = 6(g_{i,j}(u,u)g_{i,j}(v,w)-g_{i,j}(u,v)g_{i,j}(u,w))\vol_{i,j}.     
\end{align*}
Letting $j\rightarrow \infty$ in the previous equation gives
\begin{align}
\label{eq:cptpf4}
    (u\lrcorner v\lrcorner \s7_{i,\infty})\wedge (u\lrcorner w\lrcorner \s7_{i,\infty}) \wedge \s7_{i,\infty} = 6(g_{i,\infty}(u,u)g_{i,\infty}(v,w)-g_{i,\infty}(u,v)g_{i,\infty}(u,w))\vol_{i,\infty}.    
\end{align}
Since, by the Cheeger-Gromov compactness Theorem, the limit $g_{i, \infty}$ is a Riemannian metric on $\Omega_i$ and $\vol_{i, \infty}$ is its volume form, \eqref{eq:cptpf4} implies that $\s7_{i, \infty}$ is an admissible or a non-degenerate $4$-form and hence it is a $\S7$-structure on $\Omega_i$ with the associated metric $g_{i, \infty}=g$.

We now go from $\Omega_i$ to $M$. Since $\{\Omega_i\}$ is nested, we denote the inclusion map of $\Omega_i$ into $\Omega_n$ by
\begin{align*}
    I_{in}:\Omega_i\rightarrow \Omega_n,
    \qforq n\geq i.     
\end{align*}
As before, for each $\Omega_n$, we have $g_{n,j}$, $\s7_{n,j}$ which, after passing to a subsequence, converge to $g_{n, \infty}$, $\s7_{n, \infty}$ respectively as $j\rightarrow \infty$. By definition,
\begin{align*}
    I^*_{in}g_{n,j}=g_{i,j}
    \qandq 
    I^*_{in}\s7_{n,j}=\s7_{i,j}.  
\end{align*}
Note that the pull-back $I^*_{in}$ is independent of $j$, hence by taking $j\rightarrow \infty$ in the above expression, we get
\begin{align}
\label{eq:cptpf5}
    I^*_{in}g_{n,\infty}= g_{i,\infty}
    \qandq  
    I^*_{in}\s7_{n,\infty}=\s7_{i,\infty}.  
\end{align}
Thus, from \eqref{eq:cptpf5}, there exists a $4$-form $\s7$ on $M$, which is a $\S7$-structure with associated metric $g$ such that
\begin{align}\label{eq:cptpf6}
I^*_ig=g_{i, \infty}\qandq  I^*_{i}\s7 = \s7_{i, \infty}, 
\end{align}
where $I_i:\Omega_i\rightarrow M$ is the inclusion map.

Finally, we show that $(M_i, \s7_i, p_i)$ converges to $(M, \s7, p)$. Let $K$ be any compact subset of $M$. Since $\{\Omega_i\}$ exhausts $M$, there exists $i_0$ such that $K$ is contained in $\Omega_i$ for all $i\geq i_0$. Fix an $i$ such that $K\subset \Omega_i$. Then by \eqref{eq:cptpf3}, on $K$ we have
\begin{align*}
|\del^k(F^*_{l}\s7_{l}-\s7)|_g &= |\del^k(F^*_{i+j}\s7_{i+j}-\s7)|_g,\ \ \ \ \textup{where}\ l=i+j, \\
&=|\del^k(\s7_{i,j}-\s7_{i, \infty})|_g \rightarrow 0\ \ \ \ \textup{as}\ l\rightarrow \infty
\end{align*}
for all $k\geq 0$. This completes the proof.
\end{proof}
We note that, if all the metrics in the sequence $(M_i, \s7_i, g_i)$ are the same, then the limiting $\S7$-structure $\s7$ induces the same metric. 

\subsubsection{Compactness for the harmonic $\S7$-flow}

We now state and prove the compactness theorem for solutions of the harmonic $\S7$-flow \eqref{eq: Har Spin(7) Flow}.
\begin{theorem}
\label{thm:flowcptthm}
    Let $M_i$ be a sequence of compact $8$-manifolds and let $p_i\in M_i$ for each $i$. Let $\{\s7_i(t)\}$ be a sequence of solutions to the harmonic $\S7$-flow \eqref{eq: Har Spin(7) Flow} on $M_i$ for $t\in (a,b)$, with $-\infty\leq a<0<b\leq \infty$. Suppose that 
\begin{align}
\label{eq:fcptthm1}
    \underset{i}{\textup{sup}}\ \underset{x\in M_i, t\in (a,b)}{\textup{sup}} |T_i(x,t)|_{g_i} < \infty,
\end{align}
    where $T_i$ denotes the torsion of $\s7_i(t)$, and the injectivity radius satisfies
\begin{align}
\label{eq:fcptthm2}
    \underset{i}{\textup{inf}}\ inj (M_i, g_i(0), p_i) >0.    
\end{align}
    Suppose further that there are uniform constants $C_k$, for all $k \geq 0$, such that
\begin{align}
\label{eq:fcptthm3}
    \underset{i}{\textup{sup}}\ |\del^k\Riem_i|_{g_i} \leq C_k.    
\end{align}
    Then there exist an $8$-manifold $M$, a point $p\in M$ and a solution $\s7(t)$ of the flow \eqref{eq: Har Spin(7) Flow} on $M$ for $t\in (a,b)$ such that, after passing to a subsequence, 
\begin{align*}
    (M_i, \s7_i(t), p_i) \longrightarrow (M, \s7(t), p),
    \quad \textup{as } i\rightarrow \infty.    
\end{align*}
\end{theorem}

\begin{proof}
The proof is similar in spirit to that of the compactness theorem for the isometric flow of $\G2$-structures in \cite[Theorem 3.13]{dgk-isometric}. The
idea is to show that the bounds on the $\S7$-structure and on its covariant derivatives and time derivatives at time $t = 0$ extend to subsequent times, in the presence of additional bounds on the torsion and its derivatives for all time. We give the details below.

From the derivative estimates Corollary~\ref{shitypeest} and~\eqref{eq:fcptthm1}, we have
\begin{equation} \label{eq:fcptpf1}
    |\del^m_{g_i(t)}T_i(x,t)|
    \leq C_m.
\end{equation}
Since each $M_i$ is compact, we know $|\Riem_i|_{g_i}$ is bounded, and assumption~\eqref{eq:fcptthm2} allows us to use Theorem~\ref{thm:cptthm} for $t=0$ to extract a subsequence of $(M_i, \s7_i(0), p_i)$ which converges to a complete limit $(M, \s7_{\infty}(0), p)$. That is, there exist compact subsets $\Omega_i\subset M$ exhausting $M$ with $p\in \text{int}(\Omega_i)$, for each $i$, and diffeomorphisms $F_i:\Omega_i \rightarrow F_i(\Omega_i)\subset M_i$ with $F_i(p)=p_i$, such that $F_i^*g_i(0)\rightarrow g_{\infty}(0)$ and $F_i^*\s7_i(0)\rightarrow \s7_{\infty}(0)$ smoothly on any compact subset $\Omega\subset M$ as $i\rightarrow \infty$. 

Fix a compact subset $\Omega \times [c,d]\subset M\times (a,b)$ and let $i$ be sufficiently large so that $\Omega \subset \Omega_i$. Let $\bar{g}_i(t)=F_i^*g_i(t)$. Now since $\s7_{i}(t)$ are all solutions to the harmonic $\S7$-flow, we have $g_i(t)=g_i(0)$ for each $i$. Thus we trivially have
\begin{equation*}
    \sup_{\Omega\times [c,d]} |\del^m_{\bar{g}_i(0)}\bar{g}_i(t)|_{\bar{g}_i(0)} = 0.
\end{equation*} 
Since the limit metric $g_{\infty}(0)$ is uniformly equivalent to $g_i(0)$, we get
\begin{equation*}
    \sup_{\Omega\times [c,d]} |\del^m_{\bar{g}_i(\infty)}\bar{g}_i(t)|_{\bar{g}_i(\infty)} \leq C_m
\end{equation*}
for some positive constants $C_m$ and similarly
\begin{equation*}
    \sup_{\Omega \times [c,d]} \Big |\frac{\pt^l}{\pt t^{l}} \del^m_{\bar{g}_{\infty}(0)}\bar{g}_i(t) \Big |_{\bar{g}_{\infty}(0)} \leq C_{m,l}
\end{equation*} 
for some positive constants $C_{m,l}$.

Now let $\bar{\s7}_i(t)=F_i^*\s7_i(t)$. Then $\bar{\s7}_i(t)$ is a sequence of solutions of the harmonic $\S7$-flow on $\Omega \subset M$ for $t\in [c,d]$. Using~\eqref{eq:fcptpf1} and proceeding in a similar way as in \cite[Claim 3.9]{dgk-isometric} for $\S7$-structures, we deduce that there exist constants $C_m$, independent of $i$, such that
\begin{equation}
    \sup_{\Omega\times [c,d]} |\del^m_{\bar{g}_i(0)}\bar{\s7}_i(t)|_{\bar{g}_i(0)} \leq C_m
\end{equation}
and since $\bar{g}_i(0)$ and $\bar{\s7}(0)$ converge uniformly to $g_{\infty}(0)$ and $\bar{\s7}_{\infty}(0)$ with all their covariant derivatives on $\Omega$, we have
\begin{equation}
    \sup_{\Omega\times [c,d]} |\del^m_{\bar{g}_{\infty}(0)}\bar{\s7}_i(t)|_{\bar{g}_{\infty}(0)} \leq C_m.
\end{equation}
Moreover, because the time derivatives can be written in terms of spatial derivatives using the evolution equations of the harmonic $\S7$-flow, we get, for some uniform constants $C_{m,l}$, 
\begin{equation}
    \sup_{\Omega \times [c,d]} \Big |\frac{\pt^l}{\pt t^{l}} \del^m_{\bar{g}_{\infty}(0)}\bar{\s7}_i(t) \Big |_{\bar{g}_{\infty}(0)} \leq C_{m,l}.
\end{equation} 
It now follows from the Arzel\`a--Ascoli theorem that there exists a subsequence of $\bar{\s7}_i(t)$ that converges smoothly on $\Omega \times [c,d]$. A diagonal subsequence argument then produces a subsequence that converges smoothly on any compact subset of $M\times (a,b)$ to a solution $\bar{\s7}_{\infty}(t)$ of the isometric flow.
\end{proof}

\section{Monotonicity, entropy and $\varepsilon$-regularity of the $\S7$-flow}
\label{sec: analysis of evolution}

In this section, just like in the $\G2$ case, we first introduce a quantity $\Theta$ that is \emph{almost monotonic} along \eqref{eq: Har Spin(7) Flow}. We derive the evolution of $\Theta$ in Lemma \ref{lem: d/dt of Theta} below,  and use it to prove the almost monotonicity formula. We then introduce an \emph{entropy functional}, and use an $\varepsilon$-regularity theorem to prove that small initial entropy implies long-time existence and convergence to a critical point. The proofs of several results in this section are similar to the corresponding theorems for $\G2$-structures  in \cite{dgk-isometric}, so in some cases we will only sketch the proofs and refer the reader to that source for further details.

\subsection{Almost monotonicity formula}

Let $(M, g)$ be a complete Riemannian manifold. For $x_0\in M$, let $u$ be the fundamental solution of the backward heat equation, starting with the delta function at $x_0$ \cite{Hamilton1993}:
\begin{align*}
    & \Big( \frac{\partial}{\partial t} + \Delta \Big) u = 0, \quad \lim_{t \to t_0} u = \delta_{x_0}
\end{align*}
and set $u= \frac{e^{-f}}{\big( 4\pi (t_0 -t) \big)^4}.$
For a solution $\{\s7(t)\}_{t\in [0, t_0)}$ of the harmonic $\S7$-flow on $(M,g)$, we define the function
\begin{align}\label{thetadefn}
  \Theta_{(x_0, t_0)} (\s7(t)) = (t_0 -t) \int_M |T_{\s7(t)}|^2 u \, \vol.  
\end{align}
The following lemma is a direct adaptation of \cite[Lemma 5.2]{dgk-isometric}.
\begin{lemma} 
\label{lem: d/dt of Theta}
\begin{align}
    \frac{d}{d t} \Theta_{(x_0, t_0)}(\s7(t)) 
    =& - 2(t_0 -t) \int_M  |\Div T - \del f \lrcorner T|^2 u \vol \notag\\
    & - 2(t_0 -t) \int_M  \big( \nabla_m \nabla_l u - \frac{\nabla_m u \nabla_l u}{u}+ \frac{ug_{ml}}{2(t_0 -t)} \big) T_{m;is}T_{l;is} \vol  \notag\\
    & - (t_0 -t) \int_M u R_{mlis} (2T_{l;ir}T_{m;rs} - 2 T_{m;ir}T_{l;rs} + \tfrac{1}{4} R_{mlis} - \tfrac18 R_{mlab} \s7_{abis} ) \notag\\
    & - 2(t_0 -t) \int_M T_{m;is}  u \nabla_l R_{mlis} \vol. \label{eqlemma4.4}
\end{align}
\end{lemma}

\begin{proof}
By direct computation, and using \eqref{eq: dT/dt}, we have
\begin{align*}
   \frac{d}{d t} \Theta 
    =&\  \int_M (t_0 -t)  u \frac{\partial}{\partial t} |T|^2    - |T|^2 u    +    (t_0 -t) |T|^2 \frac{\partial}{\partial t} u \,  \vol\\
    =&\ \int_M 2(t_0 -t)  u T_{m;is}\Big(\nabla_r T_{r;ip} T_{m;ps} - \nabla_r T_{r;sp} T_{m;pi} + \pi_7 (\nabla_m (\nabla_r T_{r;is})) )\Big) \\   
    &- |T|^2 u - (t_0 -t) |T|^2 \Delta u  \, \vol,
\end{align*}
but $ T_{m;is} \nabla_r T_{r;ip} T_{m;ps} =0$, $ T_{m;is} \nabla_r T_{r;sp} T_{m;pi} =0$ and $T_{m;is} \in \Omega^{2}_{7}$, therefore
$$
\frac{d}{d t} \Theta = \int_M 2(t_0 -t)  u T_{m;is} (\nabla_m (\nabla_r T_{r;is})) - |T|^2 u - (t_0 -t) |T|^2 \Delta u  \, \vol.
$$
Integrating by parts and using the $\S7$-Bianchi identity \eqref{spin7bianchi}, 
we have
\begin{align*}
   \frac{d}{d t} \Theta 
   =&\ \int_M - 2(t_0 -t) \Big( |\Div T|^2 u + \nabla_m u  T_{m;is}  \nabla_r T_{r;is}\Big) - |T|^2 u  + 2 (t_0 -t) \Big(  T_{m;is} \nabla_l u \nabla_m T_{l;is} \\
    &  + 2 T_{m;is} \nabla_l u T_{l;ir}T_{m;rs} - 2 T_{m;is} \nabla_l u T_{m;ir}T_{l;rs} + \tfrac{1}{4} T_{m;is} \nabla_l u R_{mlis} - \tfrac18 T_{m;is} \nabla_l u R_{mlab} \s7_{abis} \Big) \, \vol, 
\end{align*}
but, as before, 
$$
T_{m;is} T_{l;ir}T_{m;rs} =0 ,
\quad T_{m;is} T_{m;ir}T_{l;rs} =0,
\qandq T_{m;is} (\tfrac{1}{4} R_{mlis} - \tfrac18 R_{mlab} \s7_{abis}) = T_{m;is} R_{mlis},
$$
so
\begin{align*}
    \frac{d}{d t} \Theta 
    =&\ 
    \int_M - 2(t_0 -t) \Big( |\Div T|^2 u + \nabla_m u  T_{m;is}  \nabla_r T_{r;is}\Big) - |T|^2 u  \\ 
    & + 2 (t_0 -t) \Big(  T_{m;is} \nabla_l u \nabla_m T_{l;is}
 + T_{m;is} \nabla_l u R_{mlis} \Big) \, \vol  \\
    =&\ \int_M - 2(t_0 -t) \Big( |\Div T|^2 u - 2 \langle \Div T , \nabla f \lrcorner T \rangle  u \Big)  \\
    & - 2(t_0 -t) \int_M  \big( \nabla_m \nabla_l u + \frac{ug_{ml}}{2(t_0 -t)} \big) T_{m;is}T_{l;is}\\
    & - (t_0 -t) \int_M u R_{mlis} (2T_{l;ir}T_{m;rs} - 2 T_{m;ir}T_{l;rs} + \tfrac{1}{4} R_{mlis} - \tfrac18 R_{mlab} \s7_{abis} ) \\
    & - 2(t_0 -t) \int_M T_{m;is}  u \nabla_l R_{mlis} \vol.
    \qedhere
\end{align*}
\end{proof}

\begin{theorem}[almost monotonicity formula]
\label{thm:almostmon}
Let $\{\Phi(t)\}$ be a solution of the harmonic $\S7$-flow \eqref{eq: Har Spin(7) Flow} on $(M^8,g)$.
\begin{enumerate}
    \item If $M$ is compact, then, for any $0<\tau_1 < \tau_2 < t_0$, there exist $K_1$, $K_2>0$ depending only on the geometry of $(M,g)$ such that
$$
\Theta(\Phi(\tau_2)) \leq K_1 \Theta(\Phi(\tau_1)) + K_2 (\tau_1 -\tau_2) (E(0) +1) .
$$
    \item When $(M,g) = (\bR^8, g_{\mathrm{Eucl}})$, then, for any $x_0 \in \bR^8$ and $0\leq \tau_1 < \tau_2$ we have
$$
\Theta(\Phi(\tau_2)) \leq \Theta(\Phi(\tau_1)) .
$$
\end{enumerate}
\end{theorem}

\begin{proof}
We sketch the proof following \cite[Theorem 5.3]{dgk-isometric} and rely on Lemma~\ref{lem: d/dt of Theta}.
\begin{enumerate}
    \item  The third and fourth terms of  \eqref{eqlemma4.4} are bounded by
$$
C( 1 + \Theta(\Phi(t))),
$$
due to the bounded geometry of $(M,g)$, Young's inequality and $\int_M  u \, \vol =1$.

For the second term of~\eqref{eqlemma4.4}, use \cite{hamilton-compactness} and the decreasing of 
$E(\Phi(t))$ along the harmonic $\S7$-flow to bound it by
$$
C \big(E(\Phi(0)) + \log \frac{B}{(t_0 -t)^4} \Theta(\Phi(t))\big),
$$
so that
\begin{align*}
\frac{d}{dt} \Theta(\Phi(t))  \leq& - 2(t_0 -t) \int_M  |\Div T - \del f \lrcorner T|^2 u \\
& +C_1 \Big(1 + \log \big( \frac{B}{(t_0 -t)^4}\big) \Big)\Theta(\Phi(t)) 
+ C_2 ( 1 + E(\Phi(0))) \vol.
\end{align*}
To control the logarithmic term, let $\xi(t)$ be any function satisfying
$$
\xi'(t) = 1 + \log \frac{B}{(t_0 - t)^4}.
$$
The claim is then obtained by integration over  $[t_0 -1 , t_0[\,$ of
$$
\frac{d}{dt} \Big[ e^{-C_1 \xi(t)} \Theta(\Phi(t))\Big] \leq K \big(E(\Phi(0)) + 1 \big).
$$

    \item On $(M^8,g) = (\bR^8, g_{\mathrm{Eucl}})$, the backward heat kernel is
$$
u(x,t) = \frac{1}{(4\pi(t_0-t))^4}
    \exp\left\{-\frac{|x-x_0|^2}{4(t_0-t)}\right\}
$$
so indeed
$\frac{d}{dt} \Theta(\Phi(t))  \leq 0.$\qedhere
\end{enumerate}
\end{proof}
One needs to introduce the quantity $\Theta$ because the energy functional $E$ is not scale-invariant, hence it is not sufficient to control the small-scale behaviour of a $\S7$-structure. However, we would like to prove long-time existence and convergence of the flow when the initial energy is bounded (which is a reasonable assumption) and hence will need to bound $\Theta(\s7(t))$ in terms of the initial energy $E(\s7(0))$ for $t<t_0$. This fails for small times $t_0$, which can be observed from \eqref{thetadefn}. Thus, we require a quantity sturdier than $\Theta$.
Motivated by analogous functionals for the mean curvature flow \cite{colding-minicozzi}, the high dimensional Yang--Mills flow \cite{kelleher}, the harmonic map heat flow \cite{kelleher-streets} and the isometric flow of $\G2$-structures \cite{dgk-isometric}, we define the following \emph{entropy} functional.

\begin{definition}
Let $(M^8, \s7, g)$ be a compact manifold with a $\S7$-structure. Let $u_{(x,t)}(y,s)=u^g_{(x,t)}(y,s)$ be the backwards heat kernel, starting from $\delta{(x,t)}$ as $s\rightarrow t$. For $\sigma>0$ we define
\begin{align}\label{eq:entropyeqn}
\lambda(\s7, \sigma) = \underset{(x,t)\in M\times (0, \sigma]}{\textup{max}} \left\{ t\int_M |T_{\s7}|^2(y) u_{(x,t)}(y,0) \vol\right \}.     
\end{align}
\end{definition}
One should think of $\sigma$ as the ``scale'' at which we are analyzing the flow. Since $M$ is compact, the maximum in \eqref{eq:entropyeqn} is achieved.
We need the next lemma for the proof of the $\varepsilon$-regularity theorem. We skip the proof, as it is similar to the $\G2$ case \cite[Lemma 5.6]{dgk-isometric} and uses the analysis on the backwards heat kernel from \cite{Hamilton1993}.

\begin{lemma}
\label{lem:thetainq}
    Let $(M^8,g)$ be compact and $E_0>0$. For every $\varepsilon >0$, there exist $\delta=\delta (\varepsilon, g, E_0)$ and $\bar{r}=\bar{r}(\varepsilon, g, E_0)>0$ such that the following holds:
    
    Suppose $\{\s7(t)\}_{t\in [0, t_0)}$ is a solution of the harmonic $\S7$-flow \eqref{eq: Har Spin(7) Flow}, with induced metric $g$, satisfying $E(\s7(0))\leq E_0$. Whenever
$$
\Theta_{(x_0, t_0)}(\s7(t_1))<\delta, \quad \text{for some } (x_0,t_1)\in M\times [0, t_0),
$$
    then, setting $r:=\textup{min}(\bar{r}, \sqrt{t_0-t_1})$, we have
$$
    \Theta_{(x,t)}(\s7(t_1))< \varepsilon, \quad\forall \ 
    (x,t)\in B(x_0, r)\times [t_0-r^2, t_0].
$$    


\end{lemma}

\medskip

We can now state the $\varepsilon$-regularity theorem for the harmonic $\S7$-flow.

\begin{theorem}[$\varepsilon$-regularity]\label{thm:epsilonreg}
    Let $(M^8,g)$ be compact and $E_0>0$. There exist $\varepsilon,\ \bar{\rho}>0$ such that, for every $\rho \in (0, \bar{\rho}]$, there exist $r\in (0, \rho)$ and $C<\infty$ such that the following holds:
    
    Suppose $\{\s7(t)\}_{t\in [0, t_0)}$ is a solution of the harmonic $\S7$-flow \eqref{eq: Har Spin(7) Flow}, with induced metric $g$, satisfying $E(\s7(0))\leq E_0$. Whenever 
\begin{align*}
    \Theta_{(x_0, t_0)}(\s7(t_0-\rho^2)) < \varepsilon,    
    \quad \text{for some }  x_0\in M,
\end{align*}
then, setting $\Lambda_r(x,t) = \textup{min}\ \left(1-r^{-1}d_g(x_0, x), \sqrt{1-r^{-2}(t_0-t)}    \right)$, we have
\begin{align*}
    \Lambda_r(x, t)|T_{\s7}(x,t)|\leq \frac{C}{r},    
    \quad\forall \ 
    (x,t)\in B(x_0, r)\times [t_0-r^2, t_0].
\end{align*}
\end{theorem}

\begin{proof}
Since the proof is similar to the $\G2$ case in \cite[Theorem 5.7]{dgk-isometric}, we only sketch the main ideas. The proof is by contradiction. If the claim were not true, then for any sequences $\varepsilon_i\rightarrow 0$ and $\bar{\rho}_i\rightarrow 0$, there would exist $\rho_i\in (0, \bar{\rho}_i]$ such that, for any $r_i\in (0, \rho_i)$ and $C_i\rightarrow \infty$, there exist counterexamples $\{(M, \s7_i(t))\}_{t\in [0, t_i)}$ with $g_{\s7_i(t)}=g$, $E(\s7(0))\leq E_0$, and $x_i\in M$, such that
\begin{align*}
\Theta_{(x_i, t_i)}(\s7_i(t_i-{\rho_i}^2))< \varepsilon_i,    
\end{align*}
while 
\begin{align}
r_i\left(\underset{(x,t)\in B(x_i, r_i)\times [t_i-{r_i}^2, t_i]}{\textup{max}} \Lambda_{r_i}(x,t)|T_{\s7_i}(x,t)|   \right)> C_i.     
\end{align}
Let $(\bar{x},\bar{t})\in B(x_i, r_i)\times [t_i-{r_i}^2, t_i]$ be the point where the maximum is attained and set $Q_i:= |T_{\s7_i}(\bar{x_i}, \bar{t_i})|$.

The idea now is to use the rescaled flow (see \textsection\ref{sec:pscale})
\begin{align*}
    \widetilde{\s7_i}(t) &= {Q_i}^4\s7_i(\bar{t_i}+t{Q_i}^{-2}),
    \quad t\leq 0,\\
    g_i&= {Q_i}^2g,
\end{align*}
and the pointed sequence $\{(M, \widetilde{\s7_i}(t), g_i, \bar{x}_i)\}$. By the definition of $Q_i$, we have 
$$
    |T_{\widetilde{\s7_i}}(x, t)|\leq \frac{\Lambda_i(\bar{x_i}, \bar{t_i})}{\Lambda_i(x, \bar{t_i}+t{Q_i}^{-2})} \leq 2   
\qandq
    |T_{\widetilde{\s7_i}}(\bar{x}_i, 0)|=1.
$$ 
We can then apply the derivative estimates and the compactness Theorem~\ref{thm:flowcptthm} to deduce that the sequence of pointed solutions subsequentially converges to a limit ancient solution of the harmonic $\S7$-flow $({\bR}^8, \s7_{\infty}(t), g_{\textup{Eucl}}, 0)_{t\in (-\infty, 0]}$ with 
\begin{align}
\label{epsilonregpf1}
    |T_{\s7_{\infty}}(0,0)|=1.    
\end{align}
On the other hand, using the almost monotonicity formula Theorem~\ref{thm:almostmon}, the scale invariance of $\Theta$ combined with estimates on the backwards heat kernel $u$ gives $|T_{\s7_{\infty}}(0,0)|=0$, which is not possible due to \eqref{epsilonregpf1}. 
\end{proof}

An immediate corollary of the $\varepsilon$-regularity theorem is the following result, which states that if the entropy of the initial $\S7$-structure is small then the torsion is controlled at all times. Again, the proof is similar to \cite[Cor. 5.8]{dgk-isometric}.

\begin{corollary}[small initial entropy controls torsion]
\label{cor:enttor}
    Let $\{\s7(t)\}$ be a solution of the harmonic $\S7$-flow \eqref{eq: Har Spin(7) Flow} on compact $(M, g)$,  starting at $\s7_0$. For every $\sigma>0$, there exist $\varepsilon, t_0>0$ and $C< \infty$ such that, if $\s7_0$ induces $g$ and its entropy \eqref{eq:entropyeqn} satisfies
\begin{align*}
    \lambda(\s7_0, \sigma) < \varepsilon,  
\end{align*}
then
\begin{align*}
    \underset{M}{\textup{max}}\ |T_{\s7(t)}|\leq \frac{C}{\sqrt{t}}.    
\end{align*}

\end{corollary}

\subsection{Long-time existence}

In this section, we prove that the harmonic $\S7$-flow on a compact manifold $M$, under smallness assumptions of the initial data, exists for all time and converges to a harmonic $\S7$-structure. We start with the following convexity result for the energy functional $E$ along the flow. 

\begin{lemma}
\label{lemma:conv}
    Along a solution of the harmonic $\S7$-flow \eqref{eq: Har Spin(7) Flow} on compact $(M, g)$,
\begin{align} 
\label{eq:convpf1}
    \frac{d^2}{dt^2}E(\s7 (t))
    &\geq \int_M (\Lambda -3 |T|^2) |\Div T|^2 \vol
    \end{align}
    where $\Lambda$ is the first non-zero eigenvalue of the rough Laplacian of $g$ on two-forms. 
\end{lemma}
\begin{proof}
From \eqref{1vareqn},
\begin{align*}
    \frac{d^2}{dt^2}E(\s7 (t))
    &= - 2 \int_M \langle \frac{\partial}{\partial t}\Div T , \Div T \rangle \vol \\
    &= - 2 \int_M \nabla_m  (\nabla_r T_{r;ip} T_{m;ps} - \nabla_r T_{r;sp} T_{m;pi} + \pi_7 (\nabla_m (\nabla_r T_{r;is}))) \nabla_a  T_{a;is}\vol \\
    &= 2 \int_M 2\nabla_r T_{r;ip} T_{m;ps} \nabla_m  (\nabla_a  T_{a;is}) + |\pi_7 \nabla_m (\nabla_r T_{r;is})|^2 \vol.
\end{align*}
Now $\Div T \in \Omega^2_7$, so that $(\Div T)_{is} \s7_{isab} = -6 (\Div T)_{ab}$, and therefore
\begin{align*}
    &-6 \nabla_m (\Div T)_{ab} = \nabla_m (\Div T)_{is} \s7_{isab} + 2 (\Div T)_{is} T_{m;iq}\s7_{qsab} - 6 (\Div T)_{qb}T_{m;aq} -6 (\Div T)_{aq} T_{m;bq}.
\end{align*}
We use this fact to expand the second term in the integral:
\begin{align*}
   &|\pi_7 \nabla_m (\nabla_r T_{r;is})|^2 = \langle \pi_7 \nabla_m (\nabla_r T_{r;is}),\nabla_m (\nabla_r T_{r;is})\rangle \\
   &= \Big( \tfrac14 (\nabla_m \Div T)_{ij} - \tfrac18 (\nabla_m \Div T)_{ab}\s7_{ijab}\Big) (\nabla_m \Div T)_{ij} \\
   &= \tfrac14 |\nabla_m (\Div T)|^2 - \tfrac18 (\nabla_m \Div T)_{ab}\s7_{ijab} \nabla_m (\Div T)_{ij} \\
   &= \tfrac14 |\nabla_m (\Div T)|^2 - \tfrac18 (\nabla_m \Div T)_{ab} \Big( - 2 (\Div T)_{is} T_{m;iq}\s7_{qsab} + 6 (\Div T)_{qb}T_{m;aq} + 6 (\Div T)_{aq} T_{m;bq} \\
   &\qquad -6 (\nabla_m \Div T)_{ab} \Big) \\
   &=  |\nabla_m (\Div T)|^2 + \tfrac14 (\nabla_m \Div T)_{ab} (\Div T)_{is} T_{m;iq}\s7_{qsab} 
   - \tfrac34 \Big( (\Div T)_{qb}T_{m;aq} + (\Div T)_{aq} T_{m;bq}\Big) (\nabla_m \Div T)_{ab} \\
   &= |\nabla_m (\Div T)|^2 - \tfrac34 \Big( (\Div T)_{qb}T_{m;aq} + (\Div T)_{aq} T_{m;bq}\Big) (\nabla_m \Div T)_{ab} \\
   &\qquad + \tfrac14 (\Div T)_{is} T_{m;iq} \Big( 
   - 2 (\Div T)_{ir} T_{m;ip}\s7_{prqs} + 6 (\Div T)_{ps}T_{m;qp} + 6 (\Div T)_{qp} T_{m;sp}
   -6 (\nabla_m \Div T)_{qs} \Big) \\
   &= |\nabla_m (\Div T)|^2 - \tfrac34 \Big( (\Div T)_{qb}T_{m;aq} + (\Div T)_{aq} T_{m;bq}\Big) (\nabla_m \Div T)_{ab} \\
   &\qquad + \tfrac32 (\Div T)_{is} T_{m;iq} (\Div T)_{ps}T_{m;qp} + \tfrac32 (\Div T)_{is} T_{m;iq} (\Div T)_{qp} T_{m;sp} - \tfrac32 (\Div T)_{is} T_{m;iq} (\nabla_m \Div T)_{qs} \\
   &= |\nabla_m (\Div T)|^2 + \tfrac34 \Big( (\Div T)_{qb}T_{m;aq} - (\Div T)_{aq} T_{m;bq}\Big) (\nabla_m \Div T)_{ab} \\
   &\qquad + \tfrac32 (\Div T)_{is} T_{m;iq} (\Div T)_{ps}T_{m;qp} + \tfrac32 (\Div T)_{is} T_{m;iq} (\Div T)_{qp} T_{m;sp} \\
   &= |\nabla_m (\Div T)|^2 + \tfrac32 (\Div T)_{is} T_{m;iq} \Big( (\Div T)_{ps}T_{m;qp} +  (\Div T)_{qp} T_{m;sp} \Big)\\
   &\geq |\nabla_m (\Div T)|^2 -3 |\Div T|^2 |T|^2. 
\end{align*}
\begin{align*}
    \therefore\qquad
    \frac{d^2}{dt^2}E(\s7 (t)) &\geq \int_M |\nabla_m (\Div T)|^2 -3 |\Div T|^2 |T|^2 \vol  \\
    &\geq \int_M (\Lambda -3 |T|^2) |\Div T|^2 \vol, 
\end{align*}
    where $\Lambda$ is the first non-zero eigenvalue of the rough Laplacian on $2$-forms. 
    On a compact manifold, the kernel of that rough Laplacian consists indeed of parallel $2$-forms, so it is orthogonal to $\Div T$.
\end{proof}

The following general interpolation result  can be proved in a similar way as \cite[Lemma 5.12]{dgk-isometric}.  
\begin{lemma}[interpolation]
\label{lemma-interpolation}
    Let $(M,\s7,g)$ be a compact $\S7$-structure manifold with torsion $T=T_\Phi$. Suppose $|\del T|\leq C$ and non-collapsing at every $x\in M$, i.e., 
\begin{align*}
    \vol(B(x, r))\geq v_0r^8, \qforq 0<r\leq 1,
\end{align*}
    for some constant $v_0(M,g)>0$. Then, for every $\varepsilon>0$, there exists $\delta(\varepsilon, C, v_0)\geq 0$ such that, if 
\begin{align}
\label{eq-interpolation1}
    E(\s7) < \delta,    
\end{align}
    then $|T|<\varepsilon$.
\end{lemma}

We now prove Theorem~\ref{thm:small_tor_ent}. First, we prove that if the torsion of the initial $\S7$-structure
is pointwise small, then \eqref{eq: Har Spin(7) Flow} exists for all time and converges to a critical point of the energy functional. Then Corollary~\ref{cor:enttor}, the derivative estimates from Corollary~\ref{shitypeest},  and the interpolation Lemma~\ref{lemma-interpolation} imply that, if the initial entropy is small, then the torsion does eventually become pointwise small.

\begin{theorem}[small initial torsion gives long-time existence]
\label{thm: smalltorconv}
    Let $(M,\Phi_0,g)$ be a compact  $\S7$-structure manifold. For every $\delta >0$, there exists $\varepsilon (\delta, g) >0$ such that, if $|T_{\Phi_0}| < \varepsilon$, then a harmonic $\S7$-flow \eqref{eq: Har Spin(7) Flow} starting at $\Phi_0$ exists for all time and converges subsequentially smoothly to a $\S7$-structure $\Phi_{\infty}$ such that 
$$
\Div T_{\Phi_{\infty}} =0, \quad
|T_{\Phi_{\infty}}| < \delta.
$$
\end{theorem}

\begin{proof}
We only sketch the proof, since it is very similar to that of the isometric $\G2$-flow \cite[Theorem 5.13]{dgk-isometric}.
By the doubling-time estimate [Proposition \ref{dteforif}], given $\delta>0$, there exists $\varepsilon_0>0$ 
such that, whenever $|T_{\s7_0}|<\varepsilon_0$,
\begin{align}
\label{eq:ltetorpf1}
    t_* := \textup{max}\{ t\geq 0\ :\  |T_{\s7(t)}| \leq 2\varepsilon_0\} 
    > \delta.   
\end{align}
We prove that $t_*=\infty$, hence that a flow with small initial torsion exists for all time. By contradiction, suppose $t_*< \infty$. Then, by the derivative estimates of Corollary~\ref{shitypeest} for the time interval $[t_*-\delta, t_*]$, there is a constant $c_0=c_0(g, \varepsilon_0)$ such that 
\begin{align}
\label{eq:ltetorpf2}
    |\del T_{\s7(t_*)}|< c_0.  
\end{align}
Note that \eqref{eq: Har Spin(7) Flow} is the negative gradient flow of $E(\s7(t))$ and hence $E(\s7(t_*))\leq E(\s7(0))=E(\s7_0)$. Given $a>0$, the interpolation Lemma \ref{lemma-interpolation} yields a constant $\gamma_{a}=\gamma_{a}(c_0)>0$ such that, if $E(\s7_0)< \gamma_a$, then $|T_{\s7(t_*)}|< a$. 
Now, if $\varepsilon < \textup{min} \{\varepsilon_0, \gamma_{2\varepsilon_0} \}$, then $|T_{\s7(t_*)}|<2\varepsilon$, which contradicts the maximality of $t_*$, as defined in  \eqref{eq:ltetorpf1}. Thus $t_*=\infty$ and the flow exists for all time.

Denote by $\Lambda$ the first non-zero eigenvalue of the Laplacian acting on $2$-forms. Note from \eqref{eq:convpf1} that, if $|T|^2\leq \frac{\Lambda}{6}$, then 
\begin{align}\label{eq:ltetorpf3}
\frac{d}{dt} \int_M |\Div T_{\s7(t)}|^2 \vol = -\frac{d^2}{dt^2} E(\s7(t))\leq -\frac{\Lambda}{2}\int_M |\Div T_{\s7(t)}|^2 \vol.     
\end{align}
Let $a_0= \sqrt{\frac{\Lambda}{6}}$. If we take $\varepsilon < \textup{min}\{\varepsilon, \gamma_{2\varepsilon_0}, \gamma_{a_0}\}$, then from the previous paragraph we obtain existence of the flow for all time and satisfying \eqref{eq:ltetorpf3}. Thus, we get the  decay estimate
\begin{align}
\label{eq:ltetorpf4}
    \int_M |\Div T_{\s7(t)}|^2 \vol \leq e^{-\frac{\Lambda t}{2}}\int_M |\Div T_{\s7(0)}|^2 \vol,     
    \quad\forall \ t\geq 0.
\end{align}
Now, for every $s_1<s_2$, we have
\begin{equation} \label{eq:ltetorpf5}
\begin{aligned}
    \int_M |\s7(s_2)-\s7(s_1)| \vol 
    &\leq  \int_M \int_{s_1}^{s_2} \left| \partial_t \s7(s)  \right| ds \vol
    = \int_{s_1}^{s_2} \int_M |\Div T_{\s7(s)}| \vol ds\\
    &\leq c\int_{s_1}^{s_2} \left(\int_M |\Div T_{\s7(s)}|^2 \vol \right)^{\frac{1}{2}} ds\\
    &\leq c \int_{s_1}^{s_2}  e^{-\frac{\Lambda s}{4}} ds. 
\end{aligned}
\end{equation}
Hence $\s7(t)$ converges exponentially fast to a unique limit $\s7_{\infty}$ in $L^1$ as $t\rightarrow+\infty$. 

Moreover, the uniform torsion bound gives estimates on all derivatives of the torsion. Given any sequence $t_n \rightarrow+\infty$, a subsequence of $\{\s7(t_n)\}$ will converge smoothly to a limit, which must be $\s7_\infty$ by uniqueness. Therefore, the flow $\{\s7(t)\}$ converges smoothly to $\s7_\infty$ as $t\rightarrow +\infty$. Inequality \eqref{eq:ltetorpf4} implies that $\Div T_{\s7_\infty}= 0$, and choosing $\varepsilon>0$ small enough, we can also achieve  $|T_{\s7_\infty} |<\delta$, using the interpolation Lemma \ref{lemma-interpolation}.
\end{proof}

We finally prove that small initial entropy implies long-time existence and convergence.

\begin{theorem}[small entropy]
\label{thm:smallent}
    On a compact  $\S7$-structure manifold $(M,\Phi_0,g)$, there exist constants $C_k(M,g) < +\infty$, such that the following holds. For each $\varepsilon>0$ and $\sigma >0$, there exists $\lambda_{\varepsilon} (g, \sigma) >0$ such that, if the entropy \eqref{eq:entropyeqn} satisfies 
\begin{align}
\label{smallenteq}
    \lambda(\s7_0, \sigma) < \lambda_{\varepsilon},
\end{align}
    then the torsion becomes eventually pointwise small along the harmonic $\S7$-flow \eqref{eq: Har Spin(7) Flow} starting at $\Phi_0$. Therefore the flow exists for all time and subsequentially converges to a $\S7$-structure $\Phi_{\infty}$ such that 
    $$
    \Div T_{\Phi_{\infty}} =0, \quad
    |T_{\Phi_{\infty}}| < \varepsilon \qandq
    |\nabla^k T_{\Phi_{\infty}}| < C_k,
    \,\forall\  k\geq 1.
    $$
\end{theorem}
\begin{proof}
    By Corollary~\ref{cor:enttor}, if $\lambda_{\varepsilon}>0$ is small enough, then the torsion along the flow satisfies $|T_{\s7(t)}|\leq \frac{C}{\sqrt{t}}$, for all $t\in (0, \tau]$. Thus, by the derivative estimates in Corollary~\ref{shitypeest}, $\s7(\tau)$ satisfies $|\del T_{\s7(\tau)}|\leq C'$ for some constant $C'<\infty$. Hence by the interpolation Lemma~\ref{lemma-interpolation}, if $\lambda_{\varepsilon} >0$ is even smaller then $|T_{\s7(\tau)}|< \varepsilon$. By Theorem~\ref{thm: smalltorconv}, the flow then converges smoothly to a $\S7$-structure $\s7_{\infty}$ with divergence-free torsion and derivative bounds.
\end{proof}

\subsection{Singularity structure}
\label{sec:singularity}

We investigate the singularities of the harmonic $\S7$-flow. Let $\{\s7(t)\}_{t\in [0, \tau)}$ be a solution to \eqref{eq: Har Spin(7) Flow} on a compact manifold $(M,g)$, with a finite time singularity at $\tau< + \infty$. 

\subsubsection{The singular set $S$}

Let $\varepsilon$ and $\bar{\rho}$ be the quantities from the $\varepsilon$-regularity Theorem~\ref{thm:epsilonreg}. We define the \emph{singular set} of the flow by 
\begin{align}
\label{eq:singsetdefn}
    S = \{ x\in M \ :\ \Theta_{(x, \tau)}(\s7(\tau-\rho^2)) \geq \varepsilon, \ \textup{for\ all}\ \rho\in (0, \bar{\rho}]\}.    
\end{align}
The following lemma explains why $S$ is called the singular set of the flow.

\begin{lemma}
\label{lem:sing}
    The harmonic $\S7$-flow $\{\s7(t)\}_{t\in [0, \tau)}$ restricted to $M\setminus S$ converges as $t\rightarrow \tau$, smoothly and uniformly away from $S$, to a smooth harmonic $\S7$-structure $\s7(\tau)$ on $M\setminus S$. Moreover, for every $x\in S$, there is a sequence $(x_i, t_i)\to(x,\tau)$ such that 
\begin{align*}
    \lim_{i} |T_{\s7}(x_i, t_i)|= \infty.    
\end{align*}
Thus, $S$ is indeed the singular set of the flow.
\end{lemma}

\begin{proof}
By Theorem~\ref{thm:epsilonreg}, for every $x\in M\setminus S$, there exist $r_x>0$ and $C_x<\infty$ such that, setting
$$
\Lambda_{r_x}(y,t) = \textup{min}\ \left(1-\frac{d_g(x, y)}{r_x}, \sqrt{1-\frac{\tau-t}{{r_x}^2}}\right),
$$
one has
$$
    \Lambda_{r_x}(y,t)|T_{\s7}(y,t)|\leq \frac{C_{x}}{r_x},
    \quad\forall\ 
    (y,t)\in B(x, r_x)\times [\tau-{r_x}^2, \tau].
$$
So, for $\hat{r}_x=\frac{1}{2}r_x$, we have
$$
T_{\s7}(y,t)\leq \widehat{C}_x,    
\quad\forall\ 
(y,t)\in B(x, \hat{r}_x)\times [\tau-{\hat{r}_x}^2, \tau].
$$
Hence, by the local derivative estimates of Corollary~\ref{localderest}, there exist constants $C_{x,j}$ such that 
$$
|\del^jT_{\s7}(y,t)|\leq C_{x,j},
\quad\forall\  
(y,t)\in B(x, \frac{\hat{r}_x}{2})\times [\tau-\frac{{\hat{r}_x}^2}{4}, \tau],
\quad\forall 
\ j\geq 1.
$$
Thus as $t\rightarrow \tau$, the flow $\s7(t)$ converges smoothly and uniformly away from $S$ to a harmonic $\S7$-structure $\s7(\tau)$. 
\end{proof}

\subsubsection{Proof of Theorem \ref{thm: singsize}}
\label{sec: singsize}

Note from \eqref{eq:singsetdefn} that $S$ is a closed set. We therefore estimate $\mathcal{H}^6(S)$, where $\mathcal{H}^6$ denotes the $6$-dimensional Hausdorff measure on $(M, g)$.

\medskip

Let $S'$ be any subset of $S$. As in \cite{grayson-hamilton}, there is $S''\subset S'$ such that
\begin{align}
\label{eq: singpf1}
    \mathcal{H}^6(S'')
    \geq \frac 12 \mathcal{H}^{6}(S')    
\end{align}
and 
\begin{align*}
    u_{S''}(y,s)
    = \int_{S''} u_{(x, \tau)}(y,s)\vol(x)    
\end{align*}
satisfies
\begin{align}
\label{eq: singpf2}
    u_{S''}(y,s) \leq \frac{C}{\tau-s},    
\end{align}
for some constant $C$ depending only on $g$. Thus, for every $\rho \in (0, \bar{\rho})$, using \eqref{eq:singsetdefn} with the same $\varepsilon$ and $S''\subset S$, we get
\begin{align*}
 \varepsilon \mathcal{H}^6(S'') = \int_{S''}\varepsilon d\mathcal{H}^6(x)\leq \int_{S''} \Theta_{(x, \tau)}(\s7(\tau-\rho^2))d\mathcal{H}^6(x).    
\end{align*}
Using the definition of $\Theta$, the estimate \eqref{eq: singpf2} and the hypothesis \eqref{eq:singsize}, we estimate
\begin{align*}
\varepsilon \mathcal{H}^6(S'') &\leq \int_{S''}\int_M \rho^2 |T(y, \tau-\rho^2)|^2u_{(x,\tau)}(y, \tau-\rho^2)\vol(y)d\mathcal{H}^6(x)\\
& \leq \int_M \rho^2 |T(y, \tau-\rho^2)|^2u_{S''}(y, \tau-\rho^2)\vol(y)\\
&\leq C \int_M |T(y, \tau-\rho^2)|^2\vol(y) \\
& \leq CE_0.
\end{align*}
Since $S'$ was arbitrary, the result follows.

\begin{remark}
Theorem \ref{thm: singsize} says that the singular set is \emph{at most} $6$-dimensional. In order to find a geometric interpretation of $S$ in terms of $\S7$-geometry, it is likely that $S$ would be at most $4$-dimensional, as there are no distinguished $6$-dimensional subspaces in this context.
\end{remark}

\subsubsection{Soliton model at Type-\rom{1} singularities}
\label{sec: type_I sing}

Let $\{\s7(t)\}_{t\in [0, \tau]}$ be a solution to \eqref{eq: Har Spin(7) Flow}  that exists for some maximal time $\tau$. From the evolution of the norm of the torsion tensor $T$ along \eqref{eq: Har Spin(7) Flow} in Proposition~\ref{prop:evolnormT}, and the maximum principle, we see that the rate of blow-up of the norm of the torsion tensor must be at least
\begin{align*}
\underset{x\in M}{\textup{sup}} |T(t)| \geq \frac{1}{\sqrt{C(\tau-t)}}.    
\end{align*}

\begin{definition}
    The flow $\{\s7(t)\}_{t\in [0,\tau]}$ is said to have a \emph{Type-\rom{1} singularity at $\tau$} if 
\begin{align*}
 \underset{x\in M}{\textup{sup}} |T(t)| \leq \frac{1}{\sqrt{C(\tau-t)}}.    
\end{align*}
\end{definition}

Using the almost monotonicity formula from Theorem \ref{thm:almostmon} as in \cite{grayson-hamilton} and \cite[Theorem 5.20]{dgk-isometric}, we see that a sequence of blow-ups at a Type-\rom{1} singularity admits a subsequence that converges to a shrinking soliton (cf. \textsection\ref{sec: Spin(7)-solitons}) of the flow:
\begin{theorem}
\label{thm:Type1}
    Let $\{\s7(t)\}_{t\in [0, \tau)}$ be the maximal smooth harmonic $\S7$-flow on $(M^8, \s7_0)$ starting at $\s7_0$. Suppose that the flow encounters a Type-\rom{1} singularity at  $\tau<\infty$. Let $x\in M$ and $\lambda_i \searrow 0$, and consider the rescaled sequence $\s7_i(t)= {\lambda_i}^{-3}\s7(\tau-{\lambda_i}^2t)$. Then, after possibly passing to a subsequence, $(M, \s7_i(t), x)$ converges smoothly to an ancient harmonic $\S7$-flow $\{\s7_{\infty}(t)\}_{t<0}$ on $({\bR}^8, g_{\textup{Eucl}})$ induced by a shrinking soliton, i.e.
\begin{align*}
\Div T_{\s7_{\infty}}(x,t)= -\frac{x}{2t}\lrcorner T_{\s7_{\infty}}.  
\end{align*}
    Moreover $x\in M\setminus S$ if, and only if, $\{\s7_{\infty}(t)\}$ is the stationary flow induced by a torsion-free $\S7$-structure on $({\bR}^8, g_{\textup{Eucl}})$. 
\end{theorem}

\begin{remark}
    It would be very interesting to get explicit examples of shrinking, steady or expanding harmonic solitons, even for $\bR^8$. On the contrary, if one could prove that there do not exist any shrinking solitons on $\bR^8$, then the previous theorem tells us that the harmonic $\S7$-flow has no Type-\rom{1} singularities.
\end{remark}

\newpage
\printbibliography

\noindent
(SD): Institut f\"ur Mathematik, Humboldt-Universit\"at zu Berlin, Rudower Chaussee 25, 12489 Berlin.\\ \href{mailto:dwivedis@math.hu-berlin.de}{dwivedis@math.hu-berlin.de}

\medskip

\noindent
(EL): Univ. Brest, CNRS UMR 6205, LMBA, F-29238 Brest, France.\\ \href{loubeau@univ-brest.fr}{loubeau@univ-brest.fr}

\medskip

\noindent
(HSE): Institute of Mathematics, Statistics and Scientific Computing (IMECC), University of Campinas (Unicamp), 13083-859 Campinas-SP, Brazil.\\ \href{henrique.saearp@ime.unicamp.br}{henrique.saearp@ime.unicamp.br}

\end{document}